\documentclass[11pt]{amsart}
\usepackage[babel]{csquotes}
\usepackage{enumitem}
\usepackage{amsmath,amsthm,amssymb,mathrsfs,amsfonts,verbatim,enumitem,color,leftidx,mathtools}
\usepackage{mathabx}
\usepackage{etoolbox} 
\usepackage{bbm}
\usepackage[all,tips]{xy}
\usepackage{graphicx,ifpdf}
\usepackage{stmaryrd}
\usepackage{amssymb}
\usepackage{dsfont}
\usepackage{mathdots}

\usepackage{color}
\definecolor{red}{rgb}{1,0,0}

\definecolor{gre}{rgb}{0,0.7,0}

\definecolor{blu}{rgb}{0,0,1}

\ifpdf
\fi
\usepackage[colorlinks]{hyperref}
\hypersetup{
linkcolor=blue,          
citecolor=green,        
}

\usepackage{tikz}

\newtheorem{thm}{Theorem}[section]
\newtheorem{lem}[thm]{Lemma}

\newtheorem{prop}[thm]{Proposition}

\newtheorem{ques}[thm]{Question}

\theoremstyle{definition}
\newtheorem{defi}[thm]{Definition}

\theoremstyle{remark}

\numberwithin{equation}{section}
\definecolor{esperance}{rgb}{0.0,0.5,0.0}

\newcommand{\bd}{\mathbf{d}}
\newcommand{\be}{\mathbf{e}}

\newcommand{\bm}{\mathbf{m}}

\newcommand{\bp}{\mathbf{p}}

\newcommand{\bxi}{\boldsymbol{\xi}}
\newcommand{\btau}{\boldsymbol{\tau}}
\newcommand{\bsig}{\boldsymbol{\sig}}
\newcommand{\btheta}{\boldsymbol{\theta}}

\newcommand{\sh}{\mathsf{h}}
\newcommand{\sB}{\mathsf{B}}

\newcommand{\del}{\delta}

\newcommand{\eps}{\epsilon}

\newcommand{\sig}{\sigma}

\newcommand{\Om}{\Omega}

\newcommand{\cD}{\mathcal{D}}
\newcommand{\cE}{\mathcal{E}}

\newcommand{\cK}{\mathcal{K}}
\newcommand{\cL}{\mathcal{L}}

\newcommand{\cP}{\mathcal{P}}
\newcommand{\cQ}{\mathcal{Q}}
\newcommand{\cR}{\mathcal{R}}

\newcommand{\cT}{\mathcal{T}}

\newcommand{\cW}{\mathcal{W}}

\newcommand{\cY}{\mathcal{Y}}
\newcommand{\cZ}{\mathcal{Z}}

\newcommand{\bR}{\mathbb{R}}
\newcommand{\bZ}{\mathbb{Z}}
\newcommand{\bQ}{\mathbb{Q}}

\newcommand{\bN}{\mathbb{N}}

\newcommand{\bT}{\mathbb{T}}

\newcommand{\Id}{\operatorname{Id}}

\DeclareMathAlphabet{\mathpzc}{OT1}{pzc}{m}{it}

\newcommand\set[1]{\left\{#1\right\}}

\newcommand\on[1]{\operatorname{#1}}
\newcommand\diag[1]{\operatorname{diag}\left(#1\right)}

\newcommand\tb[1]{\textbf{#1}}

\newcommand{\wstar}{\overset{\on{w}^*}{\lra}}
\newcommand{\Supp}{\on{Supp}}

\newcommand{\lra}{\longrightarrow}

\newcommand{\onto}{\xymatrix{\ar@{>>}[r]&}}

  \newcounter{constantC} 
  \newcommand{\newconC}[1]{\refstepcounter{constantC}\label{#1}} 
  \newcommand{\useconC}[1]{C_{\ref{#1}}}
  
  \newcounter{constantE} 
  \newcommand{\newconE}[1]{\refstepcounter{constantE}\label{#1}} 
  \newcommand{\useconE}[1]{E_{\ref{#1},\lambda}}

\newcommand{\eq}[1]
{
\begin{equation*}
{#1}
\end{equation*}
}
\newcommand{\eqlabel}[2]
{
\begin{equation}
{#2}\label{#1}
\end{equation}
}

\makeatletter
\newcommand*{\rom}[1]{\expandafter\@slowromancap\romannumeral #1@}
\makeatother

\begin{document}

\title[Higher rank Divergent on average]{On divergent on average trajectories for higher rank actions}
\author{Wooyeon Kim}

\begin{abstract}
For $d\ge 3$ we first show that the Hausdorff dimension of the set of $A$-divergent on average points in the $(d-1)$-dimensional closed horosphere in the space of $d$-dimensional Euclidean lattices, where $A$ is the group of positive diagonal matrices, is at most $\frac{d-1}{2}$. In particular, this upper bound is sharp for $d=3$. 

We apply this to compute the Hausdorff dimension of the set of exceptions to the inhomogeneous uniform version of Littlewood conjecture. We say that a pair $(\xi_1,\xi_2)\in\bR^2$ satisfies the inhomogeneous Littlewood conjecture if
$$\liminf_{q\to\infty}q\|q\xi_1-\theta_1\|_\bZ\|q\xi_2-\theta_2\|_\bZ=0$$
for all $(\theta_1,\theta_2)\in\bR^2$, where $\|\cdot\|_\bZ$ denotes the distance to the nearest integer. We prove that the Hausdorff dimension of the set of pairs $(\xi_1,\xi_2)\in\bR^2$ not satisfying the inhomogeneous Littlewood conjecture is $1$, which is equal to the Hausdorff dimension of the conjectural set of exceptions.
\end{abstract}

\maketitle
\section{Introduction}

The study of orbits in the space of lattices, identified with $\operatorname{SL}_d(\mathbb{R})/\operatorname{SL}_d(\mathbb{Z})$, offers a rich interplay between geometry, topology, and arithmetic. As this homogeneous space is non-compact, understanding the behavior of orbits near the cusp is particularly significant, both as an intrinsic interest and due to its profound connections to number theory.

For $d\geq 2$ and a one-parameter unipotent subgroup $\set{u_t}_{t\in\mathbb{R}}$ of $\operatorname{SL}_d(\mathbb{R})$, Margulis proved that the trajectory $u_tx$ is non-divergent for any $x\in\operatorname{SL}_d(\mathbb{R})/\operatorname{SL}_d(\mathbb{Z})$ \cite{Mrg71}. This result was used as an ingredient in his proof of the arithmeticity of nonuniform lattices in semisimple Lie groups of higher rank \cite{Mrg75}. Building on this, Dani strengthened Margulis's non-divergence theorem by showing that the mass of the trajectory $u_tx$ does not escape to infinity \cite{Dan79,Dan86}. More precisely, he showed that for any $\eta>0$ there exists a compact set $K\subset \operatorname{SL}_d(\mathbb{R})/\operatorname{SL}_d(\mathbb{Z})$ such that
\eqlabel{eq:Daninondivergence}{m_{\mathbb{R}}\big(\{t\in[0,T]:\; u_tx\notin K \}\big)<\eta T}
for sufficiently large $T$ and $x\in K$, where $m_{\mathbb{R}}$ stands for the Lebesgue measure on $\mathbb{R}$. This result laid the groundwork for Ratner’s celebrated theorem, which provides a complete classification of invariant measures and orbit closures for unipotent flows.

In contrast, the dynamics of diagonalizable subgroup actions can exhibit much wilder behavior. For instance, there are divergent orbits and those whose mass escapes to infinity (divergent on average) for diagonalizable subgroups. This naturally leads to the question: What is the Hausdorff dimension of such escaping orbits? The Hausdorff dimension of the set of divergent (or divergent on-average) points for one-dimensional diagonalizable subgroups has been extensively studied by many authors, as will be detailed in the next subsection.

Although one-dimensional diagonalizable subgroups do not exhibit the measure rigidity phenomena exemplified by Ratner's theorem, Furstenberg, Margulis, and others conjectured that higher-rank diagonalizable actions have rigidity properties analogous to those of unipotent flows. Over the last two decades, there has been significant progress towards the conjectural rigidity properties of higher-rank actions. In particular, Einsiedler, Katok, and Lindenstrauss characterized invariant and ergodic measures for higher-rank actions under the assumption of positive entropy \cite{EKL06,EL15,EL18,EL23}. These rigidity results have led to striking consequences in number theory, including a significant advance towards the Littlewood conjecture in \cite{EKL06}. 

However, unlike unipotent dynamics, the link between individual orbits and invariant measures for higher-rank diagonalizable actions is weaker, primarily due to the absence of a non-divergence theorem analogous to that in unipotent dynamics. Consequently, characterizing divergent on-average orbits becomes essential for understanding individual orbits in higher-rank actions. The goal of this paper is to estimate the Hausdorff dimension of the set of divergent on-average points for higher-rank actions. As an application, we use this result to determine the Hausdorff dimension of the set of exceptions to the inhomogeneous uniform version of the Littlewood conjecture.

\subsection{one-dimensional diagonal flows}
We begin by reviewing previous results on the Hausdorff dimension of divergent trajectories (or those that are divergent on average) for one-dimensional diagonal flows. For $d\ge 3$, let $G=\operatorname{SL}_d(\bR)$, $\Gamma=\operatorname{SL}_d(\bZ)$, and $X=G/\Gamma$. We consider a one-parameter diagonal subgroup $\{g_t\}< G$ given by $g_t:=\diag{e^{t},\ldots,e^{t},e^{-(d-1)t}}$.

\begin{defi}\label{def:1dim}
    Let $x\in X$.
    \begin{enumerate}
        \item We say that $x$ is $g_t$-divergent if for any compact set $K\subset X$ there exists $T(K)>0$ such that $g_tx\notin K$ for all $t\geq T(K)$. 
        \item We say that $x$ is $g_t$-divergent on average if
        $$\lim_{N\to\infty}\frac{1}{N}\#\set{t\in\set{1,\ldots,N}: g_tx\in K}=0$$
        for any compact set $K\subset X$. 
    \end{enumerate}
\end{defi}

Let $U$ be the (unstable) horospherical subgroup for $g_t$, i.e.
\eq{\begin{aligned}
    U&=\set{g\in G: g_tgg_{-t}\rightarrow e \quad\textrm{as}\quad t\to -\infty}\\&=\set{u(\bxi):=\left(\begin{matrix} \Id_{d-1}& \bxi\\
0&  1
\end{matrix}\right)\in G: \bxi\in\bR^{d-1}}.
\end{aligned}}
It is well known from the Dani correspondence that Diophantine properties of a vector $\bxi\in\bR^{d-1}$ are characterized by dynamical properties of the $g_t$-orbit of the point $u(\bxi)\Gamma$ on the submanifold $U\Gamma$ of $X$ \cite{Dan85, KM98}. In particular, $u(\bxi)\Gamma$ is $g_t$-divergent if and only if $\bxi$ is \emph{singular}, i.e. for every $\eps>0$ there exists $T_0$ such that for all $T>T_0$ the system of inequalities
$$\|q\bxi\|_\bZ<\eps T^{-\frac{1}{d}} \quad \textrm{and}\quad 0<q<T$$
admits an integral solution $q\in\bN$ \cite[Theorem 2.14]{Dan85}. Similarly, $u(\bxi)\Gamma$ is $g_t$-divergent on average if and only if $\bxi$ is \emph{singular on average}, i.e. for every $\eps>0$ 
$$\lim_{N\to\infty}\frac{1}{N}\left| \set{l\in\set{1,\ldots,N}:\textrm{there exists } q\in\bN \ \text{s.t.} \ 
\|q\bxi\|_{\bZ}<\eps 2^{-l} \ \textrm{and} \ 0<q<2^l}\right| =1.$$

It is natural to ask about the set of singular (or singular on average) vectors, or equivalently the set of $g_t$-divergent (or divergent on average) points in the space of lattices. 
In the last decade the Hausdorff dimension of these sets has been calculated by several authors: \cite{Che11} proved that
\eq{\begin{aligned}
    \dim_H\set{\bxi\in\bR^{d-1}: \bxi \textrm{ is singular}}&=\dim_H\set{x\in U\Gamma: x \textrm{ is } g_t\textrm{-divergent}}\\&=d-1-\frac{d-1}{d}
\end{aligned}}
for $d=3$, and \cite{CC16} extended this result for all $d\ge 3$. For the set of singular on average vectors, \cite{EK12} showed that
\eq{\begin{aligned}
    \dim_H\set{\bxi\in\bR^{d-1}: \bxi \textrm{ is singular on average}}&=\dim_H\set{x\in U\Gamma: x \textrm{ is } g_t\textrm{-divergent on average}}\\&=d-1-\frac{d-1}{d}
\end{aligned}}
for $d=3$, and \cite{KKLM17} extended this for all $d\ge 3$. More generally, \cite{KKLM17} considered the set of singular (or singular on average) linear forms and gave an upper bound for the Hausdorff dimension, and \cite{DFSU19} later proved that this upper bound is indeed sharp. See also \cite{LSST19,Sol21,KP22} for the results for weighted $1$-dimensional diagonal flows on the space of lattices, and \cite{GS20,AGMS21} for the results on more general homogeneous spaces.

\subsection{Higher rank diagonal actions}
We now turn to higher-rank diagonal actions. We denote
$$a_{\btau}:=\diag{e^{\tau_1},\ldots,e^{\tau_{d-1}},e^{-(\tau_1+\cdots+\tau_{d-1})}}$$
for $\btau=(\tau_1,\ldots,\tau_{d-1})\in\bR^{d-1}$.
Let $A<G$ denote the group of $d\times d$ positive diagonal matrices, i.e. $A=\set{a_{\btau}:\btau\in\bR^{d-1}}$. We also denote $\bR^{d-1}_{>0}$ the set of vectors with positive entries, and let $A^+$ be the subsemigroup $\set{a_{\btau}:\btau\in\bR_{>0}^{d-1}}$ of $A$. Note that $A^+$ consists of the diagonal elements in $A$ expanding the horospherical subgroup $U$. As we defined for $1$-dimensional diagonal subgroups, let us define the notion of divergent (divergent on average) trajectories for higher-rank diagonal actions.

\begin{defi}\label{def:high}
    Let $x\in X$.
    \begin{enumerate}
        \item We say that $x$ is $A^+$-divergent if for any compact set $K\subset X$ there exists $T(K)>0$ such that $a_{\btau} x\notin K$ for all $\btau$ with $\displaystyle\min_{1\leq i\leq d-1}\tau_i\geq T(K)$. 
        \item We say that $x$ is $A^+$-divergent on average if
        $$\displaystyle\lim_{N\to\infty}\frac{1}{N^{d-1}}\#\set{\btau\in\set{1,\ldots,N}^{d-1}:a_{\btau}x\in K}=0$$
        for any compact set $K\subset X$. 
    \end{enumerate}
\end{defi}


The set of divergent trajectories for higher-rank actions can be characterized algebraically. Margulis showed that $u(\bxi)\Gamma$ is $A^+$-divergent if and only if $\bxi\in\bQ^{d-1}$. In general homogeneous spaces of reductive real algebraic groups,  \cite{TW03,ST23} showed that every divergent trajectory of a real torus diverges due to a purely algebraic reason if the dimension of the torus is at least the $\bQ$-rank of the lattice subgroup.

In contrast to the set of divergent trajectories, not much has been known about the set of divergent on average trajectories for higher-rank actions. The first main result of this paper is an upper bound for the Hausdorff dimension of the set of $A^+$-divergent on average points on the closed horosphere $U\Gamma$.

\begin{thm}\label{dimupperbdd}
We have
\eq{\dim_H\set{x\in U\Gamma: x \textrm{ is }A^+\textrm{-divergent on average}}\leq \frac{d-1}{2}.}
\end{thm}

We remark that the upper bound in Theorem \ref{dimupperbdd} is indeed sharp for $d=3$. Let us denote by $\cR_d$ the set of $(\xi_1,\ldots,\xi_{d-1})\in[-\frac{1}{2},\frac{1}{2}]^{d-1}$ such that \eq{\dim \operatorname{span}_\bQ(1,\xi_1,\ldots,\xi_{d-1})\leq2.}
Since $\cR_d$ is a countable union of rational lines, we have $\dim_H \cR_d=1$ for all $d\ge 3$.

\begin{thm}\label{dimlowbdd}
For any $\bxi\in\cR_d$, $u(\bxi)\Gamma$ is $A^+$-divergent on average.
\end{thm}

Combining Theorem \ref{dimupperbdd} and Theorem \ref{dimlowbdd} we have
\eqlabel{dim1}{1\leq \dim_H\set{x\in U\Gamma: x\textrm{ is }A^+\textrm{-divergent on average}}\leq \frac{d-1}{2},}
hence for $d=3$ the Hausdorff dimension is exactly $1$.

\subsection{Inhomogeneous uniform version of Littlewood conjecture}
Theorem \ref{dimupperbdd} has number-theoretic consequences, particularly in Diophantine approximation. Around 1930, Littlewood conjectured that
\eqlabel{eq:Littlewood}{\liminf_{q\to\infty}q\|q\xi_1\|_\bZ\|q\xi_2\|_\bZ=0}
for any $\xi_1,\xi_2\in\bR$. Although Littlewood conjecture remains open, \cite{EKL06} established that the set of exceptions to Littlewood conjecture has Hausdorff dimension zero. The key ingredient to this result was their classification theorem of ergodic $A$-invariant measures with positive entropy.

The Littlewood conjecture can be extended in a natural way to inhomogeneous Diophantine approximation. Shapira \cite{Sha11} proved that almost every pair of real numbers $(\xi_1,\xi_2)$ satisfies the following inhomogeneous uniform version of Littlewood conjecture:
\begin{equation}\label{eq1}\liminf_{q\to\infty}q\|q\xi_1-\theta_1\|_\bZ\|q\xi_2-\theta_2\|_\bZ=0\;\textrm{ for all }\;\theta_1,\theta_2\in\mathbb{R},\end{equation}
answering a question of Cassels \cite[p.307]{Cas72}. Furthermore, \cite{Sha11} also proved that \eqref{eq1} holds if $1,\xi_1,\xi_2$ forms a basis of a real cubic field, and does not hold if $(\xi_1,\xi_2)\in\cR_3$, i.e. $1,\xi_1,\xi_2$ are linearly dependent over $\bQ$. In \cite{GV17}, Gorodnik and Vishe obtained a quantitative refinement of \eqref{eq1} by a factor of $(\log \log \log \log \log q)^{\lambda}$ for almost every pair, for some constant $\lambda>0$. Recently, in \cite{CT23} Chow and Technau established an inhomogeneous analog of Gallagher's theorem, i.e. a quantitative refinement of \eqref{eq1} by a factor of $(\log q)^{2}$ for almost every pair.

Motivated by the results in \cite{Sha11}, Bugeaud asked the following question:

\begin{ques}\cite[Problem 24]{Bug14}\label{Ques:InhomLittlewood}
    Is it true that, we have \eqref{eq1} for all $(\xi_1,\xi_2)\notin \cR_3$?
\end{ques}

Theorem~\ref{dimupperbdd}, as well as its refinement which will be described in \S5, has the following implication toward Question~\ref{Ques:InhomLittlewood}, in the sense of the Hausdorff dimension of the set of exceptions.
\begin{thm}\label{InhomLittlewood}
For $d\ge 3$, let $\Xi$ be the set of $(\xi_1,\ldots,\xi_{d-1})\in\bR^{d-1}$ such that
\begin{equation}\label{eq1'}\liminf_{q\to\infty}q\prod_{i=1}^{d-1}\|q\xi_i-\theta_i\|_\bZ>0\;\textrm{ for some }\;\theta_1,\ldots,\theta_{d-1}\in\mathbb{R}.\end{equation}Then we have $1\leq \dim_H \Xi\leq \frac{d-1}{2}$.\end{thm}
In particular, for $d=3$ the Hausdorff dimension of the set of exceptions to the inhomogeneous uniform version of Littlewood conjecture is $1$, which is equal to the Hausdorff dimension of the conjectural set of exceptions $\cR_3$.



\subsection{Discussion of the proofs}
Throughout this paper, we denote by $\set{\be_1,\ldots,\be_{d-1}}$ the standard basis of $\bR^{d-1}$. For $1\leq i\leq d-1$ we let
$$A_i:=\set{a_{t\be_i}\in G: t\in\bR},\qquad\qquad U_i:=\set{u_i(s):=u(s\be_i)\in G:s\in\bR}.$$
We also let $U_i^{\perp}$ denote the orthogonal complement of $U_i$ in $U$, and let $H_i$ denote the $\operatorname{SL}_2(\bR)$-copy in $G$ containing $A_i$ and $U_i$.

The proofs of Theorem~\ref{dimupperbdd} and Theorem~\ref{InhomLittlewood} rely heavily on various height functions and their associated contraction properties. In \S2, we construct height functions for one-dimensional diagonal subgroups and establish their contraction hypotheses. Let $0<\lambda<1$ be fixed, sufficiently close to 1, and $t$ be sufficiently large depending on $\lambda$. For each $1\leq i\leq d-1$, the Benoist-Quint height function $\alpha_i:X\to[1,\infty]$ satisfies
\begin{enumerate}
    \item (Subharmonic estimate) For any $x\in X$, $$\int_{-\frac{1}{2}}^{\frac{1}{2}}\alpha_i(a_{t\be_i}u_i(s)x)ds\leq e^{-\lambda t}\alpha_i(x)+C,$$
    \item (Log-Lipschitz property) For any $g\in B^{H_i}(\operatorname{id},1)$ and $x\in X$, $$C^{-1}\alpha_i(x)\leq \alpha_i(gx)\leq C\alpha_i(x),$$
\end{enumerate}
where $C>1$ is a constant depending only on $\lambda$ and $t$ \cite{BQ12,Shi20,Kha20}. Notably, the log-Lipschitz property for $\alpha_i$ is restricted to perturbation under the subgroup $H_i$, which is insufficient for our purposes. To address this limitation, we define modified height functions $\widetilde{\alpha}_i:X'\to[1,\infty]$ for $1\leq i\leq d-1$, where $X'=A^+U\Gamma$. While these functions are defined only on a subset $X'$ of $X$, their log-Lipschitz property extends beyond $H_i$ to the subgroup $U$ as well. 

We then construct \textit{dynamical height functions} $\beta_{N,\del,i,\mathsf{h}}:X'\to[1,\infty]$ for $N\in\bN$, $0<\del<1$, $1\leq i\leq d-1$, and $\mathsf{h}\geq 1$. These functions $\beta_{N,\del,i,\mathsf{h}}$ capture the averaged behavior of $A_i$-trajectories near the cusp. Specifically, $\beta_{N,\del,i,\mathsf{h}}(x)$ is positive only if the trajectory $\set{a_{t\be_i}x,\ldots,a_{Nt\be_i}x}$ remains close to the cusp most of time, as measured by $\del$ (representing the proportion of mass escaping to infinity as $1-\del$) and a cut-off height $\mathsf{h}$ defined in terms of $\widetilde{\alpha}_i$. In particular, if $x$ is $A_i$-divergent on average, then \eqlabel{eq:divergentonaverageonedim}{x\in \bigcap_{\del>0}\bigcap_{\mathsf{h}>1}\limsup_{N\to\infty}\Supp\beta_{N,\del,i,\mathsf{h}}.}

The following contraction hypothesis for $\beta_{N,\del,i,\mathsf{h}}$ (Proposition~\ref{betadecay}) is derived from the above contraction hypothesis for $\alpha_i$ (or for $\widetilde{\alpha}_i$): if $\mathsf{h}$ is sufficiently large, then
\eqlabel{eq:contractionhypothesisonedimintro}{\int_{-\frac{1}{2}}^{\frac{1}{2}}\beta_{N,\del,i,\mathsf{h}}(u_{i}(s)x)ds\leq e^{-(\lambda^2-O(\delta))Nt}\widetilde{\alpha}_i(x).}
The proof of Proposition~\ref{betadecay} is inspired by the approach used in \cite[Theorem 1.5]{KKLM17}. The above inequality \eqref{eq:contractionhypothesisonedimintro} provides a covering number estimate for the sets $\Supp\big(\bxi\mapsto \beta_{N,\del,i,\mathsf{h}}(u(\bxi)x)\big)$. In view of \eqref{eq:divergentonaverageonedim}, this establishes that the Hausdorff dimension of the set of points on $Ux$ that are $A_i$-divergent on average is at most $\dim U-\frac{1}{2}$.

In \S 4.3 we introduce a dynamical height function $\psi_{N,\del,\mathsf{h}}$ for $A^+$ and prove its contraction hypothesis. Constructed from $\beta_{N,\del,i,\mathsf{h}}$'s, $\psi_{N,\del,\mathsf{h}}$ detects the averaged behavior of an $A^+$-trajectory near the cusp. Specifically, $\psi_{N,\del,\mathsf{h}}(x)$ is positive only if the $A^+$-trajectory $\set{a_{\btau}x: \btau\in\set{t,\ldots,dNt}^{d-1}}$ spends most of the time close to the cusp. Furthermore, if the $A^+$-trajectory exhibits an escape of mass exceeding $1-\del$, then there exists $\bsig_N\in (t\bN)^{d-1}$ with $\|\bsig_N\|=O(\del Nt)$ such that $\set{a_{\bsig_N+t\be_i}x,\ldots,a_{\bsig_N+dNt\be_i}}$ also has an escape of mass $1-O(\del)$ for every $1\leq i\leq d-1$. To summarize, for fixed $x\in X'$ if $u(\bxi)x$ is $A^+$-divergent then
\eqlabel{eq:divergentonaveragehigher}{\begin{aligned}\bxi&\in \bigcap_{\del>0}\bigcap_{\mathsf{h}>1}\limsup_{N\to\infty}\Supp\big(\bxi\mapsto\psi_{N,\del,\mathsf{h}}(u(\bxi)x)\big)\\&\subseteq\bigcap_{\del>0}\bigcap_{\mathsf{h}>1}\limsup_{N\to\infty}\bigcap_{i=1}^{d-1}\Supp\big(\bxi\mapsto\beta_{dN,\del,i,\mathsf{h}}(a_{\bsig_N}u(\bxi)x)\big).\end{aligned}}
In other words, to calculate the Hausdorff dimension of $A^+$-divergent on average points on $Ux$ it suffices to estimate
$$ \mathsf{Q}:=\int_{[-\frac{1}{2},\frac{1}{2}]^{d-1}}\prod_{i=1}^{d-1}\beta_{dN,\del,i,\mathsf{h}}(a_{\bsig_N}u(\bxi)x)d\bxi.$$

For a moment, let us assume that the maps $\bxi\mapsto\beta_{dN,\del,i,\mathsf{h}}(a_{\bsig_N}u(\bxi)x)$ for $1\leq i\leq d-1$ are independent. In combination with \eqref{eq:contractionhypothesisonedimintro} one can deduce that
$$\mathsf{Q}\asymp \prod_{i=1}^{d-1}\int_{[-\frac{1}{2},\frac{1}{2}]^{d-1}}\beta_{dN,\del,i,\mathsf{h}}(a_{\bsig_N}u(\bxi)x)d\bxi\ll e^{-d(d-1)(\lambda^2-o(1))Nt}\prod_{i=1}^{d-1}\widetilde{\alpha}_i(x).$$
This in turn implies that for $x\in X'$ the Hausdorff dimension of the set of points on $Ux$ that are $A^+$-divergent on average  is at most $\dim U-\frac{d-1}{2}=\frac{d-1}{2}$. Therefore, the proof of Theorem \ref{dimupperbdd} is reduced to verifying a certain independence among the maps $\bxi\mapsto\beta_{dN,\del,i,\mathsf{h}}(a_{\bsig_N}u(\bxi)x)$ for $1\leq i\leq d-1$, which constitutes the technical heart of \S3. The key observation for proving Proposition \ref{multicont} is that different one-dimensional subgroups in $A^+$ expand each direction in $U$ at different expansion rates. Namely, a diagonal element $a_{t\be_i}$ in $A_i$ expands $U_j$ at a rate of $e^{2t}$ if $i=j$, and $e^t$ otherwise. Using this, we approximate $\Supp\big(\bxi\mapsto\beta_{dN,\del,i,\mathsf{h}}(a_{\bsig_N}u(\bxi)x)\big)$ as the union of boxes in $U$ with side lengths $e^{-2dNt}$ along $U_i$ and $e^{-dNt}$ along $U_j$ for $j\neq i$. Then the intersection of these supports for $1\leq i\leq d-1$ is then described as the union of cubes of side length $e^{-2dNt}$. Computing the volume of this intersection yields the optimal contraction rate, which is crucial for bounding the Hausdorff dimension. The proof of Theorem~\ref{dimupperbdd}, based on this optimal contraction rate, is completed in \S4.4.

We now turn to the proof of Theorem~\ref{InhomLittlewood}. We identify $\widehat{X}=\operatorname{ASL}_d(\bR)/\operatorname{ASL}_d(\bZ)$ with the space of unimodular affine lattices in $\bZ^d$. The proof relies on a measure rigidity theorem from \cite{EL18}, which applies to $A$-action on this space.

For $\bxi\in\bR^{d-1}$ and $\btheta\in\bR^{d-1}$ we let $\widehat{x}_{\bxi,\btheta}\in \widehat{X}$ denote the corresponding affine lattice $u(\bxi)\bZ^d-\btheta$. By Dani's correspondence on the space of affine lattices (Lemma~\ref{InhomDani}), if $\bxi=(\xi_1,\ldots,\xi_{d-1})$ satisfies Littelwood's conjecture and $\displaystyle\liminf_{q\to\infty}q\prod_{i=1}^{d-1}\|q\xi_i-\theta_i\|_\bZ>0$ for some $\btheta=(\theta_1,\ldots,\theta_{d-1})$, then there exist $\eps>0$ such that $a_{\btau}\widehat{x}_{\bxi,\btheta}\subseteq \cL_\eps$ for any sufficiently large $\btau\in\bR^{d-1}_{>0}$, where
$$\cL_\eps:=\set{\Lambda\in\widehat{X}: \Lambda\cap B^{\bR^d}(0,\eps)=\emptyset}.$$ Note that the set $\cL_\eps$ is not compact for any $\eps>0$. In contrast, in the setting of homogeneous Diophantine approximation, the analogous set
$\set{\Lambda\in X: \Lambda\cap B^{\bR^d}(0,\eps)=\set{0}} $
in the space of lattices is compact by Mahler's compactness criterion. This is the principal reason why a more careful treatment on the escape of mass is necessary in the setting of the inhomogeneous Diophantine approximation. 

For the sake of contradiction, suppose that $\dim_H \Xi>\frac{d-1}{2}$. Theorem~\ref{dimupperbdd} implies that most of the points in the set $\Xi$ are not $A^+$-divergent on average. From the $A^+$-orbits of such points, one can construct an $A$-invariant measure $\mu$ on $\widehat{X}$ supported on $\cL_\eps$ for some $\eps>0$. Moreover, $\mu$ is nonzero, as it is constructed from the points that are not $A^+$-divergent on average. This measure construction procedure is similar in spirit to those in \cite{EKL06, LSS19, KKL}.

In Lemma \ref{Lesupp} we show that for any $\eps>0$ there is no homogeneous $A$-invariant measure supported on $\cL_\eps$. This contradicts the measure rigidity theorem \cite[Theorem 1.3]{EL18}, hence we conclude that $\dim_H\Xi\leq\frac{d-1}{2}$.

However, to apply \cite[Theorem 1.3]{EL18}, we must verify that $h_\mu(a)>0$ for some $a\in A$, as the measure rigidity theorem requires a positive entropy assumption. To establish this, in \S5 we strengthen the results in \S2 and \S3. Specifically, instead of focusing on the behavior of trajectories approaching the cusp, we construct dynamical height functions that capture the interaction of trajectories relative to a given reference trajectory. The contraction properties of these dynamical height functions ensure that the measure $\mu$ not only has positive mass on $\widehat{X}$, but also exhibits positive entropy for some diagonal element $a\in A$. This allows us to apply the measure rigidity theorem and derive the desired dimension bound.

\subsection{Structure of the paper}
This paper is organized as follows.

In Section~2 we modify the Benoist-Quint height function constructed in \cite{BQ12}. We then use these modified Benoist-Quint height functions to construct \textit{dynamical height functions} for one-dimensional diagonal subgroups and establish the contraction hypotheses for those dynamical height functions (Proposition \ref{betadecay}). The proof of Proposition \ref{betadecay} is inspired by \cite[Proof of Theorem 1.5]{KKLM17}. 

In Section~3 we introduce dynamical height functions for higher-rank actions and establish the contraction hypothesis for higher-rank actions (Proposition \ref{multicont}). 

In Section~4 we establish the upper bound for the Hausdorff dimension of $A^+$-divergent on average trajectories (Theorem \ref{dimupperbdd}) using Proposition \ref{multicont}. Additionally, we derive the lower bound (Theorem \ref{dimlowbdd}) by employing Dani's correspondence in Diophantine approximation.

In Section~5 we discuss variants of the dynamical height functions defined in Sections~2 and 3. We then establish the contraction hypotheses for these dynamical height functions for a given trajectory in Proposition~\ref{nontrivbdd:traj}.

In Section~6 we deduce Theorem \ref{InhomLittlewood} from several ingredients: Dani's correspondence on the space of affine lattices (Lemma \ref{InhomDani}), a measure rigidity result in \cite{EL18} for higher-rank diagonal actions on the space of affine lattices, and Theorem \ref{dimupperbdd}. We also make use of Proposition \ref{nontrivbdd:traj} to verify the positive entropy assumption of the measure rigidity theorem \cite[Theorem 1.3]{EL18}.

\vspace{3mm}
\tb{Acknowledgments}. This paper is part of my Ph.D. thesis conducted at ETH Z\"{u}rich under the guidance of Prof. Manfred Einsiedler. I am deeply grateful to him for numerous enlightening discussions and continuous encouragement throughout this work. I also thank Pengyu Yang for valuable conversations on height functions.

\section{Height functions for one-dimensional actions}
In this section, we construct and modify several height functions for one-dimensional diagonal subgroups to control their recurrence to compact subsets.

\subsection{Notations}
We shall use the standard notation $A\ll B$ or $A=O(B)$ to mean that $A\leq CB$ for some constant $C>0$ that depends only on $d$, unless other dependencies are explicitly stated. We shall also write $A\asymp B$ to mean both $A\ll B$ and $B\ll A$. 

We use the notation $C_1, C_2, ...$ to denote positive constants that depend only on the dimension $d$. In the course of the proofs, we will introduce an auxiliary parameter $0<\lambda<1$ for which we will take $\lambda\to 1$ in the end. The notation $\useconE{101},\useconE{102}, ...$ will denote positive constants that depend only on $d$ and $\lambda$.

Throughout this paper, we denote by $\set{\be_1,\ldots,\be_{d-1}}$ the standard basis of $\bR^{d-1}$. For $1\leq i\leq d-1$ we let
$$A_i:=\set{a_{t\be_i}\in G: t\in\bR},\qquad\qquad U_i:=\set{u(s\be_i)\in G:s\in\bR}.$$
We also denote by $U_i^{\perp}$ the orthogonal complement of $U_i$ in $U$, and $H_i$ the $\operatorname{SL}_2(\bR)$-copy in $G$ containing $A_i$ and $U_i$.

Let $\bd_G$ be a right-invariant metric on $G$. By multiplying a constant if necessary, we may assume that 
$$\bd_G(\operatorname{id},a_{\be_i})\leq 1, \quad  \bd_G(\operatorname{id},u(\be_i))\leq 1$$
for any $1\leq i\leq d-1$. Given a Lie subgroup $S$ of $G$, we denote by $\bd_S$ the metric on $S$ induced by $\bd_G$, and denote by $B^S(g,r)$ be the open $r$-ball around $g$ with respect to the metric $\bd_S$ for $g\in S$ and $r>0$.

\subsection{Benoist-Quint height function}
We first recall the construction of the height function from \cite{BQ12}. We decompose $\oplus_{k=1}^{d-1}\bigwedge^k\bR^d$ into the direct sum of irreducible sub-representations of $H_i$. For the standard basis $\set{\be_1,\ldots,\be_d}$ of $\bR^d$ denote by $W_i$ the subspace of $\bR^d$ spanned by $\set{\be_i,\be_d}$ and $W_i^{\perp}$ be the subspace of $\bR^d$ spanned by $\set{\be_1,\ldots,\be_{d-1}}\setminus\set{\be_i}$. For $1\leq i\leq d-1$, the projections from $\bR^{d-1}$ to $W_i^{\perp}$ will be denoted by $\pi_i^{\perp}$.

Note that $\bigwedge^j W_i=0$ for any $j>2$. It follows that for any $1\leq k\leq d-1$, $\bigwedge^k\bR^d$ is the direct sum of $\bigwedge^k W_i^{\perp}$, $W_i\wedge(\bigwedge^{k-1} W_i^{\perp})$, and $(\bigwedge^{2} W_i)\wedge(\bigwedge^{k-2} W_i^{\perp})$, where the last one appears only if $k\ge 2$. Let us denote
\eqlabel{eq:Vdef}{V_{i,k,0}:=\bigwedge\nolimits^{\!k} W_i^{\perp}\oplus\left(\big(\bigwedge\nolimits^{\!2} W_i\big)\wedge \big(\bigwedge\nolimits^{\!k-2}W_i^{\perp}\big)\right), \quad V_{i,k,1}:=\left(\bigwedge\nolimits^{\!k-1} W_i\right)\wedge W_i^{\perp}.}
As $H_i$ acts on $W_i$ and $\bigwedge^{2} W_i^{\perp}$ trivially, $H_i$ acts on $V_{i,k,0}$ trivially. Also, $V_{i,k,1}$ is the direct sum of all irreducible sub-representations of $\bigwedge^k\bR^d$ whose highest weight is $1$ as the highest weight of $W_i^{\perp}$ is $1$. We denote by $\pi_{i,k,0}$ and $\pi_{i,k,1}$ the projections onto $V_{k,0}$ and $V_{k,1}$, respectively.

For any $1\leq i,k\leq d-1$ and $\eps>0$ we define $\varphi_{i,k,\eps}:\bigwedge^k\bR^d\to\bR_{\ge0}$ by
\eq{\varphi_{i,k,\eps}(v)=\begin{dcases}\eps^{k(d-k)}\|\pi_{i,k,1}(v)\|^{-\lambda}\qquad \textrm{if }\|\pi_{i,k,0}(v)\|<\eps^{k(d-k)},\\ 0 \qquad\qquad\qquad\qquad\quad \textrm{otherwise}\end{dcases}} for $v\in \bigwedge^k\bR^d$.

For $x\in X$, we can identify $x$ with a unimodular lattice in $\bR^d$ and denote by $P(x)$ the set of all primitive subgroups of the lattice $x$. A non-zero monomial $v_1\wedge\cdots\wedge v_k\in \bigwedge^k\bR^d$ is said to be $x$-integral if the subgroup generated by $v_1,\ldots,v_k$ belongs to $P(x)$. We define $\alpha_{i,\eps}:X\to\bR_{\ge0}$ by
$$\alpha_{i,\eps}(x):=\max \varphi_{i,k,\eps}(v),$$
where the maximum is taken over all the non-zero $x$-integral monomial $v\in \bigwedge^k\bR^d$ for some $1\leq k\leq d-1$. The following properties of the Benoist-Quint height function were first proved by \cite{BQ12} and further refined by \cite{Shi20} and \cite{Kha20}.

\begin{prop}\label{alpha}
There exists \newconC{1}$\useconC{1}>1$ such that for any $0<\lambda<1$, there exists $T_\lambda>0$ such that the following holds. For any $t\ge T_{\lambda}$ there exist $\eps_{\lambda,t}>0$ and $C_{\lambda,t}>1$ such that for each $1\leq i \leq d-1$ the function $\alpha_{i,\lambda,t}:=\alpha_{i,\eps_{\lambda,t}}$ is a proper lower semi-continuous function with the following properties: 
\begin{enumerate}
    \item (Subharmonic estimate) For any $x\in X$ \eqlabel{alphacont}{\int_{-\frac{1}{2}}^{\frac{1}{2}}\alpha_{i,\lambda,t}(a_{t\be_i}u_i(s)x)ds\leq e^{-\lambda t}\alpha_{i,\lambda,t}(x)+C_{\lambda,t},}
    \item (Log-Lipschitz property) For any $g\in B^{H_i}(\operatorname{id},1)$ and $x\in X$ \eqlabel{Lipschitz}{\useconC{1}^{-1}\alpha_{i,\lambda,t}(x)\leq \alpha_{i,\lambda,t}(gx)\leq \useconC{1}\alpha_{i,\lambda,t}(x).}
\end{enumerate}
\end{prop}
\begin{proof}
This proposition is a special case of \cite[Lemma 4.1]{Shi20} with $L=\operatorname{SL}_d(\bR)$ and $G=H_i$. The range of $\lambda$ is optimized in \cite[Theorem 12.3]{Kha20}.
\end{proof}



We denote
$$\overline{a}_t:=\left(\begin{matrix} e^t & \\ & e^{-t}\end{matrix}\right),\quad \overline{u}(s):=\left(\begin{matrix} 1 &  s \\ & 1\end{matrix}\right)$$
for $s,t\in\bR$. For $0<\lambda<1$ we also define a height function $\operatorname{ht}_\lambda:\operatorname{SL}_2(\bR)/\operatorname{SL}_2(\bZ)\to[1,\infty)$ by 
$$\operatorname{ht}_\lambda(x):=\sup_{v\in g\bZ^2}\|v\|^{-\lambda}$$
for $x=g\operatorname{SL}_2(\bZ)$ with $g\in \operatorname{SL}_2(\bR)$.
\begin{lem}
    The function $\operatorname{ht}_\lambda$ is a proper continuous function. There exists \newconC{2}$\useconC{2}>1$ such that for any $0<\lambda<1$, there exists $T_{\lambda}'>0$ such that the following holds. For any $t\ge T_{\lambda}'$ there exists $C_{\lambda,t}'>1$ such that $\operatorname{ht}_\lambda$ satisfies the following properties: 
\begin{enumerate}
    \item For any $x\in \operatorname{SL}_2(\bR)/\operatorname{SL}_2(\bZ)$ \eqlabel{htcont}{\int_{-\frac{1}{2}}^{\frac{1}{2}}\operatorname{ht}_\lambda(\overline{a}_{t}\overline{u}(s)z)ds\leq e^{-\lambda t}\operatorname{ht}_\lambda(x)+C_{\lambda,t}',}
    \item For any $g\in B^{\operatorname{SL}_2(\bR)}(\operatorname{id},1)$ and $x\in X$ \eqlabel{htLipschitz}{\useconC{2}^{-1}\operatorname{ht}_\lambda(x)\leq \operatorname{ht}_\lambda(gx)\leq \useconC{2}\operatorname{ht}_\lambda(x).}
\end{enumerate}
\end{lem}
\begin{proof}
As in Proposition~\ref{alpha}, this proposition is a special case of \cite[Lemma 4.1]{Shi20} with $L=G=\operatorname{SL}_2(\bR)$. The range of $\lambda$ is optimized in \cite[Theorem 12.3]{Kha20}.
\end{proof}

\subsection{Modification of Benoist-Quint height function}
The Benoist-Quint height functions in Proposition~\ref{alpha} satisfy the log-Lipschitz condition \eqref{Lipschitz} along directions within the subgroup $H_i$, but may exhibit wild behavior along directions outside $H_i$. In this subsection, we shall modify the Benoist-Quint height functions to ensure that the modified functions provide control even along directions outside $H_i$.

Let $X':=A^+U\Gamma\subset X$. Note that any element in $X'$ can be uniquely parameterized in the form $a_{\btau}u(\bxi)\Gamma$, where $\btau\in\bR^{d-1}_{>0}$ and $\bxi\in[-\frac{1}{2},\frac{1}{2})^{d-1}$. For each $1\leq i\leq d-1$ define $\kappa_i:X'\to [1,\infty)$ and $\phi_i:X'\to Z$ by
$$\kappa_i(a_{\btau}u(\bxi)\Gamma)=e^{(\tau_1+\cdots+\tau_{d-1})-\tau_i},$$
$$\phi_i(a_{\btau}u(\bxi)\Gamma)=\overline{a}_{\tau_1+\cdots+\tau_{d-1}}\overline{u}(\xi_i)\operatorname{SL}_2(\bZ),$$
where $\btau=(\tau_1,\ldots,\tau_{d-1})\in\bR^{d-1}_{>0}$ and $\bxi=(\xi_1,\ldots,\xi_{d-1})\in\bR^{d-1}$. Note that $\kappa_i$ is invariant under $A_i^{+}$ and $U$, and there exists a constant \newconC{3}$\useconC{3}>1$ such that
\eqlabel{kappaLipschitz}{\useconC{3}^{-1}\kappa_i(x)\leq \kappa_i(gx)\leq \useconC{3}\kappa_i(x)}
for any $1\leq i\leq d-1$, $g\in B^{A^+U}(\operatorname{id},1)$, and $x\in X'$. We also define $\kappa:X'\to[1,\infty)$ by $$\kappa(x):=\displaystyle\prod_{i=1}^{d-1}\kappa_i(x).$$

For each $1\leq i\leq d-1$ and $0<\lambda<1$ we define an auxiliary height function $\alpha_{i,\lambda}':X'\to[1,\infty)$ by
$$\alpha_{i,\lambda}'(x)=\kappa_i(x)\operatorname{ht}_{\lambda}(\phi_i(x)).$$

\begin{prop}\label{alpha'}
Let $1\leq i\leq d-1$ and $0<\lambda<1$. There exists \newconC{4}$\useconC{4}>1$ such that the function $\alpha_{i,\lambda}'$ is a proper function with the following properties: 
\begin{enumerate}
    \item (Subharmonic estimate for $\alpha'$) For any $x\in X'$ \eqlabel{alpha'cont}{\int_{-\frac{1}{2}}^{\frac{1}{2}}\alpha_{i,\lambda}'(a_{t\be_i}u_i(s)x)ds\leq e^{-\lambda t}\alpha_{i,\lambda}'(x)+C_{\lambda,t}'\kappa_i(x),}
    \item (Log-Lipshictz property along $A^+U$) For any $g\in B^{A^+U}(\operatorname{id},1)$ and $x\in X'$ \eqlabel{Lipschitz'}{\useconC{4}^{-1}\alpha_{i,\lambda}'(x)\leq \alpha_{i,\lambda}'(gx)\leq \useconC{4}\alpha_{i,\lambda}'(x).}
    \item (Stability along $U_i^{\perp}$) For any $g\in U_i^{\perp}$ and $x\in X'$, $\alpha_{i,\lambda}'(gx)=\alpha_{i,\lambda}'(x)$.
\end{enumerate}
\end{prop}
\begin{proof}
    We first show that $\alpha_{i,\lambda}'$ is proper. By Mahler's compactness criterion, it suffices to prove the following claim:

    \textbf{Claim.} Let $\btau\in\bR^{d-1}_{>0}$, $\bxi\in\bR^{d-1}$, and $\mathsf{h}\geq1$. If $\alpha_{i,\lambda}'(a_{\btau}u(\bxi)\Gamma)\leq \mathsf{h}$ then $\|v\|\geq \mathsf{h}^{-\frac{1}{\lambda}}$ for any $v\in a_{\btau}u(\bxi)\bZ^d$.

    \begin{proof}[Proof of \textbf{Claim}]
        If $\alpha_{i,\lambda}'(a_{\btau}u(\bxi)\Gamma)\leq \mathsf{h}$, then $\operatorname{ht}_\lambda(\phi_i(a_{\btau}u(\bxi)\Gamma))\leq \kappa_i(x)^{-1}\mathsf{h}$. Note that the nonzero vectors in $\phi_i(a_{\btau}u(\bxi)\Gamma)$ can be written as $\left(\begin{matrix}
        e^{\tau_1+\cdots+\tau_{d-1}}(m+n\xi_i)\\ e^{-(\tau_1+\cdots+\tau_{d-1})}n
    \end{matrix}\right)$, where $(m,n)\in\bZ^2\setminus\set{0}$. It follows that for any $(m,n)\in\bZ^2\setminus\set{0}$,
    $$\left\|\left(\begin{matrix}
        e^{\tau_1+\cdots+\tau_{d-1}}(m+n\xi_i)\\ e^{-(\tau_1+\cdots+\tau_{d-1})}n
    \end{matrix}\right)\right\|\geq \big(\kappa_i(x)^{-1}\sh\big)^{-\frac{1}{\lambda}}\geq (e^{-(\tau_1+\cdots+\tau_{d-1})+\tau_i}\mathsf{h})^{-\frac{1}{\lambda}},$$
    hence
    \eqlabel{TwoDimensionalBound}{\left\|\left(\begin{matrix}
        e^{\tau_i}(m+n\xi_i)\\ e^{-(\tau_1+\cdots+\tau_{d-1})}n
    \end{matrix}\right)\right\|\geq e^{-(\tau_1+\cdots+\tau_{d-1})+\tau_i}\left\|\left(\begin{matrix}
        e^{\tau_1+\cdots+\tau_{d-1}}(m+n\xi_i)\\ e^{-(\tau_1+\cdots+\tau_{d-1})}n
    \end{matrix}\right)\right\| \geq \mathsf{h}^{-\frac{1}{\lambda}}.}
    Observe that any nonzero vector in $a_{\btau}u(\bxi)\bZ^d$ can be written as $$v_{(m_1,\ldots,m_d)}=\left(\begin{matrix}
        e^{\tau_1}(m_1+m_d\xi_1)\\ \vdots \\ e^{\tau_{d-1}}(m_{d-1}+m_d\xi_{d-1}) \\  e^{-(\tau_1+\cdots+\tau_{d-1})}m_d
    \end{matrix}\right)$$ for some $(m_1,\ldots,m_d)\in\bZ^d\setminus\set{0}$. If $m_d=0$, then we have $\|v_{(m_1,\ldots,m_d)}\|\geq 1$. If $m_d\neq 0$, then $$\|v_{(m_1,\ldots,m_d)}\|\geq \left\|\left(\begin{matrix}
        e^{\tau_i}(m_i+m_d\xi_i)\\ e^{-(\tau_1+\cdots+\tau_{d-1})}m_d
    \end{matrix}\right)\right\|\geq\mathsf{h}^{-\frac{1}{\lambda}}$$ by \eqref{TwoDimensionalBound}. We thus conclude that $\|v\|\geq \mathsf{h}^{-\frac{1}{\lambda}}$ for any $v\in a_{\btau}u(\bxi)\bZ^d$.
    \end{proof}

    We now verify the properties (1)-(3).

    (1) It follows from the definition of $\alpha_i'$ and \eqref{htcont}:
    \eq{\begin{aligned}
        \int_{-\frac{1}{2}}^{\frac{1}{2}}\alpha_{i,\lambda}'(a_{t\be_i}u_i(s)x)ds&=\int_{-\frac{1}{2}}^{\frac{1}{2}}\kappa_i(a_{t\be_i}u_i(s)x)\operatorname{ht}_\lambda(\phi_i(a_{t\be_i}u_i(s)x))ds\\
        &=\kappa_i(x)\int_{-\frac{1}{2}}^{\frac{1}{2}}\operatorname{ht}_\lambda(\phi_i(a_{t\be_i}u_i(s)x))ds\\
        &\leq \kappa_i(x)(e^{-\lambda t}\operatorname{ht}_\lambda(\phi_i(x))+C_{\lambda,t}')\\&=e^{-\lambda t}\alpha_{i,\lambda}'(x)+C_{\lambda,t}'\kappa_i(x).
    \end{aligned}}

    (2) This is a direct consequence of \eqref{htLipschitz} and \eqref{kappaLipschitz}.

    (3) One can easily check that both $\kappa_i$ and $\phi_i$ are invariant under $U_i^{\perp}$.
\end{proof}

The following lemma gives an estimate between the height functions $\alpha_{i,\lambda,t}$ and $\alpha_{i,\lambda}'$.
\begin{lem}\label{alpha'upperbound}
     For any $1\leq i\leq d-1$, $0<\lambda<1$, $t\ge T_{\lambda}$ and $x\in X'$ we have \eq{\alpha_{i,\lambda,t}(x)\geq \eps_{\lambda,t}^{d-1}\kappa_i(x)^{-2}\alpha_{i,\lambda}'(x).}
\end{lem}
\begin{proof}
By definition of $\alpha_{i,\lambda,t}$ it suffices to prove that $$\varphi_{i,d-1,\eps_{\lambda,t}}(x)\geq \eps_{\lambda,t}^{d-1}\kappa_i(x)^{-2}\alpha_{i,\lambda}'(x).$$ For $1\leq i\leq d$ let us write $\widehat{\be}_i=\bigwedge_{j\neq i}\be_j$. It follows from \eqref{eq:Vdef} that $V_{i,d-1,0}$ is the $(d-2)$-dimensional subspace of $\bigwedge^{d-1}\bR^{d}$ spanned by $\set{\widehat{\be}_1,\ldots,\widehat{\be}_{d-1}}\setminus\set{\widehat{\be}_{i}}$ and $V_{i,d-1,1}$ is the $2$-dimensional subspace of $\bigwedge^{d-1}\bR^{d}$ spanned by $\set{\widehat{\be}_i,\widehat{\be}_{d}}$.

Let $x=a_{\btau}u(\bxi)\Gamma$ with $\btau=(\tau_1,\ldots,\tau_{d-1})\in \bR^{d-1}_{>0}$ and $\bxi=(\xi_1,\ldots,\xi_{d-1})\in \bR^{d-1}$. By straightforward matrix calculation one can check that any non-zero $x$-integral monomial in $\bigwedge^{d-1}\bR^{d}$ is a $\bZ$-linear combination of $$\set{e^{-\tau_1}\widehat{\be}_1+e^{\tau_1+\cdots+\tau_{d-1}}\widehat{\be}_d,\ldots,e^{-\tau_{d-1}}\widehat{\be}_{d-1}+e^{\tau_1+\cdots+\tau_{d-1}}\widehat{\be}_d,e^{\tau_1+\cdots+\tau_{d-1}}\widehat{\be}_d},$$
hence can be written as
$$\widehat{v}_{(m_1,\ldots,m_d)}=\left(\begin{matrix}
        e^{-\tau_1}m_1\\ \vdots \\ e^{-\tau_{d-1}}m_{d-1} \\  e^{\tau_1+\cdots+\tau_{d-1}}(m_d+\xi_1m_1+\cdots+\xi_{d-1}m_{d-1})
    \end{matrix}\right)$$
    with respect to the basis $\set{\widehat{\be}_1,\ldots,\widehat{\be}_d}$. Note that $\pi_{i,d-1,0}(\widehat{v}_{(m_1,\ldots,m_d)})=0$ if $m_j=0$ for all $j\neq i,d$. It follows that
\eq{\varphi_{i,d-1,\eps_{\lambda,t}}(x)\geq \eps_{\lambda,t}^{d-1}\sup\|\pi_{i,d-1,1}(\widehat{v}_{(m_1,\ldots,m_d)})\|^{-\lambda}}
where the supremum is taken over all $(m_1,\cdots,m_d)\in\bZ^{d}\setminus\set{0}$ with $m_j=0$ for all $j\neq i,d$. Hence,
\eq{\begin{aligned}\varphi_{i,d-1,\eps_{\lambda,t}}(x)&\geq \eps_{\lambda,t}^{d-1}\sup_{(m_i,m_d)\in\bZ^2\setminus\set{0}}\left\|\left(\begin{matrix}e^{-\tau_i}m_i \\ e^{\tau_1+\cdots+\tau_{d-1}}(m_d+\xi_im_i) \end{matrix}\right)\right\|^{-\lambda}\\&\geq\eps_{\lambda,t}^{d-1}\sup_{(m_i,m_d)\in\bZ^2\setminus\set{0}}\left\|e^{(\tau_1+\cdots+\tau_{d-1})-\tau_i}\left(\begin{matrix}e^{-(\tau_1+\cdots+\tau_{d-1})}m_i \\ e^{\tau_1+\cdots+\tau_{d-1}}(m_d+\xi_im_i) \end{matrix}\right)\right\|^{-\lambda}\\&=\eps_{\lambda,t}^{d-1}e^{-\lambda\{(\tau_1+\cdots+\tau_{d-1})-\tau_i\}}\operatorname{ht}_\lambda(\overline{a}_{\tau_1+\cdots+\tau_{d-1}}\overline{u}(\xi_i)\Gamma)\\&\geq \eps_{\lambda,t}^{d-1}\kappa_i(x)^{-1}\operatorname{ht}_\lambda(\phi_i(x))=\eps_{\lambda,t}^{d-1}\kappa_i(x)^{-2}\alpha_{i,\lambda}'(x), \end{aligned}}
completing the proof.
\end{proof}

We are now ready to construct modified Benoist-Quint functions. Let \newconC{5}$\useconC{5}:=\max\set{\useconC{1},\useconC{2}}$,  $t_\lambda:=\lceil\max\set{T_\lambda,T_{\lambda}',\frac{\log (10\useconC{5})}{\lambda-\lambda^2}}\rceil$, and \newconE{101}$\useconE{101}:=\max\set{5C_{\lambda,t_\lambda},5C_{\lambda,t_\lambda}',\eps_{\lambda,t_\lambda}^{-(d-1)}}$. For $0<\lambda<1$ and $1\leq i\leq d-1$ we define $\widetilde{\alpha}_{i,\lambda}:X'\to[1,\infty)$ by
\eqlabel{alphatildedefinition}{\widetilde{\alpha}_{i,\lambda}(x)=\begin{cases}\max\set{\alpha_{i,\lambda,t_{\lambda}}(x),1} & \text{if $\alpha_{i,\lambda,t_\lambda}(x)\leq \useconE{101}\kappa_i(x)$,}\\ \min\set{\alpha_{i,\lambda,t_{\lambda}}(x),\alpha_{i,\lambda}'(x)} & \text{otherwise}.\end{cases}}

\begin{prop}\label{alphatilde}
    Let $0<\lambda<1$, $1\leq i\leq d-1$, and $t\in t_{\lambda}\bN$. The function $\widetilde{\alpha}_{i,\lambda}$ is a proper lower semi-continuous function with the following properties: 
\begin{enumerate}
    \item For any $x\in X'$
    \eqlabel{alphatildeupperbound}{\int_{-\frac{1}{2}}^{\frac{1}{2}}\widetilde{\alpha}_{i,\lambda}(a_{t\be_i}u_i(s)x)ds\leq \useconE{101}\widetilde{\alpha}_{i,\lambda}(x).}
    \item For any $x\in X'$ with $\widetilde{\alpha}_{i,\lambda}(x)\ge \useconE{101}e^{t}$, \eqlabel{alphatildecont}{\int_{-\frac{1}{2}}^{\frac{1}{2}}\widetilde{\alpha}_{i,\lambda}(a_{t\be_i}u_i(s)x)ds\leq 2e^{-\lambda^2 t}\widetilde{\alpha}_{i,\lambda}(x).}
    \item For any $g\in B^{A_i^{+}U_i}(\operatorname{id},1)$ and $x\in X'$ \eqlabel{alphatildeLipschitz}{\useconC{5}^{-1}\widetilde{\alpha}_{i,\lambda}(x)\leq \widetilde{\alpha}_{i,\lambda}(gx)\leq \useconC{5}\widetilde{\alpha}_{i,\lambda}(x).}
    \item For any $g\in U_i^{\perp}$ and $x\in X'$, \eqlabel{LipschitzOrthogonal}{(\useconE{101}\kappa_i(x)^{2})^{-1}\widetilde{\alpha}_{i,\lambda}(x)\leq \widetilde{\alpha}_{i,\lambda}(gx)\leq \useconE{101}\kappa_i(x)^{2}\widetilde{\alpha}_{i,\lambda}(x).}
\end{enumerate}
\end{prop}
\begin{proof}
For any $m\in\bN$ we may cover $[-\frac{1}{2},\frac{1}{2}]$ by disjoint intervals $I_1,\ldots, I_{\lceil e^{2t_\lambda m}\rceil}$ of length $e^{-2t_{\lambda}m}$, where $s_k=-\frac{1}{2}+(k-\frac{1}{2})e^{-2t_{\lambda}m}$ and $I_k=[s_k-\tfrac{1}{2}e^{-2t_\lambda m},s_k+\frac{1}{2}e^{-2t_\lambda m}]$ for $1\leq k\leq \lceil e^{2t_\lambda m}\rceil$. 

Then by Proposition \ref{alpha} we have
\eqlabel{eq:inductivecont1}{\begin{aligned}
    \int_{-\frac{1}{2}}^{\frac{1}{2}}\alpha_{i,\lambda,t_\lambda}(a_{(m+1)t_\lambda\be_i}&u_i(s)x)ds \leq e^{-2t_\lambda m}\displaystyle\sum_{k=1}^{\lceil e^{2t_\lambda m}\rceil}\int_{-\frac{1}{2}}^{\frac{1}{2}}\alpha_{i,\lambda,t_\lambda}(a_{t_\lambda }u(s)a_{mt_\lambda\be_i}u_i(s_k)x)ds\\&\leq e^{-2t_\lambda m}\displaystyle\sum_{k=1}^{\lceil e^{2t_\lambda m}\rceil}\big(e^{-\lambda t_\lambda}\alpha_{i,\lambda,t_\lambda}(a_{mt_\lambda\be_i}u_i(s_k)x)+C_{\lambda,t_{\lambda}}\big)\\
    &\leq 2\useconC{1}e^{-\lambda t_\lambda} \int_{-\frac{1}{2}}^{\frac{1}{2}}\alpha_{i,\lambda,t_\lambda}(a_{mt_\lambda\be_i}u_i(s)x)ds+2C_{\lambda,t_\lambda}.
\end{aligned}}
Applying \eqref{eq:inductivecont1} repeatedly for $m=0,1,\ldots, \frac{t}{t_\lambda}-1$, we get
\eqlabel{eq:inductivecont1'}{\begin{aligned}\int_{-\frac{1}{2}}^{\frac{1}{2}}\alpha_{i,\lambda,t_\lambda}(a_{t\be_i}u_i(s)x)ds &\leq (2\useconC{1} e^{-\lambda t_\lambda})^{\frac{t}{t_\lambda}}\int_{-\frac{1}{2}}^{\frac{1}{2}}\alpha_{i,\lambda,t_\lambda}(u_i(s)x)ds+2C_{\lambda,t_\lambda}\sum_{m=0}^{\frac{t}{t_\lambda}-1}(2\useconC{1}e^{-\lambda t_{\lambda}})^{-m}\\&\leq \useconC{1}(2\useconC{1} e^{-\lambda t_\lambda})^{\frac{t}{t_\lambda}}\alpha_{i,\lambda,t_\lambda}(x)+\frac{2C_{\lambda,t_\lambda}}{1-2\useconC{1}e^{-\lambda t_\lambda}}\\&\leq e^{-\lambda^2 t}\alpha_{i,\lambda,t_\lambda}(x)+4C_{\lambda,t_\lambda}.\end{aligned}}
Similarly, by Proposition \ref{alpha'} we also have
\eqlabel{eq:inductivecont2}{\begin{aligned}
    \int_{-\frac{1}{2}}^{\frac{1}{2}}\alpha_{i,\lambda}'(a_{(m+1)t_\lambda\be_i}&u_i(s)x)ds \leq e^{-2t_\lambda m}\displaystyle\sum_{k=1}^{\lceil e^{2t_\lambda m}\rceil}\int_{-\frac{1}{2}}^{\frac{1}{2}}\alpha_{i,\lambda}'(a_{t_\lambda }u(s)a_{mt_\lambda\be_i}u_i(s_k)x)ds\\&\leq e^{-2t_\lambda m}\displaystyle\sum_{k=1}^{\lceil e^{2t_\lambda m}\rceil}\big(e^{-\lambda t_\lambda}\alpha_{i,\lambda}'(a_{mt_\lambda\be_i}u_i(s_k)x)+C_{\lambda,t_{\lambda}}'\kappa_i(x)\big)\\
    &\leq 2\useconC{4} e^{-\lambda t_\lambda} \int_{-\frac{1}{2}}^{\frac{1}{2}}\alpha_{i,\lambda}'(a_{mt_\lambda\be_i}u_i(s)x)ds+2C_{\lambda,t_\lambda}'\kappa_i(x).
\end{aligned}}
Applying \eqref{eq:inductivecont2} repeatedly for $m=0,1,\ldots, \frac{t}{t_\lambda}-1$, we get
\eqlabel{eq:inductivecont2'}{\begin{aligned}\int_{-\frac{1}{2}}^{\frac{1}{2}}\alpha_{i,\lambda}'(a_{t\be_i}u_i(s)x)ds &\leq \useconC{1}(2\useconC{1} e^{-\lambda t_\lambda})^{\frac{t}{t_\lambda}}\alpha_{i,\lambda}'(x)+\frac{2C_{\lambda,t_\lambda}'\kappa_i(x)}{1-2\useconC{1}e^{-\lambda t_\lambda}}\\&\leq e^{-\lambda^2 t}\alpha_{i,\lambda}'(x)+4C_{\lambda,t_\lambda}'\kappa_i(x).\end{aligned}}

(1) If $\widetilde{\alpha}_{i,\lambda,t_\lambda}(x)=\max\set{\alpha_{i,\lambda,t_\lambda}(x),1}$, then by \eqref{eq:inductivecont1'} we have
\eq{\int_{-\frac{1}{2}}^{\frac{1}{2}}\alpha_{i,\lambda,t_\lambda}(a_{t\be_i}u_i(s)x)ds\leq e^{-\lambda^2 t}\alpha_{i,\lambda,t_\lambda}(x)+4C_{\lambda,t_\lambda}\leq 5C_{\lambda,t_\lambda}\widetilde{\alpha}_{i,\lambda}(x),}
hence
\eq{\int_{-\frac{1}{2}}^{\frac{1}{2}}\widetilde{\alpha}_{i,\lambda}(a_{t\be_i}u_i(s)x)ds\leq \useconE{101}\widetilde{\alpha}_{i,\lambda}(x).}
Otherwise we have $\widetilde{\alpha}_{i,\lambda,t_\lambda}(x)=\alpha_{i,\lambda}'(x)$, then we use \eqref{eq:inductivecont2'} and get
\eq{\begin{aligned}
    \int_{-\frac{1}{2}}^{\frac{1}{2}}\widetilde{\alpha}_{i,\lambda}(a_{t\be_i}u_i(s)x)ds&\leq \int_{-\frac{1}{2}}^{\frac{1}{2}}\alpha_{i,\lambda}'(a_{t\be_i}u_i(s)x)ds\\&\leq e^{-\lambda^2 t}\alpha_{i,\lambda}'(x)+4C_{\lambda,t_\lambda}'\kappa_i(x)\\&\leq (e^{-\lambda^2 t}+4C_{\lambda,t_\lambda}')\alpha_{i,\lambda}'(x)\leq \useconE{101}\alpha_{i,\lambda}'(x).
\end{aligned}}

(2) In this case $\widetilde{\alpha}_{i,\lambda}(x)$ is either $\alpha_{i,\lambda,t_{\lambda}}(x)$ or $\alpha_{i,\lambda}'(x)$ by definition of $\widetilde{\alpha}_{i,\lambda}$. If $\alpha_{i,\lambda,t_{\lambda}}(x)=\widetilde{\alpha}_{i,\lambda}(x)\geq \useconE{101}e^t\geq4\widetilde{C}_{\lambda,t_{\lambda}}e^{t}$, then \eqref{eq:inductivecont1'} implies that
\eq{\begin{aligned}\int_{-\frac{1}{2}}^{\frac{1}{2}}\widetilde{\alpha}_{i,\lambda}&(a_{t\be_i}u_i(s)x)ds\leq \int_{-\frac{1}{2}}^{\frac{1}{2}}\alpha_{i,\lambda,t_{\lambda}}(a_{t\be_i}u_i(s)x)ds\\&\leq e^{-\lambda^2 t}\alpha_{i,\lambda,t_{\lambda}}(x)+4C_{\lambda,t_\lambda}<  2e^{-\lambda^2 t}\alpha_{i,\lambda,t_{\lambda}}(x).\end{aligned}}
Otherwise we may assume that 
$$\widetilde{\alpha}_{i,\lambda}(x)=\alpha_{i,\lambda}'(x)>\alpha_{i,\lambda,t_{\lambda}}(x)>4\widetilde{C}_{\lambda,t_{\lambda}}e^{t}\kappa_i(x).$$
In this case \eqref{eq:inductivecont2'} implies that
\eq{\begin{aligned}\int_{-\frac{1}{2}}^{\frac{1}{2}}\widetilde{\alpha}_{i,\lambda}&(a_{t\be_i}u_i(s)x)ds\leq \int_{-\frac{1}{2}}^{\frac{1}{2}}\alpha_{i,\lambda}'(a_{t\be_i}u_i(s)x)ds\\&\leq e^{-\lambda^2 t}\alpha_{i,\lambda}'(x)+4C_{\lambda,t_\lambda}'\kappa_i(x)<  2e^{-\lambda^2 t}\alpha_{i,\lambda}'(x).\end{aligned}}

(3) It directly follows from \eqref{Lipschitz} and $\eqref{Lipschitz'}$.

(4) By Lemma \ref{alpha'upperbound} we have
$$\eps_{\lambda,t_\lambda}^{d-1}\kappa_i(x)^{-2}\alpha_{i,\lambda}'(x)\leq \widetilde{\alpha}_{i,\lambda,t}(x)\leq \alpha_{i,\lambda}'(x)$$
for any $x\in X'$. Combining with (3) of Proposition \ref{alpha'} we obtain
$$\widetilde{\alpha}_{i,\lambda,t}(gx)\leq \alpha_{i,\lambda}'(gx)= \alpha_{i,\lambda}'(x)\leq (\eps_{\lambda,t_\lambda}^{d-1}\kappa_i(x)^{-2})^{-1}\widetilde{\alpha}_{i,\lambda,t}(x),$$
$$(\eps_{\lambda,t_\lambda}^{d-1}\kappa_i(x)^{-2})^{-1}\widetilde{\alpha}_{i,\lambda,t}(gx)\geq \alpha_{i,\lambda}'(gx)= \alpha_{i,\lambda}'(x)\geq\widetilde{\alpha}_{i,\lambda,t}(x) $$
for any $g\in U_i^{\perp}$ and $x\in X$.
\end{proof}

We also record the following property from the fact that $\widetilde{\alpha}_{i,\lambda}$'s are proper.

\begin{lem}\label{compactness}
    For each $0<\lambda<1$ and $\mathsf{h}\ge 1$ there exists $D_{\lambda}>\mathsf{h}$ such that the following holds. For any $1\leq i\leq d-1$, $g\in B^U(\operatorname{id},1)$, and $x\in X'$ with $\widetilde{\alpha}_{i,\lambda}(x)\leq\mathsf{h}$,
    $\widetilde{\alpha}_{i,\lambda}(gx)\leq D$ holds.
\end{lem}
\begin{proof}
    This is a consequence of the fact that $\widetilde{\alpha}_{i,\lambda}$'s are proper.
\end{proof}

Using Lemma \ref{compactness}, for each $0<\lambda<1$ we may define a monotone increasing function $D_{\lambda}:[1,\infty)\to [1,\infty)$ such that $D_{\lambda}(\sh)>\sh$ satisfies the property of Lemma \ref{compactness}.

\subsection{Dynamical height functions for rank one actions}
 The aim of this subsection is to construct dynamical height functions describing the behavior of trajectories near the cusp under one-dimensional diagonal flows. 
 
 Let $0<\lambda<1$ be given, and let $\widetilde{\alpha}_{i,\lambda}$'s be as in \eqref{alphatildedefinition} for $1\leq i\leq d-1$. From now on the parameter $\lambda$ will be regarded as fixed until the end of the proof of Theorem \ref{dimupperbdd}, so we may drop the subscripts $\lambda$ from $\widetilde{\alpha}_{i,\lambda}$, $t_\lambda$, and $D_{\lambda}$, and instead write $\widetilde{\alpha}_i$, $t$, and $D$, respectively. For $x\in X$ and $\mathsf{h}\geq 1$ we denote by $J_{i,\mathsf{h}}(x)\subseteq\bN$ the set of $j\in\bN$ satisfying $\widetilde{\alpha}_i(a_{jt\be_i}x)\geq \mathsf{h}$. For $N\in\bN$, $0\leq \del\leq1$, $1\leq i\leq d-1$, and $\mathsf{h}\geq 1$ we define a proper lower semi-continuous function $\beta_{N,\del,i,\mathsf{h}}:X'\to[1,\infty)$ as follows:
\eq{\beta_{N,\del,i,\mathsf{h}}(x)=\begin{cases}\widetilde{\alpha}_i(a_{Nt\be_i}x) & \text{if $\#\big(J_{i,\mathsf{h}}(x)\cap\set{1,\ldots,N}\big)\ge(1-\del)N$,}\\0 & \text{otherwise}\end{cases}}

Loosely speaking, $x\in\Supp \beta_{N,\del,i,\sh}$ only if the trajectory $\set{a_{t\be_i}x,\cdots, a_{Nt\be_i}x}$ spends most of the time outside the compact set $K_{\sh}$. The following proposition can be viewed as the contraction hypothesis for the dynamical height functions. The idea of the proof is inspired by the approach used in \cite[Proof of Theorem 1.5]{KKLM17}.

\begin{prop}[Contraction for $\beta$]\label{betadecay}
For any $N\in\bN$, $0<\del<1$, $1\leq i\leq d-1$, $\mathsf{h}\ge \useconE{101}e^{t}$, and $x\in X'$,
\eqlabel{betacont}{\int_{-\frac{1}{2}}^{\frac{1}{2}}\beta_{N,\del,i,\mathsf{h}}(u_{i}(s)x)ds\leq (8\useconC{5})^{N}e^{-(\lambda^2-\useconC{6}\del) Nt}\widetilde{\alpha}_i(x)}
for some constant \newconC{6}$\useconC{6}>1$.
\end{prop}
\begin{proof}
Let $\Theta_{N,\del}$ be the family of subsets of $\set{1,\ldots,N}$ containing at least $(1-\del)N$ elements. For each $J\in\Theta_{N,\del}$ we denote by $I_J(x)$ the set of $s\in[-\frac{1}{2},\frac{1}{2}]$ such that $J_{i,\mathsf{h}}(u_i(s)x)\cap\set{1,\ldots,N}=J$. Then
\eq{\int_{-\frac{1}{2}}^{\frac{1}{2}}\beta_{N,\del,i,\mathsf{h}}(u_{i}(s)x)ds=\sum_{J\in\Theta_{N,\del}}\int_{I_J(x)}\widetilde{\alpha}_i(a_{Nt\be_i}u_{i}(s)x)ds}
by definition of $\beta_{N,\del,i,\mathsf{h}}$. For $1\leq n\leq N$ we also define 
$$I_{J,n}(x)=\set{s\in[-\tfrac{1}{2},\tfrac{1}{2}]: \widetilde{\alpha}_i(a_{jt\be_i}u_i(s)x)\ge \mathsf{h} \textrm{ for all } j\in J\cap\set{1,\ldots,n}}$$
and set $I_{J,0}(x)=[-\frac{1}{2},\frac{1}{2}]$. Clearly we have $$I_J(x)=I_{J,N}(x)\subseteq I_{J,N-1}(x)\subseteq\cdots\subseteq I_{J,0}(x)=[-\tfrac{1}{2},\tfrac{1}{2}].$$

We choose a maximal $\frac{1}{2}e^{-2nt}$-separated subset $\set{s_1,\ldots,s_M}$ of $I_{J,n}$. Then $I_{J,n}$ is covered by the intervals $[s_k-\frac{1}{2}e^{-2nt},s_k+\frac{1}{2}e^{-2nt}]$ for $1\leq k\leq M$. Note that
\eq{\widetilde{\alpha}_i(a_{nt\be_i}u_i(s_k+s')x)=\widetilde{\alpha}_i(u_i(e^{2nt}s')a_{nt\be_i}u_i(s_k)x)\ge \useconC{5}^{-1}\widetilde{\alpha}_i(a_{nt\be_i}u_i(s_k)x)}
for any $1\leq k\leq M$ and  $s'\in [-\frac{1}{2}e^{-2nt},\frac{1}{2}e^{-2nt}]$. We also observe that the intervals $[s_k-\frac{1}{2}e^{-2nt},s_k+\frac{1}{2}e^{-2nt}]$ for $1\leq k\leq M$ can overlap at most twice, as ${s_1,\ldots,s_M}$ are $\tfrac{1}{2}e^{-2nt}$-separated. Thus, we have
\eqlabel{skcover}{\displaystyle\sum_{k=1}^{M}\widetilde{\alpha}_i(a_{nt\be_i}u_{i}(s_k)x)\leq 2\useconC{5}\int_{I_{J,n}(x)}\widetilde{\alpha}_i(a_{nt\be_i}u_i(s)x)ds.}

Note that $\widetilde{\alpha}_i(a_{nt\be_i}u_i(s)x)\ge\mathsf{h}\ge \useconE{101}e^{t}$ holds for any $n\in J$. It follows from \eqref{alphatildecont} that for any $n\in J$
\eqlabel{Jest}{\begin{aligned}
    \int_{I_{J,n+1}(x)}\widetilde{\alpha}_i(a_{(n+1)t\be_i}u_{i}(s)x)ds&\leq \int_{I_{J,n}(x)}\widetilde{\alpha}_i(a_{(n+1)t\be_i}u_{i}(s)x)ds \\&\leq \displaystyle\sum_{k=1}^{M}\int_{-\frac{1}{2}}^{\frac{1}{2}}\widetilde{\alpha}_i(a_{(n+1)t\be_i}u_{i}(s_k+e^{-2nt}s')x)ds'ds\\
    &=\displaystyle\sum_{k=1}^{M}\int_{-\frac{1}{2}}^{\frac{1}{2}}\widetilde{\alpha}_i(a_{t\be_i}u_i(s')a_{nt\be_i}u_{i}(s_k)x)ds'\\
    &\leq 2e^{-\lambda^2 t}\displaystyle\sum_{k=1}^{M}\widetilde{\alpha}_i(a_{nt\be_i}u_{i}(s_k)x).
\end{aligned}}
Combining \eqref{skcover} and \eqref{Jest}, we get
\eqlabel{Jest'}{\int_{I_{J,n+1}(x)}\widetilde{\alpha}_i(a_{(n+1)t\be_i}u_{i}(s)x)ds\leq 4\useconC{5}e^{-\lambda^2 t}\int_{I_{J,n}(x)}\widetilde{\alpha}_i(a_{nt\be_i}u_i(s)x)ds.}
On the other hand, for any $n\notin J$ we have $\widetilde{\alpha}_i(a_{(n+1)t\be_i}u_{i}(s)x)\leq \useconC{5}^t\widetilde{\alpha}_i(a_{nt\be_i}u_{i}(s)x)$ by \eqref{alphatildeLipschitz}, hence
\eqlabel{trivest}{\int_{I_{J,n+1}(x)}\widetilde{\alpha}_i(a_{(n+1)t\be_i}u_{i}(s)x)ds\leq e^{t\log \useconC{5}}\int_{I_{J,n}(x)}\widetilde{\alpha}_i(a_{nt\be_i}u_{i}(s)x)ds.}

Applying \eqref{Jest'} and \eqref{trivest} inductively, we obtain
\eqlabel{indest}{\begin{aligned}
\int_{I_J(x)}\widetilde{\alpha}_i(a_{Nt\be_i}u_i(s)x)ds&\leq (4\useconC{5}e^{-\lambda^2 t})^{|J|}e^{(N-|J|)t\log \useconC{5}}\widetilde{\alpha}_i(x)\\
&\leq (4\useconC{5})^Ne^{-\{\lambda^2-(1+\log \useconC{5})\del\} Nt}\widetilde{\alpha}_i(x).
\end{aligned}}
Since the cardinality of $\Theta_{N,\del}$ is at most $2^N$, we finally get
$$\int_{-\frac{1}{2}}^{\frac{1}{2}}\beta_{N,\del,i,\mathsf{h}}(u_{i}(s)x)ds\leq (8\useconC{5})^{N}e^{-\{\lambda^2-\useconC{6}\del\} Nt}\widetilde{\alpha}_i(x)$$
for $\useconC{6}:=1+\log \useconC{5}$.
\end{proof}

\section{Dynamical height functions for higher rank actions}
The purpose of this section is to construct dynamical height functions for higher-rank actions and to prove properties associated with them, such as the contraction hypothesis.
\subsection{Construction of higher-rank dynamical height functions}
As in the previous subsection, we fix $0<\lambda<1$ and set $t=t_\lambda$. For $N\in\bN$ let us define
\eqlabel{Ddef}{\cD_N:=\set{(\tau_1,\ldots,\tau_{d-1})\in(t\bN)^{d-1}: \max_{1\leq j\leq d-1}\tau_j+\sum_{j=1}^{d-1}\tau_j\leq 2dNt}.}
Then $\cD_N$ is the intersection of $(t\bN)^{d-1}$ and the convex hull of the set
$$\set{0, dNt\be_1,\ldots,dNt\be_{d-1}, 2Nt(\be_1+\cdots+\be_{d-1}))}.$$ 

\begin{figure}
\begin{tikzpicture}

\draw[->] (-0.5,0) -- (5,0) node[anchor=west] {\tiny{$\bR$}};
\draw[->] (0,-0.5) -- (0,5) node[anchor=south] {\tiny{$\bR$}};

\filldraw[draw=red, fill=blue!20] (0,0) -- (4.5,0) -- (3,3) -- (0,4.5) -- cycle;
\fill (0.5,0) circle[radius=2pt] node[anchor=north] {\tiny{$(t,0)$}};
\fill (1,0) circle[radius=2pt];
\fill (1.5,0) circle[radius=2pt];
\fill (3.5,0)  circle[radius=2pt];
\fill (4,0)  circle[radius=2pt];
\fill (4.5,0.5)  circle[radius=2pt];
\fill (4.5,1)  circle[radius=2pt];
\fill (4,0)  circle[radius=2pt];
\fill (4.5,0)  circle[radius=2pt] node[anchor=north] {\tiny{$(3Nt,0)$}};
\fill (0,0.5)  circle[radius=2pt] node[anchor=east] {\tiny{$(0,t)$}};
\fill (0,1)  circle[radius=2pt];
\fill (0,1.5)  circle[radius=2pt];
\fill (0,3.5)  circle[radius=2pt];
\fill (0,4)  circle[radius=2pt];
\fill (0.5,4.5)  circle[radius=2pt];
\fill (1,4.5)  circle[radius=2pt];
\fill (0,4.5)  circle[radius=2pt] node[anchor=east] {\tiny{$(0,3Nt)$}};
\fill (3.5,2.5)  circle[radius=2pt];
\fill (2.5,3.5)  circle[radius=2pt];
\fill (3,2.5)  circle[radius=2pt];
\fill (3.5,2)  circle[radius=2pt];
\fill (3,3.5)  circle[radius=2pt];
\fill[gray!50] (3.5,3)  circle[radius=2pt];
\fill (3.5,3.5)  circle[radius=2pt];
\fill[blue] (3,3)  circle[radius=2pt] node[anchor=west] {\tiny{$(2Nt,2Nt)$}};
\fill[blue] (2.5,2.5)  circle[radius=2pt];
\fill[blue] (2.5,3)  circle[radius=2pt];
\fill[blue] (2.0,3)  circle[radius=2pt];
\fill[blue] (2.0,3.5)  circle[radius=2pt];
\fill[blue] (3,2.5)  circle[radius=2pt];
\fill[blue] (3.5,2.0)  circle[radius=2pt];
\fill[blue] (3,2.0)  circle[radius=2pt];
\fill[blue] (0.5,0.5)  circle[radius=2pt];
\fill[blue] (1,0.5)  circle[radius=2pt];
\fill[blue] (0.5,1)  circle[radius=2pt];
\fill[blue] (4,0.5)  circle[radius=2pt];
\fill[blue] (3.5,0.5)  circle[radius=2pt];
\fill[blue] (0.5,4)  circle[radius=2pt];
\fill[blue] (0.5,3.5)  circle[radius=2pt];
\fill[blue] (4,1)  circle[radius=2pt];
\fill[blue] (3.5,1)  circle[radius=2pt];
\fill[blue] (1,4)  circle[radius=2pt];
\fill[blue] (1,3.5)  circle[radius=2pt];
\draw (2.5,0) node[anchor=north] {$\cdots$};
\draw (0,2.5) node[anchor=east] {$\vdots$};
\draw (2.5,1) node[anchor=north] {$\cdots$};
\draw (1,2.5) node[anchor=east] {$\vdots$};
\draw (2.0,5.0) node[anchor=north] {$\vdots$};
\draw (5.0,2.0) node[anchor=east] {$\cdots$};
\draw (1.7,2.2) node[anchor=north] {$\iddots$};
\draw (4.2,4.7) node[anchor=north] {$\iddots$};
\draw[very thin,dashed] (0,4.5) -- (3,3);
\draw[very thin,dashed] (4.5,0) -- (3,3);

\end{tikzpicture}
\caption{$\cD_N$ for $d=3$}
\label{atx}
\end{figure}

Note that $a_{\btau}u_i(s)a_{-\btau}=u_i\big(e^{\tau_i+\sum_{j=1}^{d-1}\tau_j}s\big)$ for any $1\leq i\leq d-1$, $s\in\bR$, and $\btau=(\tau_1,\ldots,\tau_{d-1})$. It follows that
\eqlabel{Dsize}{a_{\btau}B^U(\operatorname{id},r)a_{-\btau}\subseteq B^U(\operatorname{id},e^{2dNt}r)=g_{2Nt}B^U(\operatorname{id},r)g_{-2Nt}}
for any $\btau\in\cD_N$ and $r>0$.  

For $0<\del<1$ and $1\leq i\leq d-1$ we also define
\eqlabel{Edef}{\begin{aligned}
&\cE_{N,\del,i}:=\set{\btau\in\cD_N: \|\btau-dNt\be_i\|\leq \del Nt},\\
\widehat{\cE}_{N,\del,i}&:=\left\{\btau=(\tau_1,\ldots,\tau_{d-1})\in\cD_N:\left(\tfrac{d}{2}-\del\right)Nt< \tau_i \leq\left(d-\tfrac{\del}{2}\right)Nt, \|\btau-\tau_i\be_i\|\leq\tfrac{\del}{d} Nt\right\}.
\end{aligned}}
Clearly there exists \newconC{7}$\useconC{7}>1$ such that
\eqlabel{Ecount}{\begin{aligned}
\useconC{7}^{-1}\del^{d}\#\cD_N&<\#\cE_{N,\del,i}<\useconC{7}\del^{d}\#\cD_N,\\
\useconC{7}^{-1}\del^{d-1}\#\cD_N&<\#\widehat{\cE}_{N,\del,i}<\useconC{7}\del^{d-1}\#\cD_N
\end{aligned}}
holds for any $N,i$, and $\del$.

For $N\in\bN$, $\mathsf{h}>\max\set{\useconE{101}e^{t},1}$, and $x\in X$ let us define
\eqlabel{Omdef}{\Om(N,\sh,x):=\set{\btau\in\cD_N: \min_{1\leq i\leq d-1}\widetilde{\alpha}_i(a_{\btau}x)\ge \mathsf{h}}.}
We also define $\psi_{N,\del,\sh}:X\to[0,\infty]$ by
\eqlabel{psidef}{\psi_{N,\del,\sh}(x):=\begin{cases}
\displaystyle\prod_{i=1}^{d-1}\displaystyle\min_{\btau\in\cE_{N,\del,i}\cap\Om(N,\sh,x)}\widetilde{\alpha}_i(a_{\btau}x) & \text{if $\frac{\#\Om(N,\sh,x)}{\#\cD_N}\ge(1-\frac{\del^d}{4\useconC{7}})$,}\\
0 & \text{otherwise.}
\end{cases}}
Loosely speaking, $x\in\Supp \psi_{N,\del,\sh}$ only if the truncated orbit $\set{a_{\btau}x: \btau\in\cD_N}$ spends most of time outside the compact set $K_{\sh}$.

\begin{lem}\label{psibetabdd}
For each $x\in X$ with $\psi_{2N,\del,\sh}(x)>0$, there exist $\btau_1,\ldots,\btau_{d-1}$ such that 
\begin{enumerate}
    \item $\btau_i\in\cE_{N,2\del,i}\cap \Om(N,\sh,x)$ for $1\leq i\leq d-1$,
    \item $\psi_{2N,\del,\sh}(x)\leq \displaystyle\prod_{i=1}^{d-1}\beta_{dN,\del,i,\sh}(a_{\btau_i}x).$
\end{enumerate}
\end{lem}
\begin{proof}
By \eqref{psidef}, $\psi_{2N,t,\del,\sh}(x)>0$ implies $\#\Om(2N,\sh,x)\ge (1-\frac{\del^d}{4\useconC{7}})\#\cD_{2N}$. It follows from \eqref{Ecount} that
$$\#(\widehat{\cE}_{2N,\del,i}\setminus\Om(2N,\sh,x))\leq \frac{\del^d}{4\useconC{7}}\#\cD_{2N}<\frac{\del}{4} \#\widehat{\cE}_{2N,\del,i}$$
for any $1\leq i\leq d-1$. Note that the set $\widehat{\cE}_{2N,\del,i}$ can be expressed by
\eqlabel{Ehatdec}{\begin{aligned}\widehat{\cE}_{2N,\del,i}&=\left\{\btau+kt\be_i: \btau=(\tau_1,\ldots,\tau_{d-1})\in\cD_N,\tau_i=\lfloor (d-2\del)N\rfloor t,\right.\\ &\qquad\qquad\qquad\qquad\left. \|\btau-\tau_i\be_i\|\leq \frac{2\del}{d} Nt, 1\leq k\leq \lfloor (d+\del)N\rfloor\right\}.\end{aligned}}
It follows that for each $i$ we can find $\btau_i'\in \cE_{N,2\del,i}$ such that the $i$-th coordinate of $\btau_i'$ is $\lfloor (d-2\del)N\rfloor t$ and
\eqlabel{tauicount}{\#(\set{\btau_i'+kt\be_i: 1\leq k\leq \lfloor (d+\del)N\rfloor}\setminus\Om(2N,\sh,x))<\tfrac{1}{4}\del N.} 

Let us denote
$$\cK_i:=\set{1\leq k\leq \lfloor\del N\rfloor: \btau_i'+kt\be_i\in \Om(2N,\sh,x)},$$
$$\cK_i':=\set{1\leq k\leq \lfloor\del N\rfloor: \btau_i'+dNt\be_i+kt\be_i\in \Om(2N,\sh,x)}.$$
Then \eqref{tauicount} implies that
$$\#(\set{1,\ldots,\lfloor\del N\rfloor}\setminus \cK_i)<\tfrac{1}{4}\del N,\quad \#(\set{1,\ldots,\lfloor\del N\rfloor}\setminus \cK_i')<\tfrac{1}{4}\del N,$$
hence for each $i$ there exists some $\btau_i\in\set{\btau_i',\btau_i'+t\be_i,\ldots,\btau_i'+\lfloor\del dN\rfloor t\be_i}$ such that $\btau_i\in \cK_i\cap\cK_i'$. Let $\btau_i=\btau_i'+k_it\be_i$, where $1\leq k_i\leq \lfloor\del dN\rfloor$.
We have $\btau_i\in \cE_{N,2\del,i}$ since the $i$-th coordinate of $\btau_i$ is between $(d-2\del)N$ and $(d-\del)N$, and $\btau_i+dNt\be_i\in \cE_{2N,\del,i}$ since the $i$-th coordinate of $\btau_i+dNt$ is between $(2d-2\del)N$ and $(2d-\del)N$. Thus, we obtain
\eqlabel{eq:property31}{\btau_i\in\cE_{N,2\del,i}\cap\Om(2N,\sh,x)=\cE_{N,2\del,i}\cap\Om(N,\sh,x),}
hence (1) is proved.

We now show (2). Note that \eqref{eq:property31} implies
$\btau_i+dNt\be_i\in\cE_{2N,\del,i}\cap\Om(2N,\sh,x)$, hence
\eqlabel{alpest1}{\min_{\btau\in\cE_{2N,\del,i}\cap\Om(2N,\sh,x)}\widetilde{\alpha}_i(a_{\btau}x)\leq\widetilde{\alpha}_i(a_{\btau_i+dNt\be_i}x).}
On the other hand, \eqref{tauicount} implies that
$$\#\big(\set{1\leq k\leq \lfloor (d+\del)N\rfloor: \widetilde{\alpha}_i(a_{kt\be_i}a_{\btau_i'}x)<\sh}\big)<\tfrac{1}{4}\del N.$$
It follows that
$$\#\big(\set{1\leq k\leq dN: \widetilde{\alpha}_i(a_{kt\be_i}a_{\btau_i}x)<\sh}\big)<\del N,$$
hence $\#(J_i(a_{\btau_i}x)\cap\set{1,\ldots,dN})\ge (d-\del)N$. Thus, we obtain
\eqlabel{alpest2}{\beta_{dN,\del,i,\sh}(a_{\btau_i}x)=\widetilde{\alpha}_i(a_{dNt\be_i}a_{\btau_i}x)=\widetilde{\alpha}_i(a_{\btau_i+dNt\be_i}x)}
from the definition of $\beta_{dN,\del,i,\sh}$. Combining \eqref{alpest1} and \eqref{alpest2}, we get
\eqlabel{alpest3}{\displaystyle\min_{\btau\in\cE_{2N,\del,i}\cap\Om(2N,\sh,x)}\widetilde{\alpha}_i(a_{\btau}x)\leq \beta_{dN,\del,i,\sh}(a_{\btau_i}x)}
for all $1\leq i\leq d-1$. Therefore,
$$\psi_{2N,\del,\sh}(x)=\displaystyle\prod_{i=1}^{d-1}\displaystyle\min_{\btau\in\cE_{2N,\del,i}\cap\Om(2N,\sh,x)}\widetilde{\alpha}_i(a_{\btau}x)\leq \prod_{i=1}^{d-1}\beta_{dN,\del,i,\sh}(a_{\btau_i}x).$$
\end{proof}

\subsection{Approximation of height functions}
In this subsection, we approximate the dynamical height functions on horospherical orbits by characteristic functions on the union of certain boxes. We begin with approximating $\psi_{N,\del,D(\sh)}$ on horospherical orbits.
\begin{lem}[Approximation of $\psi$]\label{psisupp}
For $N\in\bN$, $0<\del<1$, $\sh>\useconE{101}e^{t}$, and $x\in X'$, let us define $F:[-\frac{1}{2},\frac{1}{2}]^{d-1}\to\bR$ by $F(\bxi)=\psi_{N,\del,D(\sh)}(u(\bxi)x)$, where $D(\sh)$ is as in Lemma \ref{compactness}, and let $\set{\bxi_1,\ldots,\bxi_M}\subset[-\frac{1}{2},\frac{1}{2}]^{d-1}$ be a maximal $e^{-(2d+4\del)Nt}$-separated set of $\Supp F$. Then we have
\begin{enumerate}
    \item $0\leq F(\bxi)\leq \useconE{102}^{d-1}e^{d(d-1)\del Nt}\kappa(x)^2\displaystyle\sum_{k=1}^{M} F(\bxi_k)\mathds{1}_{B(\bxi_k,e^{-(2d+4\del)Nt})}(\bxi)$
    for any $\bxi\in[-\frac{1}{2},\frac{1}{2}]^{d-1}$,
    \item 
        $\displaystyle\sum_{k=1}^{M} F(\bxi_k)\leq 
    \useconE{102}^{d-1}e^{\{2d(d-1)+5d^2\del\} Nt}\kappa(x)^2\int_{[-\frac{1}{2},\frac{1}{2}]^{d-1}}\psi_{N,\del,\sh}(u(\bxi)x)d\bxi,$\\
    where \newconE{102}$\useconE{102}:=\useconC{5}\eps_{\lambda,t}^{-(d-1)}$.
\end{enumerate}
\end{lem}
\begin{proof}
For any $\bxi\in\Supp F$ there exists $1\leq k\leq M$ such that $\bxi\in B(\bxi_k,e^{-(2d+4\del)Nt})$. By \eqref{Dsize}, we have $a_{\btau}u(\bxi-\bxi_k)a_{-\btau}\in B^U(\operatorname{id},1)$ for any $1\leq i\leq d-1$ and $\btau\in\cD_N$. It follows from Lemma \ref{compactness} that
\eqlabel{perturbation}{\begin{aligned}&\Om(N,D(\sh),u(\bxi)x)\subseteq \Om(N,\sh,u(\bxi_k)x),\\&\Om(N,\sh,u(\bxi)x)\subseteq \Om(N,D(\sh),u(\bxi_k)x).\end{aligned}}
Moreover, using \eqref{alphatildeLipschitz} and \eqref{LipschitzOrthogonal} we get
\eqlabel{Lipschitz1}{\begin{aligned}\widetilde{\alpha}_i(a_{\btau}u(\bxi)x)&=\widetilde{\alpha}_i\big(\big(a_{\btau}u(\bxi-\bxi_k)a_{-\btau}\big)a_{\btau}u(\bxi_k)x\big)\\&\leq \useconC{5}\alpha_i\big(\big(a_{\btau}u(\pi_i^{\perp}(\bxi-\bxi_k))a_{-\btau}\big)a_{\btau}u(\bxi_k)x\big)\\&\leq \useconC{5}\eps_{\lambda,t}^{-(d-1)}\kappa_i(a_{\btau}u(\bxi_k)x)^{2}\widetilde{\alpha}_i(a_{\btau}u(\bxi_k)x)\\&=\useconE{102}\kappa_i(a_{\btau}x)^{2}\widetilde{\alpha}_i(a_{\btau}u(\bxi_k)x)\end{aligned}}
for all $1\leq i\leq d-1$. Similarly, we also have
\eqlabel{Lipschitz1'}{\widetilde{\alpha}_i(a_{\btau}u(\bxi_k)x)\geq \useconE{102}\kappa_i(x)^{2}\widetilde{\alpha}_i(a_{\btau}u(\bxi)x)}
for all $1\leq i\leq d-1$. Combining \eqref{perturbation} and \eqref{Lipschitz1},
\eq{\displaystyle\min_{\btau\in\cE_{N,\del,i}\cap\Om(N,D(\sh),u(\bxi)x)}\widetilde{\alpha}_i(a_{\btau}u(\bxi)x)\leq \useconE{102}\displaystyle\min_{\btau\in\cE_{N,\del,i}\cap\Om(N,\sh,u(\bxi_k)x)}\big(\kappa_i(a_{\btau} x)^{2}\widetilde{\alpha}_i(a_{\btau}u(\bxi_k)x)\big)}
holds for all $1\leq i\leq d-1$. Note that for any $\btau\in \cE_{N,\del,i}$
\eqlabel{kappaestimate}{\kappa_i(a_{\btau}x)\leq e^{d\del Nt}\kappa_i(a_{dNt\be_i}x)=e^{d\del Nt}\kappa_i(x)}
since $\|\btau-dNt\be_i\|\leq \del Nt$. It follows that
\eq{\displaystyle\min_{\btau\in\cE_{N,\del,i}\cap\Om(N,D(\sh),u(\bxi)x)}\widetilde{\alpha}_i(a_{\btau}u(\bxi)x)\leq \useconE{102}e^{2d\del Nt}\kappa_i(x)^{2}\displaystyle\min_{\btau\in\cE_{N,\del,i}\cap\Om(N,\sh,u(\bxi_k)x)}\widetilde{\alpha}_i(a_{\btau}u(\bxi_k)x).}
for all $1\leq i\leq d-1$. Thus, if $F(\bxi)>0$ then
\eqlabel{Fupperbound1}{\begin{aligned}F(\bxi)&=\psi_{N,\del,D(\sh)}\big(u(\bxi)x\big)=\displaystyle\prod_{i=1}^{d-1}\displaystyle\min_{\btau\in\cE_{N,\del,i}\cap\Om(N,D(\sh),u(\bxi)x)}\widetilde{\alpha}_i(a_{\btau}u(\bxi)x)\\&\leq \useconE{102}^{d-1}e^{d(d-1)\del Nt}\kappa(x)^2\displaystyle\prod_{i=1}^{d-1}\displaystyle\min_{\btau\in\cE_{N,\del,i}\cap\Om(N,\sh,u(\bxi_k)x)}\widetilde{\alpha}_i(a_{\btau}u(\bxi_k)x)\\&\leq \useconE{102}^{d-1}e^{d(d-1)\del Nt}\kappa(x)^2\psi_{N,\del,\sh}\big(u(\bxi_k)x\big),\end{aligned}}
so (1) is proved.

We now verify (2). As in \eqref{Fupperbound1}, we obtain from \eqref{perturbation} and \eqref{Lipschitz1'} that \eq{F(\bxi_k)=\psi_{N,\del,D(\sh)}(u(\bxi_k)x)\leq \useconE{102}^{d-1}e^{d(d-1)\del Nt}\kappa(x)^2\psi_{N,\del,\sh}(u(\bxi)x)} for any $\bxi\in B(\bxi_k,e^{-(2d+4\del)Nt})$. It follows that
\eq{\begin{aligned}
&(2e^{-(2d+4\del)Nt})^{d-1}\displaystyle\sum_{k=1}^{M} F(\bxi_k)=\displaystyle\sum_{k=1}^{M}\int_{B(\bxi_k,e^{-(2d+4\del)Nt})}F(\bxi_k)d\bxi\\&\leq \useconE{102}^{d-1}e^{d(d-1)\del Nt}\kappa(x)^2\displaystyle\sum_{k=1}^{M}\int_{B(\bxi_k,e^{-(2d+4\del)Nt})}\psi_{N,\del,\sh}(u(\bxi)x)d\bxi\\
&\leq 2^{d-1}\useconE{102}^{d-1}e^{d(d-1)\del Nt}\kappa(x)^2\int_{[-\frac{1}{2},\frac{1}{2}]^{d-1}}\psi_{N,\del,\sh}(u(\bxi)x)d\bxi.
\end{aligned}}
In the last line, we use a fact derived from the maximality of $\bxi_i$'s: the sets $B(\bxi_k,e^{-(2d+4\del)Nt})$ can overlap at most $2^{d-1}$ times. This completes the proof. 
\end{proof}

We now approximate $\beta_{dN,\del,i,\sh}$ on certain expanding translates of horospherical orbits. For $N\in\bN$ and $1\leq i\leq d-1$, we denote by $\sB_{N,i}\subset\bR^{d-1}$ the box centered at $0$ with side-length $2e^{-2dNt}$ along the direction of $\be_i$, and side-length $2$ along the other $d-2$ directions.

\begin{lem}[Approximations of $\beta$]\label{betasupp}
Let $N\in\bN$, $0<\del<1$, $\sh>\useconE{101}e^{t}$, and $x\in X'$. Suppose that $\btau_i\in \cE_{N,2\del,i}$ for all $1\leq i\leq d-1$. For each $1\leq i\leq d-1$ let us define $f_i:[-\frac{1}{2},\frac{1}{2}]^{d-1}\to\bR$ by $f_i(\bxi)=\beta_{dN,\del,i,D(\sh)}(a_{\btau_i}u(e^{-(2d+4\del)Nt}\bxi)x)$. Then there exist $\bxi_{i,1},\ldots,\bxi_{i,M_i}\in[-\frac{1}{2},\frac{1}{2}]^{d-1}$ such that
\begin{enumerate}
    \item $0\leq f_i(\bxi)\leq \useconE{102}e^{2d\del Nt}\kappa_i(x)^2\displaystyle\sum_{k=1}^{M_i} f_i(\bxi_{i,k})\mathds{1}_{\bxi_{i,k}+\sB_{N,i}}(\bxi)$ for any $\bxi\in[-\frac{1}{2},\frac{1}{2}]^{d-1}$,
    \item $\displaystyle\sum_{k=1}^{M_i} f_i(\bxi_{i,k})\leq \useconE{102}e^{2d(1+\del)Nt}\kappa_i(x)^2\int_{[-\frac{1}{2},\frac{1}{2}]^{d-1}}\beta_{dN,\del,i,\sh}(a_{\btau_i}u(e^{-(2d+4\del)Nt}\bxi)x)d\bxi.$
\end{enumerate}
\end{lem}

\begin{proof}
Let $\set{\bxi_{i,1},\ldots,\bxi_{i,M_i}}$ be a maximal $\sB_{N,i}$-set of $\Supp f_i$, i.e. the set $\set{\bxi_{i,k}}_{k=1}^{M_i}$ is a maximal set such that the boxes $\bxi_{i,k}+\sB_{N,i}$'s are disjoint. Then for any $\bxi\in\Supp f_i$, there exists $1\leq k\leq M_i$ such that $\bxi\in \bxi_{i,k}+\sB_{N,i}$. Then for any $1\leq j\leq dN$ $$a_{jt\be_i}a_{\btau_i}u\big(e^{-(2d+4\del)Nt}(\bxi-\bxi_{i,k})\big)a_{-\btau_i}a_{-jt\be_i}\in a_{\btau_i}a_{dNt\be_i}u(e^{-(2d+4\del)Nt}\sB_{N,i})a_{-dNt\be_i}a_{-\btau_i}.$$

Note that $u(e^{-(2d+4\del)Nt}\sB_{N,i})$ is the box centered at $\operatorname{id}\in U$ with side-length $e^{-(4d+4\del)Nt}$ along the direction of $u(\be_i)$ and side-length $e^{-(2d+4\del)\del Nt}$ along the other $d-2$ directions. It follows that the box $a_{dNt\be_i}u(e^{-(2d+4\del)Nt}\sB_{N,i})a_{-dNt\be_i}$ has side-length $e^{-(2d+4\del)Nt}$ along the direction of $u(\be_i)$ and side-length $e^{-(d+4\del)Nt}$ along the other directions. Hence, the box $a_{\btau_i}a_{dNt\be_i}u(e^{-(2d+4\del)Nt}\sB_{N,i})a_{-dNt\be_i}a_{-\btau_i}$ is contained in $B^U(\operatorname{id},1)$.

It follows from Lemma \ref{compactness} that
\eqlabel{perturbation'}{\begin{aligned}&J_{i,D(\sh)}(u(\bxi)x)\subseteq J_{i,\sh}(u(\bxi_k)x),\\&J_{i,D(\sh)}(u(\bxi_k)x)\subseteq J_{i,\sh}(u(\bxi)x).\end{aligned}}
Moreover, using \eqref{alphatildeLipschitz} and \eqref{LipschitzOrthogonal} we obtain
\eqlabel{Lipschitz2}{\widetilde{\alpha}_i(a_{dNt\be_i}a_{\btau_i}u(e^{-(2d+4\del)Nt}\bxi)x)\leq \useconE{102}\kappa_i(a_{\btau_i}x)^{2}\widetilde{\alpha}_i(a_{dNt\be_i}a_{\btau_i}u(e^{-(2d+4\del)Nt}\bxi_{i,k})x),}
\eqlabel{Lipschitz2'}{\widetilde{\alpha}_i(a_{dNt\be_i}a_{\btau_i}u(e^{-(2d+4\del)Nt}\bxi_{i,k})x)\leq \useconE{102}\kappa_i(a_{\btau_i}x)^{2}\widetilde{\alpha}_i(a_{dNt\be_i}a_{\btau_i}u(e^{-(2d+4\del)Nt}\bxi)x)}
for any $\bxi\in\bxi_{i,k}+\sB_{N,i}$. Recall \eqref{kappaestimate} that $\kappa_i(a_{\btau}x)\leq e^{d\del Nt}\kappa_i(x)$ for any $\btau\in \cE_{N,\del,i}$. Combining this with \eqref{perturbation'} and \eqref{Lipschitz2} we get
\eqlabel{fiupperbound}{\begin{aligned}
    f_i(\bxi)&=\beta_{dN,\del,i,D(\sh)}(a_{\btau_i}u(e^{-(2d+4\del)Nt}\bxi)x)\\&=\widetilde{\alpha}_i(a_{dNt\be_i}a_{\btau_i}u(e^{-(2d+4\del)Nt}\bxi)x)\\&=\useconE{102}\kappa_i(a_{\btau_i}x)^{2}\widetilde{\alpha}_i(a_{dNt\be_i}a_{\btau_i}u(e^{-(2d+4\del)Nt}\bxi_{i,k})x)\\&\leq \useconE{102}e^{2d\del Nt}\kappa_i(x)^2\beta_{dN,\del,i,\sh}(a_{\btau_i}u(e^{-(2d+4\del)Nt}\bxi_{i,k})x),
\end{aligned}}
so (1) is proved.

We now verify (2). As in \eqref{fiupperbound}, we obtain from \eqref{perturbation'} and \eqref{Lipschitz2'} that 
\eq{\begin{aligned}f_i(\bxi_{i,k})&=\beta_{dN,\del,i,\sh}(a_{\btau_i}u(e^{-(2d+4\del)Nt}\bxi_{i,k})x)\\&\leq \useconE{102}e^{2d\del Nt}\kappa_i(x)^2\beta_{dN,\del,i,\sh}(a_{\btau_i}u(e^{-(2d+4\del)Nt}\bxi)x)\end{aligned}}
for any $\bxi\in\bxi_{i,k}+\sB_{N,i}$. Note that the volume of $\sB_{N,i}$ is $2^{d-1}e^{-2dNt}$. It follows that
\eq{\begin{aligned}
2^{d-1}e^{-2dNt}&\displaystyle\sum_{k=1}^{M_i} f_i(\bxi_{i,k})=\displaystyle\sum_{k=1}^{M_i}\int_{\bxi_{i,k}+\sB_{N,i}}f_i(\bxi_{i,k})d\bxi\\&\leq \useconE{102}e^{2d\del Nt}\kappa_i(x)^2\displaystyle\sum_{k=1}^{M_i}\int_{\bxi_{i,k}+\sB_{N,i}}\beta_{dN,\del,i,\sh}(a_{\btau_i}u(e^{-(2d+4\del)Nt}\bxi)x)d\bxi\\
&\leq 2^{d-1}\useconE{102}e^{2d\del Nt}\kappa_i(x)^2\int_{[-\frac{1}{2},\frac{1}{2}]^{d-1}}\beta_{dN,\del,i,\sh}(a_{\btau_i}u(e^{-(2d+4\del)Nt}\bxi)x)d\bxi.\end{aligned}}
In the last line, we use a fact derived from the maximality of the set $\set{\bxi_{i,1},\ldots,\bxi_{i,M_i}}$: the sets $\set{\bxi_{i,k}+\sB_{N,i}}_{k=1}^{M_i}$ can overlap at most $2^{d-1}$ times. This completes the proof. 
\end{proof}

\subsection{Contraction for higher rank actions}
We are now ready to prove the following inductive contraction inequality for higher-rank dynamical height functions.
\begin{prop}[Inductive contraction for $\psi$]\label{multicont}
For any $N\in\bN$, $0<\del<2^{-(d-1)}$, $\sh>\useconE{101}e^{t}$, and $x\in X'$, we have
\eq{\begin{aligned}\int_{[-\frac{1}{2},\frac{1}{2}]^{d-1}}\psi_{2N,\del,D(\sh)}(u(\bxi)x)d\bxi\leq \useconE{104}&N^{d^2}(8\useconC{5})^{d(d-1)N}e^{-\{d(d-1)\lambda^2-\useconC{10}\del\}Nt}\\&\times\kappa(x)^8\int_{[-\frac{1}{2},\frac{1}{2}]^{d-1}}\psi_{N,2^{d-1}\del,\sh}(u(\bxi)x)d\bxi,\end{aligned}}
for some constants \newconE{104}\newconC{10}$\useconE{104},\useconC{10}>0$.
\end{prop}
\begin{proof}
Since $\cD_N\setminus\Om(N,\sh,x)\subseteq\cD_{2N}\setminus\Om(2N,\sh,x)$, we have
\eq{\begin{aligned}1-\frac{\#\Om(N,\sh,x)}{\#\cD_N}&\leq \frac{\#\big(\cD_{2N}\setminus\Om(2N,\sh,x)\big)}{\#\cD_N}\\&=2^{d-1}\frac{\#\big(\cD_{2N}\setminus\Om(2N,\sh,x)\big)}{\#\cD_{2N}}=2^{d-1}\left(1-\frac{\#\Om(2N,\sh,x)}{\#\cD_{2N}}\right),\end{aligned}}
for any $x\in X'$, hence $\Supp \psi_{2N,\del,D(\sh)}\subseteq\Supp\psi_{N,2^{d-1}\del,D(\sh)}$. 

Let $F:[-\frac{1}{2},\frac{1}{2}]^{d-1}\to\bR$ be the function defined by $F(\bxi)=\psi_{N,2^{d-1}\del,D(\sh)}(u(\bxi)x)$, and $\set{\bxi_1,\ldots,\bxi_M}$ be a maximal $e^{-(2d+4\del)Nt}$-separated set of $\Supp F$ as in Lemma~\ref{psisupp}. Note that
\eqlabel{SuppFcov}{\Supp F\subseteq\bigcup_{k=1}^{M}B(\bxi_k,e^{-(2d+4\del)Nt}).}
By Lemma \ref{psibetabdd}, for each $\bxi\in\Supp F$ we can find some $$\big(\btau_1,\cdots,\btau_{d-1}\big)\in \cE_{N,2\del,1}\times\cdots\times\cE_{N,2\del,d-1}$$ such that $\btau_i\in \Om(N,\eps,\sh,u(\bxi)x)$ for all $i$ and $$\psi_{2N,\del,\sh}(u(\bxi)x)\leq \prod_{i=1}^{d-1}\beta_{dN,\del,i,\sh}(a_{\btau_i}u(\bxi)x).$$ Hence, we have
\eqlabel{psibetabdd1}{\begin{aligned} &\int_{[-\frac{1}{2},\frac{1}{2}]^{d-1}}\psi_{2N,\del,D(\sh)}(u(\bxi)x)d\bxi\leq\displaystyle\sum_{k=1}^{M}\int_{B(\bxi_k,e^{-(2d+4\del)Nt})}\psi_{2N,\del,D(\sh)}(u(\bxi)x)d\bxi\\&\leq\displaystyle\sum_{k=1}^{M}\displaystyle\sum_{\btau_1\in\cE_{N,2\del,1}}\cdots\displaystyle\sum_{\btau_{d-1}\in\cE_{N,2\del,d-1}}\int_{B(\bxi_k,e^{-(2d+4\del)Nt})}\prod_{i=1}^{d-1}\beta_{dN,\del,i,D(\sh)}(a_{\btau_i}u(\bxi)x)d\bxi\\
&=\displaystyle\sum_{k=1}^{M}\displaystyle\sum_{\btau_1\in\cE_{N,2\del,1}}\cdots\displaystyle\sum_{\btau_{d-1}\in\cE_{N,2\del,d-1}}e^{-(2d+4\del)(d-1)Nt}\\&\qquad\qquad\qquad\times\int_{[-\frac{1}{2},\frac{1}{2}]^{d-1}}\prod_{i=1}^{d-1}\beta_{dN,\del,i,D(\sh)}\left(a_{\btau_i}u(e^{-(2d+4\del)Nt}\bxi)u(\bxi_k)x\right)d\bxi.\end{aligned}}

We now estimate the integral of the last line for each fixed $1\leq k\leq M$ and $(\btau_1,\ldots,\btau_{d-1})\in \cE_{N,2\del,1}\times\cdots\times\cE_{N,2\del,d-1}$. For each $1\leq i\leq d-1$ let us define $f_i:[-\frac{1}{2},\frac{1}{2}]^{d-1}\to\bR$ by $f_i(\bxi)=\beta_{dN,\del,i,\sh}(a_{\btau_i}u(e^{-(2d+4\del)Nt}\bxi)u(\bxi_k)x)$ and let $\bxi_{i,1},\ldots,\bxi_{i,M_i}\in[-\frac{1}{2},\frac{1}{2}]^{d-1}$ be as in Lemma \ref{betasupp}. Recall that we have 
$$0\leq f_i\leq \useconE{102}e^{2d\del Nt}\kappa_i(x)^2\displaystyle\sum_{k=1}^{M_i} f_{i}(\bxi_{i,k})\mathds{1}_{\bxi_{k,i}+\sB_{N,i}}$$
by (1) of Lemma \ref{betasupp}. It follows that
\eqlabel{betaprod1}{\begin{aligned} &\int_{[-\frac{1}{2},\frac{1}{2}]^{d-1}}\prod_{i=1}^{d-1}\beta_{dN,\del,i,\sh}\left(a_{\btau_i}u(e^{-(2d+4\del)Nt}\bxi)u(\bxi_k)x\right)d\bxi=\int_{[-\frac{1}{2},\frac{1}{2}]^{d-1}}\prod_{i=1}^{d-1}f_i(\bxi)d\bxi\\&\leq \useconE{102}^{d-1}e^{2d(d-1)\del Nt}\kappa(x)^2\displaystyle\sum_{k_1=1}^{M_1}\cdots\displaystyle\sum_{k_{d-1}=1}^{M_{d-1}}\left(\prod_{i=1}^{d-1}f_{i}(\bxi_{i,k_i})\right)\int_{[-\frac{1}{2},\frac{1}{2}]^{d-1}}\prod_{i=1}^{d-1}\mathds{1}_{\bxi_{i,k_i}+\sB_{N,i}}(\bxi)d\bxi.\end{aligned}}
We observe that for any $k_1,\ldots,k_{d-1}$, $\displaystyle\bigcap_{i=1}^{d-1}\left(\bxi_{i,k_i}+\sB_{N,i}\right)$ is always contained in a box with side-length $2e^{-2dNt}$ since each $\sB_{N,i}$ has side-length $2e^{-2dNt}$ along the direction of $\be_i$. We also have 
$$\displaystyle\sum_{k_i=1}^{M_i} f_i(\bxi_{i,k_i})\leq \useconE{102}e^{2d(1+\del)Nt}\kappa_i(x)^{2}\int_{[-\frac{1}{2},\frac{1}{2}]^{d-1}}\beta_{dN,\del,i,\sh}\left(a_{\btau_i}u(e^{-(2d+4\del)Nt}\bxi)u(\bxi_k)x\right)d\bxi$$
for all $1\leq i\leq d-1$ by (2) of Lemma \ref{betasupp}.
It follows that \eqref{betaprod1} is bounded by
\eqlabel{betaprod2}{\begin{aligned} &\leq \useconE{102}^{d-1}e^{2d(d-1)\del Nt}\kappa(x)^2(2e^{-2dNt})^{d-1}\displaystyle\sum_{k_1=1}^{M_1}\cdots\displaystyle\sum_{k_{d-1}=1}^{M_{d-1}}\prod_{i=1}^{d-1}f_i(\bxi_{i,k_i})\\ 
&\leq 2^{d-1}\useconE{102}^{2(d-1)}e^{4d(d-1)\del Nt}\kappa(x)^4\prod_{i=1}^{d-1}\int_{[-\frac{1}{2},\frac{1}{2}]^{d-1}}\beta_{dN,\del,i,\sh}\left(a_{\btau_i}u(e^{-(2d+4\del)Nt}\bxi)u(\bxi_k)x\right)d\bxi. \end{aligned}}

Recall that by Proposition \ref{betadecay} we have
\eq{\int_{-\frac{1}{2}}^{\frac{1}{2}}\beta_{dN,\del,i,\mathsf{h}}(u_{i}(s)x)ds\leq (8\useconC{5})^{dN}e^{-(\lambda^2-\useconC{6}\del) dNt}\alpha_i(x)}
for any $1\leq i\leq d-1$ and $x\in X'$. It follows that
\eq{\begin{aligned}\int_{[-\frac{1}{2},\frac{1}{2}]^{d-1}}&\beta_{dN,\del,i,\sh}\left(a_{\btau_i}u(e^{-(2d+4\del)Nt}\bxi)u(\bxi_k)x\right)d\bxi\\&=\int_{[-\frac{1}{2},\frac{1}{2}]^{d-1}}\int_{-\frac{1}{2}}^{\frac{1}{2}}\beta_{dN,\del,i,\sh}\left(a_{\btau_i}u_i(e^{-(2d+4\del)Nt}s)u(e^{-(2d+4\del)Nt}\pi_i^{\perp}(\bxi))u(\bxi_k)x\right)dsd\bxi \\&\leq e^{8\del Nt}\int_{[-\frac{1}{2},\frac{1}{2}]^{d-1}}\int_{-\frac{1}{2}}^{\frac{1}{2}}\beta_{dN,\del,i,\sh}\left(u_i(s)a_{\btau_i}u(e^{-(2d+4\del)Nt}\pi_i^{\perp}(\bxi))u(\bxi_k)x\right)dsd\bxi\\ &\leq (8\useconC{5})^{dN}e^{-\{d\lambda^2-(d\useconC{6}+8)\del\} Nt}\int_{[-\frac{1}{2},\frac{1}{2}]^{d-1}}\widetilde{\alpha}_i(a_{\btau_i}u(e^{-(2d+4\del)Nt}\pi_i^{\perp}(\bxi))u(\bxi_k)x)d\bxi.\end{aligned}}
for any $1\leq i\leq d-1$. By \eqref{LipschitzOrthogonal} one has
\eq{\widetilde{\alpha}_i(a_{\btau_i}u(e^{-(2d+4\del)Nt}\pi_i^{\perp}(\bxi))u(\bxi_k)x)\leq \eps_{\lambda,t}^{-(d-1)}\kappa_i(x)^{2}\widetilde{\alpha}_i(a_{\btau_i}u(\bxi_k)x),}
hence for any $1\leq i\leq d-1$
\eqlabel{ithicken}{\begin{aligned}\int_{[-\frac{1}{2},\frac{1}{2}]^{d-1}}&\beta_{dN,\del,i,\sh}\left(a_{\btau_i}u(e^{-(2d+4\del)Nt}\bxi)u(\bxi_k)x\right)d\bxi\\ &\qquad\leq \eps_{\lambda,t}^{-(d-1)}(8\useconC{5})^{dN}e^{-(d\lambda^2-9d\useconC{6}\del) Nt}\kappa_i(x)^{2}\widetilde{\alpha}_i(a_{\btau_i}u(\bxi_k)x).\end{aligned}}
Combining \eqref{betaprod1}, \eqref{betaprod2}, and \eqref{ithicken} together, we get
\eqlabel{betaprod3}{\begin{aligned}\int_{[-\frac{1}{2},\frac{1}{2}]^{d-1}}&\prod_{i=1}^{d-1}\beta_{dN,\del,i,\sh}\left(a_{\btau_i}u(e^{-(2d+4\del)Nt}\bxi)u(\bxi_k)x\right)d\bxi\\&\leq \useconE{103}^{-(d-1)^2}(8\useconC{5})^{d(d-1)N}e^{-d(d-1)(\lambda^2-13\useconC{6}\del) Nt}\kappa(x)^6\prod_{i=1}^{d-1}\widetilde{\alpha}_i(a_{\btau_i}u(\bxi_k)x),\end{aligned}}
where \newconE{103}$\useconE{103}:=2^{d-1}\useconE{102}^{2(d-1)}\eps_{\lambda,t}^{-(d-1)^2}$.

Since $\bxi_k\in \Supp F$, we have $$\psi_{N,2^{d-1}\del,D(\sh)}(u(\bxi_k)x)\ge\psi_{N,2^{d-1}\del,D(\sh)}(u(\bxi_k)x)>0.$$ By definition of $\psi_{N,2^{d-1}\del,D(\sh)}$ it implies that there exist $\btau_1(\bxi_k),\cdots,\btau_{d-1}(\bxi_k)$ such that $$\btau_i(\bxi_k)\in\cE_{N,2^{d-1}\del, i}\cap \Om(N,D(\sh),u(\bxi_k)x)$$ for all $1\leq i\leq d-1$ and
$$\psi_{N,2^{d-1}\del,D(\sh)}(u(\bxi_k)x)=\prod_{i=1}^{d-1}\widetilde{\alpha}_i\left(a_{\btau_i(\bxi_k)}u(\bxi_k)x\right).$$
Since $\|\btau_i(\bxi_k)-\btau_i\|\leq 2^d\del Nt$ for all $i$, it follows that
\eqlabel{tauiposi}{\begin{aligned}\prod_{i=1}^{d-1}\widetilde{\alpha}_i\left(a_{\btau_i}u(\bxi_k)x\right)&\leq \useconC{5}^{d2^d\del Nt}\prod_{i=1}^{d-1}\widetilde{\alpha}_i\left(a_{\btau_i(\bxi_k)}u(\bxi_k)x\right)\\&=e^{\useconC{8} \del Nt}\psi_{N,2^{d-1}\del,D(\sh)}(u(\bxi_k)x),\end{aligned}}
where \newconC{8}$\useconC{8}:=d2^d\log \useconC{5}$. Combining \eqref{psibetabdd1}, \eqref{betaprod3}, and \eqref{tauiposi}, we obtain
\eq{\begin{aligned}\int_{[-\frac{1}{2},\frac{1}{2}]^{d-1}} \psi_{2N,\del,D(\sh)}(u(\bxi)x)d\bxi&\leq\displaystyle\sum_{k=1}^{M}\displaystyle\sum_{\btau_1\in\cE_{N,2\del,1}}\cdots\displaystyle\sum_{\btau_{d-1}\in\cE_{N,2\del,d-1}}e^{-(2d+4\del)(d-1)Nt} \\& \qquad\times \useconE{103}(8\useconC{5})^{d(d-1)N}e^{-\{d(d-1)\lambda^2-\useconC{9}\del\} Nt}\kappa(x)^6\psi_{N,2^{d-1}\del,D(\sh)}(u(\bxi_k)x),\end{aligned}}
where \newconC{9}$\useconC{9}:=13d(d-1)\useconC{6}+\useconC{8}$. Since $\#\cE_{N,2\del,i}\leq \#\cD_{dN}\leq (dN)^{d-1}$ for all $i$, we have
\eqlabel{Eprodcount}{\#\left(\cE_{N,2\del,1}\times\cdots\times\cE_{N,2\del,d-1}\right)\leq (dN)^{(d-1)^2}.}
It follows that
\eqlabel{psibdd}{\begin{aligned}\int_{[-\frac{1}{2},\frac{1}{2}]^{d-1}} \psi_{2N,\del,\sh}(u(\bxi)x)d\bxi\leq \useconE{103}\eps_{\lambda,t}^{-(d-1)^2}&(dN)^{(d-1)^2}(8\useconC{5})^{d(d-1)N}e^{-\{d(d-1)\lambda^2-\useconC{9}\del\}Nt}\\&\times e^{-(2d+4\del)(d-1)Nt}\kappa(x)^6\displaystyle\sum_{k=1}^{M}F(\bxi_k).\end{aligned}}
On the other hand, by (2) of Lemma \ref{psisupp} we have
\eqlabel{Fbdd}{\displaystyle\sum_{k=1}^{M} F(\bxi_k)\leq 
    \useconE{102}^{d-1}e^{\{2d(d-1)+5d^2\del\} Nt}\kappa(x)^2\int_{[-\frac{1}{2},\frac{1}{2}]^{d-1}}\psi_{N,2^{d-1}\del,\sh}(u(\bxi)x)d\bxi.}
Therefore, combining \eqref{psibdd} and \eqref{Fbdd}, we obtain the desired inequality for 
$$\useconE{104}:=\useconE{102}^{d-1}d^{(d-1)^2}\useconE{103},\qquad \useconC{10}:=\useconC{9}+5d^2-4(d-1).$$
\end{proof}

\section{Hausdorff dimension of divergent on average trajectories}
The goal of this section is to prove Theorem \ref{dimupperbdd} and Theorem \ref{dimlowbdd}.
\subsection{Upper bound}
We shall prove Theorem \ref{dimupperbdd} by employing the higher rank height functions developed in the previous sections. For $x\in X'$ we denote $$\mathsf{M}_i(x):=\max\set{\sup_{\bxi\in[-\frac{1}{2},\frac{1}{2}]^{d-1}}\widetilde{\alpha}_i(u(\bxi)x),\useconE{101}e^{t}}$$
for $1\leq i\leq d-1$ and $\mathsf{M}(x):=\prod_{i=1}^{d-1}\mathsf{M}_i(x)$. The upper bound for the Hausdorff dimension follows from the following contraction estimate which will be proved by repeatedly applying Proposition \ref{multicont}.

\begin{prop}[Contraction for $\psi$]\label{nontrivbdd}
For any $l, N\in\bN$, $0<\del<2^{-(d-1)l}$, $\sh>D^{(l)}(\useconE{101}e^{t})$ and $x\in X'$, we have
\eq{\begin{aligned}\int_{[-\frac{1}{2},\frac{1}{2}]^{d-1}}\psi_{2^l N,\del,\sh}(u(\bxi)x)d\bxi&\leq \useconE{101}(2\useconE{104})^{d^2l^2}N^{d^2l}(8\useconC{5})^{d(d-1)2^{l}N}\\&\quad\times e^{-\{d(d-1)(2^l-1)\lambda^2-(\useconC{10}2^{dl}\del+\useconC{11})\}Nt}\mathsf{M}(x).\end{aligned}}
\end{prop}

To show Proposition \ref{nontrivbdd} we will also need the following estimate.
\begin{lem}[Trivial bound for $\psi$]\label{trivlem}
For any $N\in\bN$, $0<\del<1$, $0<\eps<\eps_{t,\lambda}$, $\sh\geq 1$, and $x\in X'$ we have
\eqlabel{trivbdd}{\int_{[-\frac{1}{2},\frac{1}{2}]^{d-1}}\psi_{N,\del,\sh}(u(\bxi)x)d\bxi\leq \useconE{101}e^{\useconC{11} Nt}\mathsf{M}(x)}
for some constant \newconC{11}$\useconC{11}>0$.
\end{lem}
\begin{proof}
For any $\btau\in\cD_N$ and $\bxi\in[-\frac{1}{2},\frac{1}{2}]^{d-1}$, we have $a_{\btau} \in B^A(\operatorname{id}, 2dNt)$. It follows from \eqref{Lipschitz'} that 
\eq{\alpha_i'(a_{\btau}u(\bxi)x)\leq \useconC{4}^{2dNt}\alpha_i'(u(\bxi)x)}
for any $1\leq i\leq d-1$. Since $\kappa_i(a_{\btau}u(\bxi)x)\leq e^{2dNt}$, we have
\eq{\begin{aligned}\widetilde{\alpha}_i(a_{\btau}u(\bxi)x)&\leq \max\set{\useconE{101}e^{t} \kappa_i(a_{\btau}u(\bxi)x), \alpha_i'(a_{\btau}u(\bxi)x)}\\&\leq \max\set{\useconE{101}e^{(2dN+1)t}, \useconC{4}^{2dNt}\mathsf{M}_i(x)}\leq \useconE{101}\big(\max\set{e^2,\useconC{4}}\big)^{2dNt}\mathsf{M}_i(x)\end{aligned}}
for any $1\leq i\leq d-1$. By definition of $\psi_{N,\del,\sh}$ we clearly have a trivial estimate
\eq{\psi_{N,\del,\sh}(u(\bxi)x)\leq \big(\max\set{e,\useconC{4}}\big)^{2d(d-1)Nt}\mathsf{M}(x)}
for any $\bxi\in[-\frac{1}{2},\frac{1}{2}]^{d-1}$, hence

\eq{\int_{[-\frac{1}{2},\frac{1}{2}]^{d-1}}\psi_{N,\del,\sh}(u(\bxi)x)d\bxi\leq \useconE{101}e^{\useconC{11} Nt}\mathsf{M}(x),}
where $\useconC{11}:=2d(d-1)\max\set{2,\log \useconC{4}}$.
\end{proof}

\begin{proof}[Proof of Proposition \ref{nontrivbdd}]
For the sake of simplicity, let us write
\eq{\mathsf{Q}_j:=\int_{[-\frac{1}{2},\frac{1}{2}]^{d-1}}\psi_{2^{l-j}N,2^{(d-1)j}\del,D^{(l-j)}(\sh)}(u(\bxi)x)d\bxi,}
where $0\leq j\leq l$. We apply the inductive contraction inequality for $\psi$ in Proposition \ref{multicont} with $2^{l-j-1}N\in\bN$, $2^{(d-1)j}\del$, and $D^{(l-j-1)}(\sh)$, where $0\leq j\leq l-1$. Then we have
\eqlabel{jindiv}{\mathsf{Q}_j\leq \useconE{104}(2^{l-j-1}N)^{d^2}(8\useconC{5})^{d(d-1)2^{l-j-1}N}e^{-\{d(d-1)\lambda^2-\useconC{10} 2^{(d-1)j}\del\}2^{l-j-1}Nt}\mathsf{Q}_{j+1}}
for any $0\leq j\leq l-1$. By Lemma \ref{trivlem} we also have
\eqlabel{trivbdd'}{\mathsf{Q}_{l}\leq \useconE{101}e^{\useconC{11} Nt}\mathsf{M}(x).}
Multiplying all the inequalities \eqref{jindiv} for $0\leq j\leq l-1$ and \eqref{trivbdd'}, we get
\eq{\mathsf{Q}_0\leq \useconE{101}(2\useconE{104})^{d^2l^2}N^{d^2l}(8\useconC{5})^{d(d-1)2^{l}N} e^{-\{d(d-1)(2^l-1)\lambda^2-(\useconC{10}2^{dl}\del+\useconC{11})\}Nt}\mathsf{M}(x).}
\end{proof}

\begin{proof}[Proof of Theorem \ref{dimupperbdd}]
For $x\in X'$ we denote by $Z_x$ the set of all $\bxi\in[-\frac{1}{2},\frac{1}{2}]^{d-1}$ such that $u(\bxi)x$ is $A^+$-divergent on average. We first consider $0<\lambda<1$, $t=t_\lambda$, and $l\in\bN$ as fixed. For $N\in\bN$, $0<\del<2^{-(d-1)l}$, $0<\eps<2^{-4l}\eps_{\lambda,t}$ and a compact set $K\subset X$ we denote by $Z_x(N,\del,K)$ the set
\eq{\set{\bxi\in[-\tfrac{1}{2},\tfrac{1}{2}]^{d-1}: \frac{1}{\#\cD_N}|\{\btau\in\cD_N: a_{\btau}u(\bxi)x\notin K\}|\ge 1-\frac{\del^d}{4\useconC{7}}}.}

We claim that
\eqlabel{Zxcont}{Z_x\subset \bigcup_{N_0\ge 1}\bigcap_{N\ge N_0}Z_x(N,\del,K)}
for any $\del$ and $K$. To see this, suppose that $\bxi\notin \displaystyle\bigcup_{N_0\ge 1}\displaystyle\bigcap_{N\ge N_0}Z_x(N,\del,K)$. Then we have
\eqlabel{escape1}{\frac{1}{\#\cD_N}|\{\btau\in\cD_N: a_{\btau}u(\bxi)x\notin K\}|< 1-\frac{\del^d}{4\useconC{7}}}
for infinitely many $N$. It follows that 
\eqlabel{escape2}{\frac{1}{(dN)^{d-1}}|\{\btau\in\set{t,2t,\ldots,dNt}^{d-1}: a_{\btau}u(\bxi)x\notin K\}|< 1-\frac{\del^d}{4d^{d-1}\useconC{7}}}
for infinitely many $N$ since $\set{t,2t,\ldots,Nt}^{d-1}\subseteq\cD_N\subseteq\set{t,2t,\ldots,dNt}^{d-1}$. We can find a compact set $\widetilde{K}\subset X$ such that $a_{\btau}y\subset \widetilde{K}$ for any $\btau\in[-t,t]^{d-1}$ and $y\in K$. Then \eqref{escape2} implies that
\eqlabel{escape3}{\frac{1}{(dNt)^{d-1}}|\{\btau\in\set{1,2,\ldots,dNt}^{d-1}: a_{\btau}u(\bxi)x\notin \widetilde{K}\}|< 1-\frac{\del^d}{4d^{d-1}\useconC{7}}}
for infinitely many $N$, hence $u(\bxi)x$ is not $A^+$-divergent on average. This proves the claim.

For $\sh>D^{(l)}(\useconE{101}e^{t})$, let $$K_{\sh}=\set{x\in X: \displaystyle\max_{1\leq i\leq d-1}\widetilde{\alpha}_i(x)\leq\sh}.$$ Since $\widetilde{\alpha}_i$ is proper for any $1\leq i\leq d-1$, $K_{\sh}$ is compact. Let $N'=2^lN$. By Lemma \ref{compactness}, $B(\bxi,e^{-2dN't})\subset K_{\sh}$ holds for any $\bxi\in K_{D(\sh)}$. Thus, if $\bxi\in Z_x(N',\del,K_{D(\sh)})$ then for any $\bxi'\in B(\bxi,e^{-2dN't})$ we have $\bxi'\in Z_x(N',\del,K_{\sh})$, hence 
\eq{\frac{\#\Om(N',\sh,u(\bxi')x)}{\#\cD_{N'}}=\frac{|\{\btau\in\cD_{N'}: a_{\btau}u(\bxi')x\notin K_{\sh}\}|}{\#\cD_{N'}}\ge 1-\frac{\del^d}{4\useconC{7}}.}
It follows that \eqlabel{psilowbdd}{\psi_{N',\del,\sh}(u(\bxi')x)=\displaystyle\prod_{i=1}^{d-1}\displaystyle\min_{\btau\in\cE_{N',\del,i}\cap\Om(N',\sh,x)}\widetilde{\alpha}_i(a_{\btau}u(\bxi')x)\ge\sh^{d-1}.}

Let $\set{\bxi_1,\cdots,\bxi_M(N)}$ be a maximal $e^{-2dN't}$-separated subset of $Z_x(N',\del,K_{D(\sh)})$. Then $Z_x(N',\del,K_{D(\sh)})$ is covered with $M=M(N',\del,K_{D(\sh)})$ balls of length $e^{-2dN't}$. By \eqref{psilowbdd} we have $\psi_{N',\del,\eps,\sh}(u(\bxi')x)\ge\sh^{d-1}$ for any $1\leq k\leq M$ and $\bxi'\in B(\bxi_k,\frac{1}{2}e^{-2dN't})$. Note that the balls $B(\bxi_k,\frac{1}{2}e^{-2dN't})$ are disjoint since $\bxi_k$'s are $e^{-2dN't}$-separated. It follows that
\eq{\begin{aligned}\int_{[-\frac{1}{2},\frac{1}{2}]^{d-1}}\psi_{N',\del,\sh}(u(\bxi)x)d\bxi&\ge \displaystyle\sum_{k=1}^{M}\int_{B(\bxi_k,\frac{1}{2}e^{-2dN't})}\psi_{N',\del,\sh}(u(\bxi)x)d\bxi\\&\ge M\left(\tfrac{1}{2}\sh\right)^{d-1}e^{-2d(d-1)N't}.\end{aligned}}
Hence, using Proposition \ref{nontrivbdd} we obtain
\eq{\begin{aligned}M(N',\del,K_{D(\sh)})\leq& (2\useconE{104})^{d^2l^2}\mathsf{M}(x)N^{d^2l}(8\useconC{5})^{d(d-1)2^{l}N}\\&\qquad\times e^{\{d(d-1)(2^{l+1}-2^l\lambda^2+\lambda^2)+(\useconC{10}2^{dl}\del+\useconC{11})\}Nt}.\end{aligned}}
Note that the constants $\useconC{5},\useconE{104}$ and $\mathsf{M}(\lambda)$ are independent of $N,l$, and $\del$ while they might depend on $t$ and $\lambda$. Therefore, for any $N_0\in\bN$, $l\in\bN$, and $0<\del<2^{-(d-1)l}$,
\eq{\begin{aligned}\dim_H&\left(\bigcap_{N\ge N_0}Z_x(2^lN,\del,K_{D(\sh)})\right)\leq \overline{\dim}_{\operatorname{box}}\left(\bigcap_{N\ge N_0}Z_x(2^lN,\del,K_{D(\sh)})\right)\\&\qquad\qquad\qquad\qquad\qquad\quad\;\;\leq\displaystyle\limsup_{N\to\infty}\frac{\log M(N',\del,K_{D(\sh)})}{-\log (e^{-2dN't})}\\&\leq \displaystyle\limsup_{N\to\infty}\left(\frac{d^2l\log N+d(d-1)2^lN\log (8\useconC{5})}{2^{l+1}dNt}\right.\\&\qquad\qquad\qquad+\left.\frac{\{d(d-1)(2^{l+1}-2^l\lambda^2+\lambda^2)+(\useconC{10}2^{dl}\del+\useconC{11})\}Nt}{2^{l+1}dNt}\right)\\&=\frac{(d-1)\log (8\useconC{5})}{2t}+\frac{(d-1)(2^{l+1}-2^l\lambda^2+\lambda^2)}{2^{l+1}}+\frac{\useconC{10}2^{dl}\del+\useconC{11}}{2^{l+1}d}.\end{aligned}}
Using \eqref{Zxcont} and taking $\del\to 0$, $l\to\infty$ we have
\eq{\dim_H Z_x\leq \frac{(d-1)\log (8\useconC{5})}{2t}+\frac{(d-1)(2-\lambda^2)}{2}.}
Since this estimate holds for any $0<\lambda<1$ and $t=t_\lambda \to\infty$ as $\lambda\to 1$, we finally obtain
\eq{\dim_HZ_x\leq\frac{d-1}{2}.}

\end{proof}

\subsection{Lower bound}
In this subsection we prove Theorem \ref{dimlowbdd}, which directly implies that for any $d\geq 3$
\eq{\dim_H\set{x\in U\Gamma: x \textrm{ is }A^+\textrm{-divergent on average}}\geq 1.}
Recall that $u(\bxi)\Gamma$ is $g_t$-divergent (or $g_t$-divergent on average) if and only if $\bxi$ is singular (or singular on average). We can extend this notion to define analogous concepts of singular (or singular on average) vectors for higher-rank actions, and use them to prove Theorem \ref{dimlowbdd}.
\begin{defi}
Let $\bxi \in\bR^{d-1}$.
\begin{enumerate}
    \item We say that $\bxi$ is \textit{multiplicatively singular} if for any $\eps>0$ there exists $T(\eps)>0$ such that for any $n_1,\ldots,n_{d-1}\in\bN$ with $n_1+\cdots+n_{d-1}>T(\eps)$ there exists $q\in\bN$ such that
\eqlabel{simulsing}{\|q\xi_i\|_\bZ\leq \eps e^{-n_i} \textrm{ for all }1\leq i\leq d-1  \textrm{ and }0<q<\eps e^{\sum n_i}.}
    \item We say that $\bxi$ is \textit{multiplicatively singular on average} if
    \eq{\lim_{N\to\infty}\frac{1}{N^{d-1}}\left|\set{(n_1,\ldots,n_{d-1})\in\set{1,\ldots,N}^{d-1}: \textrm{ there exists } q\in\bZ^{d-1} \textrm{ s.t. }\eqref{simulsing} \textrm{ holds}}\right|=1}
    for any $\eps>0$.
\end{enumerate}
\end{defi}

The following lemma is a version of Dani's correspondence between divergent on average trajectories for higher rank actions and multiplicatively singular on average vectors.

\begin{lem}\label{Danicor}
For any $\bxi=(\xi_1,\ldots,\xi_{d-1})\in\bR^{d-1}$, $u(\bxi)\Gamma$ is $A^+$-divergent on average if and only if $\bxi$ is multiplicatively singular on average.
\end{lem}
\begin{proof}
Suppose that $u(\bxi)\Gamma$ is $A^+$-divergent on average. For $\epsilon>0$ we denote
$$\cK_{\epsilon}:=\set{g\Gamma\in X: \min_{v\in\mathbb{Z}^d\setminus\set{0}}\|gv\|\geq\epsilon}.$$
Then for any $\eps>0$ the set $\cK_{\epsilon}$ is compact by the classical Mahler's criterion, thus we have
$$\lim_{N\to\infty}\frac{1}{N^{d-1}}\#\set{\btau\in\set{1,\ldots,N}^{d-1}: a_{\btau}u(\bxi)\Gamma\notin \cK_\eps}=1.$$
For $\btau=(n_1,\ldots,n_{d-1})$ the condition $a_{\btau}u(\bxi)\Gamma\notin \cK_\eps$ implies that there exists a vector $(m_1,\ldots,m_{d-1},q)\in\bZ^{d}\setminus\set{0}$ such that
$$v_{(m_1,\cdots,m_{d-1},q)}=\left(\begin{matrix}
        e^{n_1}(m_1+q\xi_1)\\ \vdots \\ e^{n_{d-1}}(m_{d-1}+q\xi_{d-1}) \\  e^{-(n_1+\cdots+n_{d-1})}q
    \end{matrix}\right)$$
    satisfies $\|v_{(m_1,\ldots,m_{d-1},q)}\|< \eps$, hence also implies the condition \eqref{simulsing} for $\btau=(n_1,\ldots,n_{d-1})$. Thus $\bxi$ is multiplicatively singular on average. The proof of the converse direction is similar.
\end{proof}

\begin{proof}[Proof of Theorem \ref{dimlowbdd}]
If $\dim \operatorname{span}_\bQ(1,\xi_1,\ldots,\xi_{d-1})\leq 2$, then there exists $\xi\in\bR$ such that for any $1\leq i\leq d-1$
there exist $a_i,b_i\in\bQ$ satisfying $\xi_i=a_i\xi+b_i$. Let $D$ be the l.c.m of $a_1,\ldots,a_{d-1},b_1,\ldots,b_{d-1}$ and write $a_i=\frac{a_i'}{D}$ and $b_i=\frac{b_i'}{D}$ for all $1\leq i\leq d-1$. 

We claim that for any $\eps>0$ if $(n_1,\ldots,n_{d-1})$ satisfies \eqlabel{edgeest}{\max_{1\leq i\leq d-1}n_i\leq \sum_{1\leq i\leq d-1}n_i-\log \left(\frac{1}{2\eps^2}\displaystyle\max_{1\leq i\leq d-1}Da_i'\right),}
then $(n_1,\ldots,n_{d-1})$ also satisfies \eqref{simulsing}. To see this, let $T=\sum_i n_i$ and $Q=\lfloor\frac{\eps e^T}{D}\rfloor$. By Dirichlet's theorem, there exists $1\leq q\leq Q$ such that $\|q\xi\|_\bZ\leq\frac{1}{Q}$. It follows that for any $1\leq i\leq d-1$ we have $\|(Dq)\xi_i\|_\bZ\leq \frac{a_i'}{Q}$. Assuming \eqref{edgeest} for $(n_1,\ldots,n_{d-1})$, we have
\eq{\|(Dq)\xi_i\|_\bZ\leq \frac{a_i'}{Q}\leq \frac{\displaystyle\max_{1\leq i\leq d-1}Da_i'}{\eps e^{T+1}}\leq \eps e^{-n_i}}
for any $i$. Then \eqref{simulsing} holds for $(n_1,\ldots,n_{d-1})$ since $Dq\leq DQ\leq \eps e^T$. Hence the claim is proved.

For any $\eps>0$, the density of the set of $(n_1,\ldots,n_{d-1})\in\bN^{d-1}$ satisfying \eqref{edgeest} is $1$. By Lemma \ref{Danicor} and the above claim, $u(\bxi)\Gamma$ is $A^+$-divergent on average.
\end{proof}

\section{Dynamical height functions for a reference trajectory}
\subsection{Injectivity radius}
For $x\in X$ we denote by $\operatorname{inj}(x)$ the maximal injectivity of radius of $x\in X$, which is the supremum of $r>0$ such that the map $g\mapsto gx$ is injective in the ball $B^G(\operatorname{id},r)$. According to \cite[Proposition 3.5]{KM12} there exists a constant \newconC{12}$\useconC{12}>1$ such that the following holds: for any $x\in X$ there exists a nonzero $x$-integral vector $v\in\bR^d$ such that $\|v\| < \useconC{12}\operatorname{inj}(x)^{\frac{1}{d}}$.

 Let $0<\lambda<1$ be given, let $\widetilde{\alpha}_i:=\widetilde{\alpha}_{i,\lambda,t_\lambda}$ for $1\leq i\leq d-1$ be as in \eqref{alphatildedefinition}.
 
\begin{lem}
    For any $1\leq i\leq d-1$ and $x\in X'$ with $\operatorname{inj}(x)\leq \useconC{12}^{-d}\eps_{\lambda,t_\lambda}^{d(d-1)}$, 
    $$\widetilde{\alpha}_{i}(x)> \useconC{12}^{-1}\eps_{\lambda,t_\lambda}^{d-1}\operatorname{inj}(x)^{-\frac{\lambda}{d}}.$$
\end{lem}
\begin{proof}
    If $\operatorname{inj}(x)\leq \useconC{12}^{-d}\eps_{\lambda,t_\lambda}^{d(d-1)}$, there exists an $x$-integral vector $v$ such that $\|v\| < \useconC{12}\operatorname{inj}(x)^{\frac{1}{d}}\leq\eps_{\lambda,t_\lambda}^{d-1}$, hence $\|\pi_{i,1,0}(v)\|\leq \|v\|<\eps_{\lambda,t_\lambda}^{d-1}$. It follows that
    $$\varphi_{i,1,\eps_{\lambda,t_\lambda}}(v)=\eps_{\lambda,t_\lambda}^{d-1}\|\pi_{i,1,1}(v)\|^{-\lambda}\geq \eps_{\lambda,t_\lambda}^{d-1}\|v\|^{-\lambda}>\useconC{12}^{-1}\eps_{\lambda,t_\lambda}^{d-1}\operatorname{inj}(x)^{-\frac{\lambda}{d}}.$$
    By definition of $\alpha_{i,\lambda,t_0}=\alpha_{i,\eps_{\lambda,t_\lambda}}$ we have
    $$\alpha_{i,\lambda,t_0}(x)\geq \varphi_{i,1,\eps_{\lambda,t_\lambda}}(v)>\useconC{12}^{-1}\eps_{\lambda,t_\lambda}^{d-1}\operatorname{inj}(x)^{-\frac{\lambda}{d}}.$$

    On the other hand, the \textbf{Claim} in the proof of Proposition \ref{alpha'} implies
    $$\alpha_{i,\lambda}'(x)\geq\|v\|^{-\lambda}>\useconC{12}^{-1}\operatorname{inj}(x)^{-\frac{\lambda}{d}}.$$
    Since $\widetilde{\alpha}_i(x)\geq \min\set{\alpha_{i,\lambda,t_0}(x),\alpha_{i,\lambda}'(x)}$, we obtain $\widetilde{\alpha}_i(x)>\useconC{12}^{-1}\eps_{\lambda,t_\lambda}^{d-1}\operatorname{inj}(x)^{-\frac{\lambda}{d}}$.
\end{proof}

\subsection{Dynamical height function for a reference trajectory under one-dimensional action}
This subsection aims to construct dynamical height functions for a given reference trajectory under one-dimensional diagonal actions and to establish the contraction hypotheses for these height functions. 

Given $0<\lambda<1$ let us choose and fix $l_\lambda\in\bN$ such that $1-2^{-l_\lambda+1}>\lambda$. For $t\in t_\lambda\bN$ we write 
$$\sh_\lambda(t):=D^{(l_\lambda)}\big(\useconE{101}(3\useconC{5})^t\big).$$
Denote by $Q$ the parabolic subgroup
$$\set{\left(\begin{matrix}
    g' & 0 \\ v & \tau
\end{matrix}\right)\in G: g'\in\operatorname{GL}_{d-1}(\bR), v\in\bR^{d-2}, \tau\in\bR^{\times}}$$ of $G$, and for $t\in t_\lambda\bN$ and $0<\eta\leq 1$ define boxes
\eqlabel{boxDefinition}{\sB_{t,\eta}:=B^U(\operatorname{id}, \eta e^{-2t})B^{Q}(\operatorname{id}, \eta e^{-\frac{t}{2}}), \quad \widetilde{\sB}_{t,\eta}:=B^U(\operatorname{id}, \eta e^{-\frac{t}{2}})B^{Q}(\operatorname{id}, \eta e^{-\frac{t}{2}}).} 
Observe that the boxes $\sB_{t,\eta}$ and $\widetilde{\sB}_{t,\eta}$ are open sets in $G$ and satisfy the following property: for any $1\leq i\leq d-1$ and $g_1,g_2\in G$, the set $\set{s\in\bR: u_i(s)g_1\in \sB_{t,\eta}g_2}$ is either the empty set or an interval of length $\eta e^{-2t}$, and similarly, the set $\set{s\in\bR: u_i(s)g_1\in \widetilde{\sB}_{t,\eta}g_2}$ is either the empty set or an interval of length $\eta e^{-\frac{t}{2}}$. Note that it also implies the analogous property for $x_1,x_2\in X$ with $\operatorname{inj}(x_1),\operatorname{inj}(x_2)>\eta e^{-\frac{t}{2}}$ in the place of $g_1,g_2\in G$.
 
For $1\leq i\leq d-1$, $t\in t_\lambda\bN$, and $y\in X$ let us consider the following three sets:
$$ J_{i,t,\operatorname{low}}(y):=\set{j\in\bN: \widetilde{\alpha}_i(a_{jt\be_i}y)\leq e^{\frac{t}{3d}}},$$
$$ J_{i,t,\operatorname{mid}}(y):=\set{j\in\bN: e^{\frac{t}{3d}}< \widetilde{\alpha}_i(a_{jt\be_i}y)\leq \sh_\lambda(t)},$$
$$ J_{i,t,\operatorname{high}}(y):=\set{j\in\bN: \widetilde{\alpha}_i(a_{jt\be_i}y)> \sh_\lambda(t)}.$$
For $N\in\bN$, $0<\eta\leq1$, and $0\leq l\leq l_\lambda$ we denote by $\Upsilon_{N,i,t,\eta,l}(y)\subseteq X$ the set of $x\in X$ such that 
$$a_{jt\be_i}x\in \sB_{t,\eta}a_{jt\be_i}y \textrm{ for all } j\in J_{i,t,\operatorname{low}}(y)\cap\set{1,\ldots,N},$$ $$\widetilde{\alpha}_i(a_{jt\be_i}x)>D^{(l)}\big(\useconE{101}(3\useconC{5})^t\big) \textrm{ for all } j\in J_{i,t,\operatorname{high}}(y)\cap \set{1,\ldots, N}.$$ We now define a proper lower semi-continuous function $\beta_{N,i,t,l}'(\;\cdot\;;y):X'\to[1,\infty)$ by
\eq{\beta_{N,i,t,l}'(x;y)=\begin{cases}\widetilde{\alpha}_i(a_{Nt\be_i}x) & \text{if $x\in\Upsilon_{N,i,t,1,l}(y)$,}\\ 0 & \text{otherwise}.\end{cases}}
We note that $0\leq \beta_{N,i,t,l}'\leq \beta_{N,i,t,l'}'$ for any $0\leq l'\leq l\leq l_{\lambda}$, as the sets $\Upsilon_{N,i,t,1,l}(y)$ are decreasing in $l$.

We state the contraction hypothesis for $\beta'_{N,i,t,l}$.

\begin{prop}[Contraction for $\beta'$]\label{beta'tdecay}
For any $N\in\bN$, $1\leq i\leq d-1$, $t\in t_\lambda\bN$, $0\leq l\leq l_\lambda$, $y\in X$ and $x\in X'$,
\eqlabel{beta'tcont}{\int_{-\frac{1}{2}}^{\frac{1}{2}}\beta_{N,i,t,l}'(u_{i}(s)x;y)ds\leq 2\useconC{5}\useconE{101} e^t(3C_{5}\useconE{101})^{|\set{1,\ldots,N}\cap J_{i,t,\operatorname{mid}}(y)|t}e^{-\lambda^2 Nt}\widetilde{\alpha}_i(x).}
\end{prop}
\begin{proof}
We may assume $l=0$, since $\beta_{N,i,t,l}'\leq \beta_{N,i,t,0}'$ holds for any $0\leq l\leq l_{\lambda}$. We denote $\varpi(N)=0$ if $N\in J_{i,t,\operatorname{mid}}(y)\cup J_{i,t,\operatorname{high}}(y)$ and $\varpi(N)=1$ if $N\in J_{i,t,\operatorname{low}}(y)$. The desired inequality \eqref{beta'tcont} immediately follows the following claim:

\textbf{Claim.} The support of the map $s\mapsto \beta_{N,i,t,l}'(u_{i}(s)x;y)$ is covered by $M_N$ intervals $I_{N,1},\ldots, I_{N,M_N}$, where $$I_{N,k}=\left[s_{N,k}-\tfrac{1}{2}e^{-2(N+\varpi(N))t}, s_{N,k}+\tfrac{1}{2}e^{-2(N+\varpi(N))t}\right]$$ with  $\frac{1}{2}e^{-2(N+\varpi(N))t}$-separated $s_{N,k}\in I$ for $1\leq k\leq M_N$, such that
$$\sum_{k=1}^{M_N} \widetilde{\alpha}_i(a_{Nt\be_i}u_{i}(s_{N,k})x) e^{-2(N+\varpi(N))t}\leq 2\useconE{101} e^t(3C_{5}\useconE{101})^{|\set{1,\ldots,N}\cap J_{i,t,\operatorname{mid}}(y)|t}e^{-\lambda^2 Nt}\widetilde{\alpha}_i(x).$$

Indeed, \textbf{Claim} directly implies that
\eq{\begin{aligned}
    \int_{-\frac{1}{2}}^{\frac{1}{2}}\beta_{N,i,t,l}'(u_{i}(s)x;y)ds&\leq  \sum_{k=1}^{M_N}\int_{I_{N,k}}\widetilde{\alpha}_i(a_{Nt\be_i}u_i(s)x)ds\\&\leq \useconC{5}\sum_{k=1}^{M_N}\widetilde{\alpha}_i(a_{Nt\be_i}u_{i}(s_{N,k})x) e^{-2(N+\varpi(N))t}\\&\leq 2\useconC{5}\useconE{101} e^t(3C_{5}\useconE{101})^{|\set{1,\ldots,N}\cap J_{i,t,\operatorname{mid}}(y)|t}e^{-\lambda^2 Nt}\widetilde{\alpha}_i(x).
\end{aligned}}

From now on, we shall prove \textbf{Claim} by induction on $N$.

We first verify \textbf{Claim} for $N=1$. We may choose intervals $I_{1,1},\cdots,I_{1,M_1}$ of length $e^{-2(1+\varpi(1))t}$ so that the centers of $I_{1,k}$'s are $\frac{1}{2}e^{-2(1+\varpi(1))t}$-separated. Then at most $2$ intervals can overlap, hence by \eqref{alphatildeupperbound} we have
\eq{\sum_{k=1}^{M_1}\int_{I_{1,k}}\widetilde{\alpha}_i(a_{t\be_i}u_i(s)x)ds\leq 2\int_{-\frac{1}{2}}^{\frac{1}{2}}\widetilde{\alpha}_i(a_{t\be_i}u_i(s)x)ds\leq 2\useconE{101} \widetilde{\alpha}_i(x).}

We now suppose that \textbf{Claim} holds for $N$, and prove \textbf{Claim} for $N+1$. The support of the map $s\mapsto \beta_{N+1,i,t}'(u_{i}(s)x;y)$ is contained in $\bigcup_{k=1}^{M_N}I_{N,k}$. For each $1\leq k\leq M_N$, by \eqref{alphatildeupperbound} we have \eqlabel{eq:onesteplow}{\int_{I_{N,k}}\widetilde{\alpha}_i(a_{(N+1)t\be_i}u_i(s)x)ds\leq \useconE{101} \widetilde{\alpha}_i(a_{Nt\be_i}u_i(s_{N,k})x).}
    Moreover, if $\widetilde{\alpha}_i(a_{Nt\be_i}u_i(s_{N,k})x)\geq \useconE{101}e^{t}$, then by \eqref{alphatildecont} \eqlabel{eq:onestephigh}{\int_{I_{N,k}}\widetilde{\alpha}_i(a_{(N+1)t\be_i}u_i(s)x)ds\leq 2e^{-\lambda^2t}\widetilde{\alpha}_i(a_{Nt\be_i}u_i(s_{N,k})x).}

    \textbf{Case 1.} $N+1\in J_{i,t,\operatorname{low}}(y)$.

In this case $N$ is either in $J_{i,t,\operatorname{low}}(y)$ or $J_{i,t,\operatorname{mid}}(y)$. Indeed, $N+1\in J_{i,t,\operatorname{low}}(y)$ implies that
\eq{\widetilde{\alpha}_i(a_{Nt\be_i}y)\leq \useconC{5}^t\widetilde{\alpha}_i(a_{(N+1)t\be_i}y)\leq \useconC{5}^te^{\frac{t}{3d}}\leq \useconE{101}(3\useconC{5})^t,}
hence $N\in J_{i,t,\operatorname{low}}(y)\cup J_{i,t,\operatorname{mid}}(y)$. Let us denote by
$$\cT:=\set{s\in [-\tfrac{1}{2},\tfrac{1}{2}]: a_{(N+1)t\be_i}u_i(s)x\in \sB_ta_{(N+1)t\be_i}y},$$
$$\widetilde{\cT}:=\set{s\in [-\tfrac{1}{2},\tfrac{1}{2}]: a_{(N+1)t\be_i}u_i(s)x\in \widetilde{\sB}_{t,\eps}a_{(N+1)t\be_i}y}.$$
Note that $\cT$ is a union of disjoint intervals of length $e^{-2t}\cdot e^{-2(N+1)t}$, and $\widetilde{\cT}$ is a union of disjoint intervals of length $e^{-\frac{t}{2}}\cdot e^{-2(N+1)t}$. Moreover, each subinterval of $\cT$ is contained in exactly one subinterval of $\widetilde{\cT}$. It follows from \eqref{alphatildeLipschitz} that for any $1\leq k\leq M_N$
\eqlabel{eq:Case1estimate}{\begin{aligned}
    \int_{I_{N,k}\cap \cT}\widetilde{\alpha}_i(a_{(N+1)t\be_i}u_i(s)x)ds&\leq \useconC{5}e^{-\frac{3t}{2}}\int_{I_{N,k}\cap \cT'}\widetilde{\alpha}_i(a_{(N+1)t\be_i}u_i(s)x)ds\\&\leq \useconC{5}e^{-\frac{3t}{2}}\int_{I_{N,k}}\widetilde{\alpha}_i(a_{(N+1)t\be_i}u_i(s)x)ds.
\end{aligned}}

\textbf{Subcase 1-1.} $N\in J_{i,t,\operatorname{low}}(y)$.\\
In this subcase $I_{k,N}$ is an interval of length $e^{-2(N+\varpi(N))t}=e^{-2(N+1)t}$. We first show that
\eqlabel{eq:subcase1}{\int_{I_{N,k}\cap \cT}\widetilde{\alpha}_i(a_{(N+1)t\be_i}u_i(s)x)ds\leq \useconC{5}^3 e^{-(\frac{3}{2}-\frac{1}{3d})t}|I_{N,k}|}
for any $1\leq k\leq M_N$. We may assume that $I_{N,k}\cap\cT\neq\emptyset$. Then there exists $s_{N,k}'\in I_{N,k}$ such that 
$$a_{(N+1)t\be_i}u_i(s_{N,k}')x\in \sB_t a_{(N+1)t\be_i}y,$$ hence
$$\widetilde{\alpha}_i(a_{(N+1)t\be_i}u_i(s)x)\leq \useconC{5}\widetilde{\alpha}_i(a_{(N+1)t\be_i}y)\leq \useconC{5}e^{\frac{t}{3d}}.$$
It follows that
\eq{\int_{I_{N,k}}\widetilde{\alpha}_i(a_{(N+1)t\be_i}u_i(s)x)ds\leq \useconC{5}\widetilde{\alpha}_i(a_{(N+1)t\be_i}u_i(s_{N,k}')x)|I_{N,k}|\leq \useconC{5}^2e^{\frac{t}{3d}}|I_{N,k}|.}
Thus, combining with \eqref{eq:Case1estimate} we get \eqref{eq:subcase1}.

Note that each connected component of $I_{N,k}\cap \cT$ is an interval of length at most $e^{-2(N+2)t}=e^{-2(N+1+\varpi(N+1))t}$. Let us define $I_{N+1,1},\ldots,I_{N+1,M_{N+1}}$ so that each $I_{N+1,l}$ for $1\leq l\leq M_{N+1}$ is of length $e^{-2(N+2)t}$ and contains exactly one connected component of $\bigcup_{k=1}^{M_k}I_{N,k}\cap \cT$. Then using \eqref{eq:subcase1} and the inductive hypothesis,
\eq{\begin{aligned}
    \sum_{l=1}^{M_{N+1}} \widetilde{\alpha}_i(a_{(N+1)t\be_i}u_{i}(s_{N+1,l})x) e^{-2(N+2)t}&\leq \useconC{5}\sum_{l=1}^{M_{N+1}}\int_{I_{N+1,l}}\widetilde{\alpha}_i(a_{(N+1)t\be_i}u_{i}(s)x)ds\\
    &\leq \useconC{5}\sum_{k=1}^{M_N}\int_{I_{N,k}\cap \cT}\widetilde{\alpha}_i(a_{(N+1)t\be_i}u_i(s)x)ds\\
    &\leq \useconC{5}^4 e^{-(\frac{3}{2}-\frac{1}{3d})t}\sum_{k=1}^{M_N}\widetilde{\alpha}_i(a_{Nt\be_i}u_i(s_k)x)e^{-2(N+1)t}\\&\leq  e^{-t}(3C_{5}\useconE{101})^{|\set{1,\ldots,N}\cap J_{i,t,\operatorname{mid}}(y)|t}e^{-\lambda^2 Nt}\widetilde{\alpha}_i(x)\\&\leq (3C_{5}\useconE{101})^{|\set{1,\ldots,N+1}\cap J_{i,t,\operatorname{mid}}(y)|t}e^{-\lambda^2 (N+1)t}\widetilde{\alpha}_i(x).
\end{aligned}}

\textbf{Subcase 1-2.} $N\in J_{i,t,\operatorname{mid}}(y)$.\\
In this subcase $I_{k,N}$ is an interval of length $e^{-2(N+\varpi(N))t}=e^{-2Nt}$.
We have
\eq{\begin{aligned}
    \int_{I_{N,k}}\widetilde{\alpha}_i(a_{(N+1)t\be_i}u_i(s)x)ds&=|I_{N,k}|\int_{-\frac{1}{2}}^{\frac{1}{2}}\widetilde{\alpha}_i(a_{t\be_i}u_i(s)a_{Nt\be_i}u_i(s_k)x)ds\\&\leq \useconE{101} \widetilde{\alpha}_i(a_{Nt\be_i}u_i(s_k)x)|I_{N,k}|.
\end{aligned}}
by \eqref{alphatildeupperbound}. Combining with \eqref{eq:Case1estimate} we get
\eqlabel{eq:subcase2}{\int_{I_{N,k}\cap \cT}\widetilde{\alpha}_i(a_{(N+1)t\be_i}u_i(s)x)ds\leq \useconC{5}\useconE{101} e^{-\frac{3}{2}t}\widetilde{\alpha}_i(a_{Nt\be_i}u_i(s_k)x)|I_{N,k}|.}

The rest of the argument is similar to \textbf{Subcase 1-1}. We define $I_{N+1,1},\cdots,I_{N+1,M_{N+1}}$ so that each $I_{N+1,l}$ for $1\leq l\leq M_{N+1}$ is of length $e^{-2(N+2)t}$ and contains exactly one connected component of $\bigcup_{k=1}^{M_k}I_{N,k}\cap \cT$. Then using \eqref{eq:subcase2} and the inductive hypothesis,
\eq{\begin{aligned}
    \sum_{l=1}^{M_{N+1}} \widetilde{\alpha}_i&(a_{(N+1)t\be_i}u_{i}(s_{N+1,l})x) e^{-2(N+2)t}\leq \useconC{5}\sum_{l=1}^{M_{N+1}}\int_{I_{N+1,l}}\widetilde{\alpha}_i(a_{(N+1)t\be_i}u_{i}(s)x)ds\\
    &\leq \useconC{5}\sum_{k=1}^{M_N}\int_{I_{N,k}\cap \cT}\widetilde{\alpha}_i(a_{(N+1)t\be_i}u_i(s)x)ds\\
    &\leq \useconC{5}^2\useconE{101} e^{-\frac{3}{2}t}\sum_{k=1}^{M_N}\widetilde{\alpha}_i(a_{Nt\be_i}u_i(s_k)x)e^{-2Nt}\\&\leq  e^{-t}C_{5}^{|\set{1,\ldots,N}\cap J_{i,t,\operatorname{mid}}(y)|t}e^{-\lambda^2 Nt}\widetilde{\alpha}_i(x)\leq C_{5}^{|\set{1,\ldots,N+1}\cap J_{i,t,\operatorname{mid}}(y)|t}e^{-\lambda^2 (N+1)t}\widetilde{\alpha}_i(x).
\end{aligned}}

\textbf{Case 2.} $N+1\in J_{i,t,\operatorname{mid}}(y)$.\\

\textbf{Subcase 2-1.} $N\in J_{i,t,\operatorname{low}}(y)$.\\
In this subcase $I_{k,N}$ is an interval of length $e^{-2(N+\varpi(N))t}=e^{-2(N+1)t}=e^{-2(N+1+\varpi(N+1))t}$ for any $1\leq k\leq M_N$. We set $M_{N+1}=M_N$ and define $I_{N+1,k}=I_{N,k}$ for $1\leq k\leq M_N$. It follows from \eqref{alphatildeLipschitz} and the inductive hypothesis that
\eq{\begin{aligned}
    \sum_{k=1}^{M_{N+1}} \widetilde{\alpha}_i(a_{(N+1)t\be_i}u_{i}(s_{N+1,k})x) e^{-2(N+1)t}&=
    \sum_{k=1}^{M_N} \widetilde{\alpha}_i(a_{(N+1)t\be_i}u_{i}(s_{N,k})x) e^{-2(N+\varpi(N))t}\\&\leq \sum_{k=1}^{M_N} \useconC{5}^{t}\widetilde{\alpha}_i(a_{Nt\be_i}u_{i}(s_{N,k})x) e^{-2(N+\varpi(N))t} \\&\leq \useconC{5}^{t}(3C_{5}\useconE{101})^{|\set{1,\ldots,N}\cap J_{i,t,\operatorname{mid}}(y)|t}e^{-\lambda^2 Nt}\widetilde{\alpha}_i(x)\\&\leq (3C_{5})^{|\set{1,\ldots,N+1}\cap J_{i,t,\operatorname{mid}}(y)|t}e^{-\lambda^2 (N+1)t}\widetilde{\alpha}_i(x).
\end{aligned}}

\textbf{Subcase 2-2.} $N\in J_{i,t,\operatorname{mid}}(y)\cup J_{i,t,\operatorname{high}}(y)$.\\
In this subcase $I_{N,k}$ is an interval of length $e^{-2(N+\varpi(N))t}=e^{-2Nt}$. We cover each $I_{N,k}$ by $\lceil e^{2t}\rceil$ disjoint intervals of length $e^{-2(N+1)t}$. We collect such intervals of length $e^{-2(N+1)t}$ over $1\leq k\leq M_N$ and define $I_{N+1,1},\ldots I_{N+1,M_{N+1}}$ by reordering them. Note that any point in $[-\frac{1}{2},\frac{1}{2}]$ can be contained in at most $2$ intervals among $I_{N,1},\ldots,I_{N,M_{N}}$, hence among $I_{N+1,1},\ldots,I_{N+1,M_{N+1}}$. It follows from \eqref{alphatildeupperbound} and the inductive hypothesis that
\eq{\begin{aligned}
    \sum_{k=1}^{M_{N+1}} \widetilde{\alpha}_i(a_{(N+1)t\be_i}u_{i}(s_{N+1,k})x)& e^{-2(N+1)t}\leq 2\useconC{5}\sum_{k=1}^{M_N} \int_{I_{N,k}}\widetilde{\alpha}_i(a_{(N+1)t\be_i}u_{i}(s)x) ds\\&\leq 2\useconC{5}e^{-2Nt}\sum_{k=1}^{M_N} \int_{-\frac{1}{2}}^{\frac{1}{2}}\widetilde{\alpha}_i(a_{t\be_i}u_{i}(s)a_{Nt\be_i}u_{i}(s_{N,k})x) ds \\&\leq 2\useconC{5}\useconE{101} e^{-2Nt}\sum_{k=1}^{M_N}\widetilde{\alpha}_i(a_{Nt\be_i}u_{i}(s_{N,k})x) \\&\leq 2\useconC{5}\useconE{101} (3C_{5}\useconE{101})^{|\set{1,\ldots,N}\cap J_{i,t,\operatorname{mid}}(y)|t}e^{-\lambda^2 Nt}\widetilde{\alpha}_i(x)\\& \leq (3C_{5}\useconE{101})^{|\set{1,\ldots,N+1}\cap J_{i,t,\operatorname{mid}}(y)|t}e^{-\lambda^2 (N+1)t}\widetilde{\alpha}_i(x).
\end{aligned}}

\textbf{Case 3.} $N+1\in J_{i,t,\operatorname{high}}(y)$.\\
In this case we first observe that $\widetilde{\alpha}_i(a_{Nt\be_i}y)\geq \useconC{5}^{-t}\widetilde{\alpha}_i(a_{(N+1)t\be_i}y)\geq \useconC{5}^2\useconE{101}e^{t}$ holds by \eqref{alphatildeLipschitz}. Hence $N\in J_{i,t,\operatorname{mid}}(y)\cup J_{i,t,\operatorname{high}}(y)$ and the length of $I_{k,N}$ is $e^{-2Nt}$ for all $1\leq k\leq M_N$.

As in \textbf{Subcase 2-2}, we cover each $I_{N,k}$ by $\lceil e^{2t}\rceil$ disjoint intervals of length $e^{-2(N+1)t}$. We collect such intervals of length $e^{-2(N+1)t}$ over $1\leq k\leq M_N$ and define $I_{N+1,1},\ldots I_{N+1,M_{N+1}}$ by reordering them.

We may assume that for each $1\leq k\leq M_N$ there exists $s\in I_{N,k}$ such that 
$$\widetilde{\alpha}_i(a_{Nt\be_i}u_i(s)x)\geq \useconC{5}^{-1}\widetilde{\alpha}_i(a_{Nt\be_i}u_i(s)y)\geq \useconC{5} \useconE{101}e^{t},$$
whence $\widetilde{\alpha}_i(a_{Nt\be_i}u_i(s_{N,k})x)\geq \useconE{101}e^{t}$ for all $1\leq k\leq M_N$. It follows from \eqref{alphatildecont} that for all $1\leq k\leq M_N$
\eqlabel{eq:case3}{\int_{-\frac{1}{2}}^{\frac{1}{2}}\widetilde{\alpha}_i(a_{t\be_i}u_{i}(s)a_{Nt\be_i}u_{i}(s_{N,k})x) ds\leq 2e^{-\lambda^2 t}\widetilde{\alpha}_i(a_{Nt\be_i}u_{i}(s_{N,k})x).} We then deduce from \eqref{eq:case3} and the inductive hypothesis that
\eq{\begin{aligned}
    \sum_{k=1}^{M_{N+1}} \widetilde{\alpha}_i(a_{(N+1)t\be_i}u_{i}(s_{N+1,k})x)& e^{-2(N+1)t}\leq 2\useconC{5}\sum_{k=1}^{M_N} \int_{I_{N,k}}\widetilde{\alpha}_i(a_{(N+1)t\be_i}u_{i}(s)x) ds\\&\leq 2\useconC{5}e^{-2Nt}\sum_{k=1}^{M_N} \int_{-\frac{1}{2}}^{\frac{1}{2}}\widetilde{\alpha}_i(a_{t\be_i}u_{i}(s)a_{Nt\be_i}u_{i}(s_{N,k})x) ds \\&\leq 2\useconC{5}e^{-\lambda^2 t} e^{-2Nt}\sum_{k=1}^{M_N}\widetilde{\alpha}_i(a_{Nt\be_i}u_{i}(s_{N,k})x) \\&\leq 2\useconC{5}e^{-\lambda^2 t} (3C_{5}\useconE{101})^{|\set{1,\cdots,N}\cap J_{i,t,\operatorname{mid}}(y)|t}e^{-\lambda^2 Nt}\widetilde{\alpha}_i(x)\\& \leq (3C_{5}\useconE{101})^{|\set{1,\cdots,N+1}\cap J_{i,t,\operatorname{mid}}(y)|t}e^{-\lambda^2 (N+1)t}\widetilde{\alpha}_i(x).
\end{aligned}}
\end{proof}

The estimate in Proposition \ref{beta'tdecay} is not yet sufficient for our purposes due to the additional factor $(3C_{5}\useconE{101})^{|\set{1,\cdots,N}\cap J_{i,t,\operatorname{mid}}(y)|t}$. Although controlling this additional factor for each individual $t$ is challenging, we can address this issue by considering a sufficiently large number of distinct time scales.

Let $m(\lambda):=\lceil\frac{100\log (3C_{5}\useconE{101})}{\lambda^2-\lambda^3}\rceil$. We can find a sequence $t_1,\cdots,t_{m(\lambda)}\in t_\lambda\bN$ such that $t_1=t_\lambda$, $t_m|t_{m+1}$, and $e^{\frac{t_{m+1}}{10d}}>\useconC{5}^{t_m}\useconE{101}(t_m)$ for all $1\leq m\leq m(\lambda)-1$.

Let $N\in\bN$, $1\leq i\leq d-1$, $0<\eta\leq 1$, $0\leq l\leq l_\lambda$, and $y\in X$. We now define a proper lower semi-continuous function $\widetilde{\beta}_{N,i,\eta,l}(\;\cdot\;;y):X'\to[1,\infty)$ as follows:
\eq{\widetilde{\beta}_{N,i,\eta,l}(x;y)=\begin{cases}\widetilde{\alpha}_i(a_{Nt_{m(\lambda)}\be_i}x) & \text{if $x\in\displaystyle\bigcap_{m=1}^{m(\lambda)}\Upsilon_{N\frac{t_{m(\lambda)}}{t_m},i,t_m,\eta,l}(y)$,}\\ 0 & \text{otherwise}.\end{cases}}
Note that we clearly have
$$ \widetilde{\beta}_{N,i,\eta,l}(x;y)\leq \beta_{N\frac{t_{m(\lambda)}}{t_m},i,t_{l},l}'(x;y)$$
for any $1\leq m\leq m(\lambda)$ and $x\in X'$. The following proposition establishes the contraction hypothesis for $\widetilde{\beta}_{N,i,\eta,l}$.

\begin{prop}[Contraction for $\widetilde{\beta}$]\label{betatildedecay}
For any $N\in\bN$, $1\leq i\leq d-1$, $0<\eta\leq1$, $0\leq l\leq l_\lambda$, $y\in X$ and $x\in X'$,
\eqlabel{beta'cont}{\int_{-\frac{1}{2}}^{\frac{1}{2}}\widetilde{\beta}_{N,i,\eta,l}(u_{i}(s)x;y)ds\leq \widetilde{\sh}_\lambda e^{-\lambda^3 Nt_{m(\lambda)}}\widetilde{\alpha}_i(x).}
\end{prop}
\begin{proof}
    For $t\in t_\lambda\bN$ and $1\leq m\leq m(\lambda)$ let us first define
    $$ \widetilde{J}_l(y):=\set{j\in\bN: e^{\frac{t_m}{10d}}< \widetilde{\alpha}_i(a_{jt_\lambda \be_i}y)\leq \useconC{5}^{t_m}\sh_\lambda(t)}.$$
    Note that by the choice of $t_m$'s, the sets $\widetilde{J}_1(y),\cdots,\widetilde{J}_{m(\lambda)}(y)$ are disjoint and satisfy that for all $1\leq m\leq m(\lambda)$ and sufficiently large $N\in\bN$ 
    $$|\set{1,\ldots,N}\cap\widetilde{J}_l(y)|\geq \frac{t_{\lambda}}{100 t_m}\left|\set{1,\ldots,\tfrac{t_m}{t_\lambda}N }\cap J_{i,t_m,\operatorname{mid}}(y)\right|.$$ 
    Since $\widetilde{J}_1(y),\cdots,\widetilde{J}_{m(\lambda)}(y)$ are disjoint, there exists $m_0$ such that for sufficiently large $N\in\bN$
    $$ |\set{1,\cdots,N}\cap\widetilde{J}_{m_0}(y)|\leq\frac{N}{m(\lambda)},$$
    whence $|\set{1,\cdots,N}\cap J_{i,t_{m_0},\operatorname{mid}}(y)|<\frac{100N}{m(\lambda)}<\frac{\lambda^2-\lambda^3}{\log (3C_{5}\useconE{101})}N$. Applying Proposition \ref{beta'tcont} with $t=t_{m_0}$, we have
    \eq{\begin{aligned}
        \int_{-\frac{1}{2}}^{\frac{1}{2}}\beta_{N,i,t_{m_0},l}'(u_{i}(s)x;y)ds&\leq 2\useconC{5}\useconE{101} e^{t_{m_0}}(3C_{5}\useconE{101})^{|\set{1,\cdots,N}\cap J_{i,t_{m_0},\operatorname{mid}}(y)|t_{m_0}}e^{-\lambda^2 Nt_{m_0}}\widetilde{\alpha}_i(x)\\&\leq \widetilde{\sh}_\lambda e^{(\lambda^2-\lambda^3)Nt_{m_0}}e^{-\lambda^2 Nt_{m_0}}\widetilde{\alpha}_i(x)\\&\leq \widetilde{\sh}_\lambda e^{-\lambda^3Nt_{m_0}}\widetilde{\alpha}_i(x).
    \end{aligned}}
    Hence, we obtain
    $$\int_{-\frac{1}{2}}^{\frac{1}{2}}\widetilde{\beta}_{N,i,\eta,l}(u_{i}(s)x;y)ds\leq \int_{-\frac{1}{2}}^{\frac{1}{2}}\beta_{N\frac{t_{m(\lambda)}}{t_{m_0}},i,t_{m_0},l}'(u_{i}(s)x;y)ds\leq \widetilde{\sh}_\lambda e^{-\lambda^3Nt_{m(\lambda)}}\widetilde{\alpha}_i(x). $$
\end{proof}

As in Lemma \ref{betasupp} for $\beta_{dN,\del,i,\sh}$, we approximate $\widetilde{\beta}_{N,i,\eta,l}$ on certain expanding translates of horospherical orbits using characteristic functions over unions of specific boxes. For $N\in\bN$, $0<\eta\leq 1$ and $1\leq i\leq d-1$, let $\sB_{N,i,\eta}\subset\bR^{d-1}$ denote the box centered at $0$ with side length $2\eta e^{-2(N+1)t}$ along the direction of $\be_i$, and side lengths $2\eta e^{-2t}$ along the other $d-2$ directions. We also write $t=t_{m(\lambda)}$ for simplicity.

\begin{lem}[Approximations of $\widetilde{\beta}$]\label{betatildesupp:traj}
Let $N\in\bN$, $0<\eta\leq1$, $1\leq l\leq l_\lambda-1$, $y\in X$, and $x\in X'$. For each $1\leq i\leq d-1$ let us define $f_{i}:[-\frac{1}{2},\frac{1}{2}]^{d-1}\to\bR$ by $f_i(\bxi)=\widetilde{\beta}_{N,i,\eta,l}(a_{Nt\be_i}u(e^{-2Nt}\bxi)x;y)$. Then there exist $\bxi_{i,1},\ldots,\bxi_{i,M_i}\in[-\frac{1}{2},\frac{1}{2}]^{d-1}$ such that the followings holds.
\begin{enumerate}
    \item $0\leq f_i(\bxi)\leq \useconE{102}\kappa_i(x)^2\displaystyle\sum_{k=1}^{M_i} f_i(\bxi_{i,k})\mathds{1}_{\bxi_{i,k}+\sB_{N,i,\eta}}(\bxi)$ for any $\bxi\in[-\frac{1}{2},\frac{1}{2}]^{d-1},$
    \item $\displaystyle\sum_{k=1}^{M_i} f_i(\bxi_{i,k})\leq \useconE{102}\eta^{-(d-1)}e^{2(N+1)t}\kappa_i(x)^2\int_{[-\frac{1}{2},\frac{1}{2}]^{d-1}}\widetilde{\beta}_{N,i,2\eta,l-1}(a_{Nt\be_i}u(e^{-2Nt}\bxi)x;y)d\bxi.$
\end{enumerate}
\end{lem}

\begin{proof}
Let $\set{\bxi_{i,1},\ldots,\bxi_{i,M_i}}$ be a maximal $\sB_{N,i,\eta}$-set of $\Supp f_i$, i.e. $\set{\bxi_{i,k}}_{k=1}^{M_i}$ is a maximal set such that the boxes $\bxi_{i,k}+\sB_{N,i,\eta}$'s are disjoint. Then for any $\bxi\in\Supp f_i$, there exist $1\leq k\leq M_i$ such that $\bxi\in \bxi_{i,k}+\sB_{N,i,\eta}$. Then for any $1\leq m\leq m(\lambda)$ and $1\leq j\leq \frac{t}{t_m}N$, \eqlabel{eq:WeightedBoxExpansion}{a_{jt_m\be_i}a_{Nt\be_i}u\big(e^{-2dNt}(\bxi-\bxi_{i,k})\big)a_{-Nt\be_i}a_{-jt_m \be_i}\in a_{2Nt\be_i}u\big(e^{-2Nt}\sB_{N,i,\eta}\big)a_{-2Nt \be_i}\subseteq B^U(\operatorname{id},\eta e^{-2t}).}
We have $\bxi_{i,k}\in \Supp f_i$ for any $1\leq i\leq d-1$ and $1\leq k\leq M_i$. It follows that 
$$a_{Nt\be_i}u(e^{-2Nt}\bxi_{i,k})x\in\Upsilon_{\frac{t}{t_m}N,i,t_m,\eta,l}(y)$$ for all $1\leq m\leq m(\lambda)$, i.e. 
$$a_{jt_m\be_i}a_{Nt\be_i}u(e^{-2Nt}\bxi_{i,k})x\in \sB_{t_m,\eta}a_{jt_m\be_i}y \textrm{ for all } j\in J_{i,t_m,\operatorname{low}}(y)\cap\set{1,\ldots,\tfrac{t}{t_m}N},$$ $$\widetilde{\alpha}_i(a_{jt_m\be_i}a_{Nt\be_i}u(e^{-2Nt}\bxi_{i,k})x)>D^{(l)}\big((3\useconC{5})^t\useconE{101}\big) \textrm{ for all } j\in J_{i,t_m,\operatorname{high}}(y)\cap \set{1,\ldots, \tfrac{t}{t_m}N}.$$ 
Combining this with \eqref{eq:WeightedBoxExpansion} and Lemma \ref{compactness} we have
$$a_{jt_m\be_i}a_{Nt\be_i}u(e^{-2Nt}\bxi)x\in \sB_{t_m,2\eta}a_{jt_m\be_i}y\textrm{ for all } j\in J_{i,t_m,\operatorname{low}}(y)\cap\set{1,\ldots,\tfrac{t}{t_m}N},$$ $$\widetilde{\alpha}_i(a_{jt_m\be_i}a_{Nt\be_i}u(e^{-2Nt}\bxi)x)>D^{(l-1)}\big((3\useconC{5})^t\useconE{101}\big) \textrm{ for all } j\in J_{i,t_m,\operatorname{high}}(y)\cap \set{1,\ldots,\tfrac{t}{t_m}N}$$
for all $1\leq m\leq m(\lambda)$. Here, we are using $B^U(\operatorname{id},\eta e^{-2t})\sB_{t_m,\eta}\subseteq \sB_{t_m,2\eta}$ for $j\in J_{i,t_m,\operatorname{low}}(y)$. It follows that
$$u(\bxi)x\in\bigcap_{l=1}^{m(\lambda)}\Upsilon_{N\frac{t}{t_m},i,t_m,2\eta,l-1}(y),$$
hence $f_i(\bxi)=\widetilde{\beta}_{N,i,2\eta,l-1}(a_{Nt\be_i}u(e^{-2Nt}\bxi)x;y)>0$ for any $\bxi\in\bxi_{i,k}+\sB_{N,i,\eta}$. Moreover, using \eqref{alphatildeLipschitz} and \eqref{LipschitzOrthogonal} we find
\eqlabel{Lipschitz2:traj}{\widetilde{\alpha}_i(a_{Nt\be_i}u(e^{-2Nt}\bxi)x)\leq \useconE{102}\kappa_i(x)^{2}\widetilde{\alpha}_i(a_{Nt\be_i}u(e^{-2Nt}\bxi_{i,k})x),}
\eqlabel{Lipschitz2':traj}{\widetilde{\alpha}_i(a_{Nt\be_i}u(e^{-2Nt}\bxi_{i,k})x)\leq \useconE{102}\kappa_i(x)^{2}\widetilde{\alpha}_i(a_{Nt\be_i}u(e^{-2Nt}\bxi)x)}
for any $\bxi\in\bxi_{i,k}+\sB_{N,i,\eta}$. Hence we get
\eqlabel{fiupperbound:traj}{\begin{aligned}
    f_i(\bxi)&=\widetilde{\beta}_{N,i,\eta,l}(a_{Nt\be_i}u(e^{-2Nt}\bxi)x;y)\\&=\widetilde{\alpha}_i(a_{Nt\be_i}u(e^{-2Nt}\bxi)x)\\&=\useconE{102}\kappa_i(x)^{2}\widetilde{\alpha}_i(a_{Nt\be_i}u(e^{-2Nt}\bxi_{i,k})x)\\&\leq \useconE{102}\kappa_i(x)^2f_i(\bxi_{i,k})
\end{aligned}}
for each $\bxi\in\bxi_{i,k}+\sB_{N,i,\eta}$, so (1) is proved.

We now verify (2). As in \eqref{fiupperbound:traj}, from \eqref{Lipschitz2':traj} we have
\eq{f_i(\bxi_{i,k})=\widetilde{\beta}_{N,i,\eta,l}(a_{Nt\be_i}u(e^{-2Nt}\bxi_{i,k})x)\leq \useconE{102}\kappa_i(x)^2\widetilde{\beta}_{N,i,2\eta,l-1}(a_{Nt\be_i}u(e^{-2Nt}\bxi)x)}
for any $\bxi\in\bxi_{i,k}+\sB_{N,i,\eta}$. Note that the volume of $\sB_{N,i,\eta}$ is $(2\eta)^{d-1}e^{-2(N+1)t}$. It follows that
\eq{\begin{aligned}
(2\eta)^{d-1}e^{-2(N+1)t}\displaystyle\sum_{k=1}^{M_i} f_i(\bxi_{i,k})&=\displaystyle\sum_{k=1}^{M_i}\int_{\bxi_{i,k}+\sB_{N,i,\eta}}f_i(\bxi_{i,k})d\bxi\\&\leq \useconE{102}\kappa_i(x)^2\displaystyle\sum_{k=1}^{M_i}\int_{\bxi_{i,k}+\sB_{N,i,\eta}}\widetilde{\beta}_{N,i,2\eta,l-1}(a_{Nt\be_i}u(e^{-2Nt}\bxi)x)d\bxi\\
&\leq 2^{d-1}\useconE{102}\kappa_i(x)^2\int_{[-\frac{1}{2},\frac{1}{2}]^{d-1}}\widetilde{\beta}_{N,i,2\eta,l-1}(a_{Nt\be_i}u(e^{-2Nt}\bxi)x)d\bxi.\end{aligned}}
In the last line, we use a fact derived from the maximality: the sets $\set{\bxi_{i,k}+\sB_{N,i,\eta}}_{k=1}^{M_i}$ can overlap at most $2^{d-1}$ times. This completes the proof. 
\end{proof}

\subsection{Higher-rank dynamical height function for a reference trajectory}
In this subsection, we construct dynamical height functions for a given reference trajectory under higher-rank actions and establish the contraction hypothesis, which is analogous to Proposition \ref{nontrivbdd}. We define $\widetilde{\psi}_{N,\eta,l}(\;\cdot\;;y ):X\to [1,\infty]$ by
\eqlabel{psidef:traj}{\widetilde{\psi}_{N,\eta,l}(x;y):=\prod_{i=1}^{d-1}\widetilde{\beta}_{N,i,\eta,l}(x;y)}
for $N\in\bN$, $0<\eta\leq 1$, $0\leq l\leq l_\lambda$ and $y\in X$.

In the remainder of this section, we show the following contraction hypothesis for $\widetilde{\psi}_{N,\eta,l}$.

\begin{prop}[Contraction for $\widetilde{\psi}$]\label{nontrivbdd:traj}
For any $N\in\bN$, $0\leq l\leq l_\lambda$, $0<\eta\leq 2^{-l}$, $y\in X$, and $x\in X'$ we have
\eq{\int_{[-\frac{1}{2},\frac{1}{2}]^{d-1}}\widetilde{\psi}_{2^l N,\eta,l}(u(\bxi)x;y)d\bxi\leq \useconE{102}^{d-1}\big(\useconE{106}\eta^{-d}\kappa(x)^{8}\big)^le^{-\lambda^3(d-1)(2^l-1)Nt_{m(\lambda)}}\prod_{i=1}^{d-1}\widetilde{\alpha}_i(x).}
\end{prop}

The proof of Proposition \ref{nontrivbdd:traj} follows an analogous approach to the procedure outlined in $\S 3$ and $\S4.1$ for the proof of Proposition \ref{nontrivbdd}. However, we include the full proof here for completeness. We begin with the following approximation lemma, similar to Lemma~\ref{betasupp}.
\begin{lem}[Approximation of $\widetilde{\psi}$]\label{psisupp:traj}
For $N\in\bN$, $0<\eta\leq 1$, $1\leq l\leq l_\lambda$, $y\in X$, and $x\in X'$, let us define $F:[-\frac{1}{2},\frac{1}{2}]^{d-1}\to\bR$ by $F(\bxi)=\widetilde{\psi}_{N,\eta,l}(u(\bxi)x;y)$. Then there exist $\bxi_1,\ldots,\bxi_M\in[-\frac{1}{2},\frac{1}{2}]^{d-1}$ such that
\begin{enumerate}
    \item $0\leq F(\bxi)\leq \useconE{102}^{d-1}\kappa(x)^2\displaystyle\sum_{k=1}^{M} F(\bxi_k)\mathds{1}_{B(\bxi_k,e^{-2Nt})}(\bxi)$
    for any $\bxi\in[-\frac{1}{2},\frac{1}{2}]^{d-1}$,
    \item 
        $\displaystyle\sum_{k=1}^{M} F(\bxi_k)\leq 
    \useconE{102}^{d-1}\eta^{-2(d-1)}\kappa(x)^2e^{2(d-1)Nt}\int_{[-\frac{1}{2},\frac{1}{2}]^{d-1}}\widetilde{\psi}_{N,2\eta,l-1}(u(\bxi)x;y)d\bxi.$
\end{enumerate}
\end{lem}
\begin{proof}
Let $\set{\bxi_{1},\ldots,\bxi_{M}}\subset[-\frac{1}{2},\frac{1}{2}]^{d-1}$ be a maximal $\eta e^{-2(N+1)t}$-separated set of $\Supp F$. For any $\bxi\in\Supp F$ there exists $1\leq k\leq M_\eta$ such that $\bxi\in B(\bxi_k,\eta e^{-2(N+1)t})$. Then for any $1\leq i\leq d-1$, $1\leq m\leq m(\lambda)$, and $1\leq j\leq \frac{t}{t_m}N$, \eqlabel{eq:UnweightedBoxExpansion}{a_{jt_m\be_i}u(\bxi-\bxi_{i,k})a_{-jt_m \be_i}\in a_{Nt\be_i}u\big(B(\bxi_k,\eta e^{-2(N+1)t})\big)a_{-Nt \be_i}\subseteq B^U(\operatorname{id},\eta e^{-2t}).}
We have $\bxi_{k}\in \Supp F$ for any $1\leq k\leq M_i$. It follows that 
$$u(\bxi_{k})x\in\Upsilon_{\frac{t}{t_m}N,i,t_m,\eta,l}(y)$$ for all $1\leq m\leq m(\lambda)$, i.e. 
$$a_{jt_m\be_i}u(\bxi_{k})\in \sB_{t_m,\eta}a_{jt_m\be_i}y \textrm{ for all } j\in J_{i,t_m,\operatorname{low}}(y)\cap\set{1,\ldots,\tfrac{t}{t_m}N},$$ $$\widetilde{\alpha}_i(a_{jt_m\be_i}u(\bxi_{k})x)>D^{(l)}\big((3\useconC{5})^t\useconE{101}\big) \textrm{ for all } j\in J_{i,t_m,\operatorname{high}}(y)\cap \set{1,\ldots, \tfrac{t}{t_m}N}.$$ 
Combining this with \eqref{eq:UnweightedBoxExpansion} and Lemma \ref{compactness} we have
$$a_{jt_m\be_i}u(\bxi)x\in \sB_{t_m,2\eta}a_{jt_m\be_i}y\textrm{ for all } j\in J_{i,t_m,\operatorname{low}}(y)\cap\set{1,\ldots,\tfrac{t}{t_m}N},$$ $$\widetilde{\alpha}_i(a_{jt_m\be_i}u(\bxi)x)>D^{(l-1)}\big((3\useconC{5})^t\useconE{101}\big) \textrm{ for all } j\in J_{i,t_m,\operatorname{high}}(y)\cap \set{1,\ldots,\tfrac{t}{t_m}N}$$
for all $1\leq i\leq d-1$ and $1\leq m\leq m(\lambda)$. It follows that
$u(\bxi)x\in\Upsilon_{N\frac{t}{t_m},i,t_m,2\eta,l-1}(y)$ for all $1\leq i\leq d-1$ and $1\leq m\leq m(\lambda)$, hence 
$$F(\bxi)=\prod_{i=1}^{d-1}\widetilde{\beta}_{N,i,2\eta,l-1}(u(\bxi)x;y)>0$$ for any $\bxi\in B(\bxi_k,\eta e^{-2(N+1)t})$. Moreover, Using \eqref{alphatildeLipschitz} and \eqref{LipschitzOrthogonal} we have
\eqlabel{Lipschitz1:traj}{\begin{aligned}\widetilde{\alpha}_i(a_{Nt\be_i}u(\bxi)x)&=\widetilde{\alpha}_i\big((a_{Nt\be_i}u(\bxi-\bxi_k)a_{-Nt\be_i})a_{Nt\be_i}u(\bxi_k)x\big)\\&\leq \useconC{5}\alpha_i\big(\big(a_{Nt\be_i}u(\pi_i^{\perp}(\bxi-\bxi_k))a_{-Nt\be_i}\big)a_{Nt\be_i}u(\bxi_k)x\big)\\&\leq \useconC{5}\eta_{\lambda,t}^{-(d-1)}\kappa_i(a_{Nt\be_i}u(\bxi_k)x)^{2}\widetilde{\alpha}_i(a_{Nt\be_i}u(\bxi_k)x)\\&=\useconE{102}\kappa_i(x)^{2}\widetilde{\alpha}_i(a_{Nt\be_i}u(\bxi_k)x),\end{aligned}}
 and similarly, we also have
\eqlabel{Lipschitz1':traj}{\widetilde{\alpha}_i(a_{Nt\be_i}u(\bxi_k)x)\geq \useconE{102}\kappa_i(x)^{2}\widetilde{\alpha}_i(a_{Nt\be_i}u(\bxi)x).}
for all $1\leq i\leq d-1$. Thus, for any $\bxi\in B(\bxi_k,\eta e^{-2(N+1)t})$ we get
\eqlabel{Fupperbound1:traj}{\begin{aligned}F(\bxi)&=\widetilde{\psi}_{N,\eta,l}\big(u(\bxi)x\big)=\displaystyle\prod_{i=1}^{d-1}\widetilde{\alpha}_i(a_{Nt\be_i}u(\bxi)x)\\&\leq \useconE{102}^{d-1}\kappa(x)^2\displaystyle\prod_{i=1}^{d-1}\widetilde{\alpha}_i(a_{Nt\be_i}u(\bxi_k)x)\\&\leq \useconE{102}^{d-1}\kappa(x)^2\widetilde{\psi}_{N,\eta,l}\big(u(\bxi_k)x;y\big),\end{aligned}}
so (1) is proved.

We now verify (2). As in \eqref{Fupperbound1:traj}, from \eqref{Lipschitz1':traj} we have \eq{F(\bxi_k)=\widetilde{\psi}_{N,2\eta,l-1}(u(\bxi_k)x;y)\leq \useconE{102}^{d-1}\kappa(x)^2\widetilde{\psi}_{N,2\eta,l-1}(u(\bxi)x;y)} for any $\bxi\in B(\bxi_k,\eta e^{-2Nt})$. It follows that
\eq{\begin{aligned}
&(2\eta e^{-2Nt})^{d-1}\displaystyle\sum_{k=1}^{M} F(\bxi_k)=\displaystyle\sum_{k=1}^{M}\int_{B(\bxi_k,\eta e^{-2Nt})}F(\bxi_k)d\bxi\\&\leq \useconE{102}^{d-1}\kappa(x)^2\displaystyle\sum_{k=1}^{M}\int_{B(\bxi_k,\eta e^{-2Nt})}\widetilde{\psi}_{N,2\eta,l-1}(u(\bxi)x;y)d\bxi\\
&\leq 2^{d-1}\useconE{102}^{d-1}\kappa(x)^2\int_{[-\frac{1}{2},\frac{1}{2}]^{d-1}}\widetilde{\psi}_{N,2\eta,l-1}(u(\bxi)x;y)d\bxi.
\end{aligned}}
In the last line, we use a fact derived from the maximality of $\set{\bxi_1,\ldots,\bxi_M}$: the sets $\set{B(\bxi_k,\eta e^{-2(N+1)t})}_{k=1}^{M}$ can overlap at most $2^{d-1}$ times. This completes the proof. 
\end{proof}

The following proposition is parallel to Proposition \ref{multicont}.
\begin{prop}[Inductive contraction for $\widetilde{\psi}$]\label{multicont:traj}
For any $N\in\bN$, $0<\eta\leq \frac{1}{2}$, $1\leq l\leq l_\lambda$, $y\in X$, and $x\in X'$, we have
\eq{\int_{[-\frac{1}{2},\frac{1}{2}]^{d-1}}\widetilde{\psi}_{2N,\eta,l}(u(\bxi)x;y)d\bxi\leq \useconE{106}\eta^{-(d-1)}e^{-\lambda^3(d-1)Nt}\kappa(x)^8\int_{[-\frac{1}{2},\frac{1}{2}]^{d-1}}\widetilde{\psi}_{N,2\eta,l-1}(u(\bxi)x;y)d\bxi}
for some constant \newconE{106}$\useconE{106}>0$.
\end{prop}
\begin{proof}
 Let $F:[-\frac{1}{2},\frac{1}{2}]^{d-1}\to\bR$ be the function defined by $F(\bxi)=\widetilde{\psi}_{2N,\eta,l}(u(\bxi)x;y)$, and $\set{\bxi_1,\cdots,\bxi_M}$ be a maximal $\eta e^{-2Nt}$-separated set of $\Supp F$ as in Lemma \ref{psisupp:traj}. Note that
\eqlabel{SuppFcov:traj}{\Supp F\subseteq\bigcup_{k=1}^{M}B(\bxi_k,\eta e^{-2Nt}).}
We also observe that
$$ \widetilde{\psi}_{2N,\eta,l}(u(\bxi)x;y)=\prod_{i=1}^{d-1}\widetilde{\beta}_{2N,i,\eta,l}(u(\bxi)x;y) \leq \prod_{i=1}^{d-1}\widetilde{\beta}_{N,i,\eta,l}(a_{Nt\be_i}u(\bxi)x;a_{Nt\be_i}y).$$ 
Hence, we have
\eqlabel{psibetabdd1:traj}{\begin{aligned} &\int_{[-\frac{1}{2},\frac{1}{2}]^{d-1}}\widetilde{\psi}_{2N,\eta,l}(u(\bxi)x;y)d\bxi\leq\displaystyle\sum_{k=1}^{M}\int_{B(\bxi_k,\eta e^{-2Nt})}\widetilde{\psi}_{2N,\eta,l}(u(\bxi)x;y)d\bxi\\&\leq\displaystyle\sum_{k=1}^{M}\int_{B(\bxi_k,\eta e^{-2Nt})}\prod_{i=1}^{d-1}\widetilde{\beta}_{N,i,\eta,l}(a_{Nt\be_i}u(\bxi)x;a_{aNt\be_i}y)d\bxi\\
&=\displaystyle\sum_{k=1}^{M}\int_{[-\eta,\eta]^{d-1}}\prod_{i=1}^{d-1}\widetilde{\beta}_{N,i,\eta,l}\left(a_{Nt\be_i}u(e^{-2Nt}\bxi)u(\bxi_k)x; a_{Nt\be_i}y\right)d\bxi.\end{aligned}}

We now estimate the integral of the last line in \eqref{psibetabdd1:traj} for each fixed $1\leq k\leq M$. For each $1\leq i\leq d-1$ let us define $f_i:[-\frac{1}{2},\frac{1}{2}]^{d-1}\to\bR$ by 
$$f_i(\bxi)=\widetilde{\beta}_{N,i,\eta,l}(a_{Nt\be_i}u(e^{-2Nt}\bxi)u(\bxi_k)x;a_{Nt\be_i}y)$$ and let $\bxi_{i,1},\cdots,\bxi_{i,M_i}\in[-\frac{1}{2},\frac{1}{2}]^{d-1}$ be as in Lemma \ref{betatildesupp:traj}. Recall that we have 
$$0\leq f_{i}(\bxi)\leq \useconE{102}\kappa_i(x)^2\displaystyle\sum_{k=1}^{M_{i}} f_{i}(\bxi_{i,k})\mathds{1}_{\bxi_{i,k}+\sB_{N,i}}(\bxi)$$
by (1) of Lemma \ref{betatildesupp:traj}. It follows that
\eqlabel{betaprod1:traj}{\begin{aligned} &\int_{[-\frac{1}{2},\frac{1}{2}]^{d-1}}\prod_{i=1}^{d-1}\widetilde{\beta}_{N,i,\eta,l}\left(a_{Nt\be_i}u(e^{-2Nt}\bxi)u(\bxi_k)x; a_{Nt\be_i}y\right)d\bxi\\&=\int_{[-\frac{1}{2},\frac{1}{2}]^{d-1}}\prod_{i=1}^{d-1}f_i(\bxi)d\bxi\\&\leq \useconE{102}^{d-1}\kappa(x)^2\sum_{k_1=1}^{M_1}\cdots\sum_{k_{d-1}=1}^{M_{d-1}}\left(\prod_{i=1}^{d-1}f_{i}(\bxi_{i,k_i})\right)\int_{[-\frac{1}{2},\frac{1}{2}]^{d-1}}\prod_{i=1}^{d-1}\mathds{1}_{\bxi_{k_i,i}+\sB_{N,i}}(\bxi)d\bxi.\end{aligned}}
We observe that for any $k_1,\ldots,k_{d-1}$, $\displaystyle\bigcap_{i=1}^{d-1}\left(\bxi_{k_i,i}+\sB_{N,i}\right)$ is always contained in a box with sidelength $2e^{-2(N+1)t}$ since each $\sB_{N,i}$ has sidelength $2e^{-2(N+1)t}$ along the direction of $\be_i$. We also have 
$$\displaystyle\sum_{k_i=1}^{M_i} f_i(\bxi_{i,k_i})\leq \useconE{102}\kappa_i(x)^2\int_{[-\frac{1}{2},\frac{1}{2}]^{d-1}}\widetilde{\beta}_{N,i,2\eta,l-1}\left(a_{Nt\be_i}u(e^{-2Nt}\bxi)u(\bxi_k)x; a_{Nt\be_i}y\right)d\bxi$$
for all $1\leq i\leq d-1$ by (2) of Lemma \ref{betatildesupp:traj}.
It follows that \eqref{betaprod1:traj} is bounded by
\eqlabel{betaprod2:traj}{\begin{aligned} &\leq \useconE{102}^{d-1}\kappa(x)^{2(d-1)}(2e^{-2Nt})^{d-1}\displaystyle\sum_{k_1=1}^{M_1}\cdots\displaystyle\sum_{k_{d-1}=1}^{M_{d-1}}\prod_{i=1}^{d-1}f_{i}(\bxi_{i,k_i})\\ 
&\leq 2^{d-1}\useconE{102}^{2(d-1)}\kappa(x)^4\prod_{i=1}^{d-1}\int_{[-\frac{1}{2},\frac{1}{2}]^{d-1}}\widetilde{\beta}_{N,i,2\eta,l-1}\left(a_{Nt\be_i}u(e^{-2Nt}\bxi)u(\bxi_k)x; a_{Nt\be_i}y\right)d\bxi. \end{aligned}}

Recall that by Proposition \ref{betatildedecay} we have
\eq{\int_{-\frac{1}{2}}^{\frac{1}{2}}\widetilde{\beta}_{N,i,2\eta,l-1}(u_{i}(s)x; y)ds\leq \widetilde{\sh}_\lambda e^{-\lambda^3Nt}\alpha_i(x)}
for any $1\leq i\leq d-1$, $0<\eta\leq \frac{1}{2}$, $\sh>\widetilde{\sh}_\lambda$, $y\in X$ and $x\in X'$. It follows that
\eq{\begin{aligned}&\int_{[-\frac{1}{2},\frac{1}{2}]^{d-1}}\widetilde{\beta}_{N,i,2\eta,l-1}\left(a_{Nt\be_i}u(e^{-2Nt}\bxi)u(\bxi_k)x; a_{Nt\be_i}y\right)d\bxi\\&=\int_{[-\frac{1}{2},\frac{1}{2}]^{d-1}}\int_{-\frac{1}{2}}^{\frac{1}{2}}\widetilde{\beta}_{N,i,2\eta,l-1}\left(a_{Nt\be_i}u_i(e^{-2Nt}s)u(e^{-2Nt}\pi_i^{\perp}(\bxi))u(\bxi_k)x; a_{Nt\be_i}y\right)dsd\bxi \\&\leq \int_{[-\frac{1}{2},\frac{1}{2}]^{d-1}}\int_{-\frac{1}{2}}^{\frac{1}{2}}\widetilde{\beta}_{N,i,2\eta,l-1}\left(u_i(s)a_{Nt\be_i}u(e^{-2Nt}\pi_i^{\perp}(\bxi))u(\bxi_k)x; a_{Nt\be_i}y\right)dsd\bxi\\ &\leq \widetilde{\sh}_\lambda e^{-\lambda^3Nt}\int_{[-\frac{1}{2},\frac{1}{2}]^{d-1}}\widetilde{\alpha}_i(a_{Nt\be_i}u(e^{-2Nt}\pi_i^{\perp}(\bxi))u(\bxi_k)x; a_{Nt\be_i}y)d\bxi.\end{aligned}}
for any $1\leq i\leq d-1$. By \eqref{LipschitzOrthogonal} one has
\eq{\widetilde{\alpha}_i(a_{Nt\be_i}u(e^{-2Nt}\pi_i^{\perp}(\bxi))u(\bxi_k)x)\leq \eta_{\lambda,t}^{-(d-1)}\kappa_i(x)^{2}\widetilde{\alpha}_i(a_{Nt\be_i}u(\bxi_k)x),}
hence for any $1\leq i\leq d-1$
\eqlabel{ithicken:traj}{\begin{aligned}\int_{[-\frac{1}{2},\frac{1}{2}]^{d-1}}\widetilde{\beta}_{N,i,2\eta,l-1}&\left(a_{Nt\be_i}u(e^{-2Nt}\bxi)u(\bxi_k)x; a_{Nt\be_i}y\right)d\bxi\\ &\leq \eps_{\lambda,t}^{-(d-1)}\widetilde{\sh}_\lambda e^{-\lambda^3d Nt}\kappa_i(x)^{2}\widetilde{\alpha}_i(a_{Nt\be_i}u(\bxi_k)x).\end{aligned}}
Combining \eqref{betaprod1:traj}, \eqref{betaprod2:traj}, and \eqref{ithicken:traj} together, we get
\eqlabel{betaprod3:traj}{\begin{aligned}\int_{[-\frac{1}{2},\frac{1}{2}]^{d-1}}\prod_{i=1}^{d-1}\widetilde{\beta}_{N,i,\eta,l}&\left(a_{Nt\be_i}u(e^{-2Nt}\bxi)u(\bxi_k)x; a_{Nt\be_i}y\right)d\bxi\\&\leq \useconE{105} e^{-\lambda^3(d-1)Nt}\kappa(x)^6\prod_{i=1}^{d-1}\widetilde{\alpha}_i(a_{Nt\be_i}u(\bxi_k)x),\end{aligned}}
where \newconE{105}$\useconE{105}:=2^{d-1}\useconE{102}^{3(d-1)}\eps_{\lambda,t}^{-(d-1)}\widetilde{\sh}_\lambda$.

Since $\bxi_k\in \Supp F$ we have $\widetilde{\psi}_{N,\eta,l}(u(\bxi_k)x;y)>0$, hence
\eqlabel{tauiposi:traj}{\widetilde{\psi}_{N,\eta,l}(u(\bxi_k)x;y)=\prod_{i=1}^{d-1}\widetilde{\alpha}_i\left(a_{Nt\be_i}u(\bxi_k)x\right).}
Combining \eqref{psibetabdd1:traj}, \eqref{betaprod3:traj}, and \eqref{tauiposi:traj}, 
\eq{\int_{[-\frac{1}{2},\frac{1}{2}]^{d-1}} \widetilde{\psi}_{2N,\eta,l}(u(\bxi)x;y)d\bxi\leq\displaystyle\sum_{k=1}^{M}e^{-2(d-1)Nt} \useconE{105} e^{-\lambda^3(d-1)Nt}\kappa(x)^6\widetilde{\psi}_{N,\eta,l}(u(\bxi_k)x)} 
holds. It follows that
\eqlabel{psibdd:traj}{\int_{[-\frac{1}{2},\frac{1}{2}]^{d-1}} \widetilde{\psi}_{2N,\eta,l}(u(\bxi)x;y)d\bxi\leq \useconE{105}e^{-\lambda^3(d-1)Nt}e^{-2(d-1)Nt}\kappa(x)^6\displaystyle\sum_{k=1}^{M}F(\bxi_k).}
On the other hand, by (2) of Lemma \ref{psisupp:traj} we have
\eqlabel{Fbdd:traj}{\eta^{d-1}e^{-2(d-1)Nt}\displaystyle\sum_{k=1}^{M} F(\bxi_k)\leq 
    \useconE{102}^{d-1}\kappa(x)^2\int_{[-\frac{1}{2},\frac{1}{2}]^{d-1}}\widetilde{\psi}_{N,2\eta,l-1}(u(\bxi)x;y)d\bxi.}
Therefore, combining \eqref{psibdd:traj} and \eqref{Fbdd:traj}, we obtain the desired inequality for $\useconE{106}:=\useconE{102}^{d-1}\useconE{105}$.
\end{proof}

In addition to Proposition \ref{multicont:traj} we also need the following trivial estimate.
\begin{lem}[Trivial bound for $\widetilde{\psi}$]\label{trivlem:traj}
For any $N\in\bN$, $0<\eta\leq 1$, $0\leq l\leq l_\lambda$, $y\in X$, and $x\in X'$ we have
\eqlabel{trivbdd:traj}{\int_{[-\frac{1}{2},\frac{1}{2}]^{d-1}}\widetilde{\psi}_{N,\eta,l}(u(\bxi)x;y)d\bxi\leq \useconE{102}^{d-1}\kappa(x)^2 \prod_{i=1}^{d-1}\widetilde{\alpha}_i(x).}
\end{lem}
\begin{proof}
For any $\bxi\in[-\frac{1}{2},\frac{1}{2}]^{d-1}$ and $1\leq i\leq d-1$ we have 
\eq{\widetilde{\alpha}_i(u(\bxi)x)\leq \useconC{5}\eps_{\lambda,t_\lambda}^{-(d-1)}\kappa_i(x)^2\widetilde{\alpha}_i(x)=\useconE{102}\kappa_i(x)^2\widetilde{\alpha}_i(x)}
by \eqref{alphatildeLipschitz} and \eqref{LipschitzOrthogonal}. It follows that for any $\bxi\in[-\frac{1}{2},\frac{1}{2}]^{d-1}$
\eq{\begin{aligned}
    \widetilde{\psi}_{N,\eta,l}(u(\bxi)x;y)&\leq \prod_{i=1}^{d-1}\widetilde{\alpha}_i(u(\bxi)x)\leq \useconE{102}^{d-1}\kappa(x)^2\prod_{i=1}^{d-1}\widetilde{\alpha}_i(x),
\end{aligned}}
hence we obtain \eqref{trivbdd:traj}.
\end{proof}

\begin{proof}[Proof of Proposition \ref{nontrivbdd:traj}]
For the sake of simplicity let us write
\eq{\mathsf{Q}_j:=\int_{[-\frac{1}{2},\frac{1}{2}]^{d-1}}\widetilde{\psi}_{2^{l-j}N,2^{j}\eta,l-j}(u(\bxi)x;y)d\bxi,}
where $0\leq j\leq l$. We apply Proposition \ref{multicont:traj} with $2^{l-j-1}N\in\bN$ and $2^j\eta$, where $0\leq j\leq l-1$. Then we have
\eqlabel{jindiv:traj}{\mathsf{Q}_j\leq \useconE{106}(2^j\eta)^{-(d-1)}e^{-\lambda^3(d-1)2^{l-j-1}Nt}\kappa(x)^8\mathsf{Q}_{j+1}}
for any $0\leq j\leq l-1$. By Lemma \ref{trivlem:traj} we also have
\eqlabel{trivbdd':traj}{\mathsf{Q}_{l}\leq \useconE{102}^{d-1}\kappa(x)^2 \prod_{i=1}^{d-1}\widetilde{\alpha}_i(x).}
Multiplying all the inequalities \eqref{jindiv:traj} for $0\leq j\leq l-1$ and \eqref{trivbdd':traj}, we conclude
\eq{\mathsf{Q}_0\leq \useconE{102}^{d-1}\big(\useconE{106}\eta^{-d}\kappa(x)^{8}\big)^le^{-\lambda^3(d-1)(2^l-1)Nt}\prod_{i=1}^{d-1}\widetilde{\alpha}_i(x).}
\end{proof}
 
\section{The set of exceptions to the inhomogeneous Littlewood conjecture}

\subsection{The space of grids and Dani correspondence}
In this subsection, we introduce the space of grids and give a dynamical reformulation of the inhomogeneous Diophantine approximation problem.

Let us denote by $\widehat{G}=\operatorname{SL}_d(\bR)\ltimes\bR^d$ the group of volume-preserving affine transformations and denote by $\widehat{\Gamma}=\operatorname{SL}_d(\bZ)\ltimes\bZ^d=\operatorname{Stab}_{\widehat{G}}(\bZ^d)$ the stabilizer of the standard lattice $\bZ^d$. We view $\widehat{G}$ as a subgroup of $\operatorname{SL}_{d+1}(\bR)$ by $\widehat{G}=\set{\left(\begin{matrix}
    g & v \\ 0 & 1
\end{matrix}\right): g\in G, v\in\bR^d}$. We also view $G$ as a subgroup of $\widehat{G}$ and take a lift of the element $g\in G$ to $\widehat{G}\subset \operatorname{SL}_{d+1}(\bR)$ by $g\mapsto \left(\begin{matrix}
    g & 0 \\ 0 & 1
\end{matrix}\right)$, denoted again by $g$. 

Let $\widehat{X}=\widehat{G}/\widehat{\Gamma}$ and denote by $\pi:\widehat{X}\to X$ the canonical projection from $\widehat{X}$ to $X$. One can view the homogeneous space $\widehat{X}$ as the space of unimodular \textit{grids} in $\bR^d$, i.e. unimodular lattices translated by a vector in $\bR^d$. More explicitly, for $\widehat{x}=\left(\begin{matrix}
    g & v \\ 0 & 1
\end{matrix}\right)\widehat{\Gamma}\in \widehat{X}$ with $g\in G$ and $v\in\bR^d$, we identify $\widehat{x}$ to the corresponding unimodular grid $\Lambda_{\widehat{x}}:=g\bZ^d+v$ in $\bR^d$.

For $\bxi,\btheta\in\bR^{d-1}$ we denote $$x_{\bxi}:= \left(\begin{matrix}
    \operatorname{Id}_{d-1} & \bxi \\ & 1\end{matrix}\right)\Gamma\in X,\qquad\widehat{x}_{\bxi,\btheta}:=\left(\begin{matrix}
    \operatorname{Id}_{d-1} & \bxi & -\btheta \\ & 1 & 0 \\ & & 1
\end{matrix}\right)\widehat{\Gamma}\in \widehat{X}.$$ Note that $\Lambda_{\widehat{x}_{\bxi,\btheta}}=u(\bxi)\bZ^d-(\btheta,0)$. As in the classical Dani correspondence for homogeneous Diophantine approximation, inhomogeneous Diophantine properties of $(\bxi,\btheta)$ are characterized by dynamical properties of the orbits of $\widehat{x}_{\bxi,\btheta}$ in the space $\widehat{X}$. 

We say that a pair $(\bxi,\btheta)\in \bR^{d-1}\times\bR^{d-1}$ is \textit{rational} if there exist some $\bp\in\bZ^{d-1}$ and $q\in\bZ$ such that $q\bxi-\btheta+\bp=0$, and \textit{irrational} otherwise. Let us define
$A^+_{\ge T}:=\set{a_{(\tau_1,\ldots,\tau_{d-1})}\in A: \tau_1,\ldots,\tau_{d-1}\ge T}$ for $T>0$, and
$\cL_\eps:=\set{\widehat{x}\in\widehat{X}: \Lambda_{\widehat{x}}\cap B^{\bR^d}(0,\eps)= \emptyset}$ for $\eps>0$.

\begin{lem}\label{InhomDani}
Suppose that $\bxi\in\bR^{d-1}$ satisfies Littlewood's conjecture, i.e. \eqref{eq:Littlewood} holds. Let $\btheta\in\bR^{d-1}$. If $\displaystyle\liminf_{q\to\infty}q\prod_{i=1}^{d-1}\|q\xi_i-\theta_i\|_\bZ>0$, then there exists $T>0$ such that $A^+_{\ge T}\widehat{x}_{\bxi,\btheta}\subseteq \cL_\eps$.
\end{lem}
\begin{proof}
We first consider the case of $(\bxi,\btheta)$ is rational. Then there exist $\bp_0\in\bZ^{d-1}$ and $q_0\in\bZ$ such that $q_0\bxi-\btheta+\bp=0$. It follows that
$$\liminf_{q\to\infty}q\prod_{i=1}^{d-1}\|q\xi_i-\theta_i\|_\bZ=\liminf_{q\to\infty}q\prod_{i=1}^{d-1}\|(q-q_0)\xi_i\|_\bZ=\liminf_{q\to\infty}q\prod_{i=1}^{d-1}\|q\xi_i\|_\bZ=0,$$
as we are assuming that $\bxi\in\bR^{d-1}$ satisfies Littlewood's conjecture.

We now suppose that $(\bxi,\btheta)$ is irrational, and for any $\eps>0$ and $T>0$ there exists $\btau\in A_{\ge T}^{+}$ such that $a_{\btau}\widehat{x}_{\bxi,\btheta}\notin \cL_\eps$. By a matrix calculation, it implies that there exists a vector $v\in\bR^{d}$ with $\|v\|_\infty\leq\eps$ of the form 
$$v=\big(e^{\tau_1}(q\xi_1-\theta_1-p_1),\ldots,e^{\tau_{d-1}}(q\xi_{d-1}-\theta_{d-1}-p_{d-1}), e^{-(\tau_{1}+\cdots+\tau_{d-1})}q)\big),$$
where $(p_1,\ldots,p_{d-1})\in\bZ^{d-1}$ and $q\in\bZ$. It follows that
$\|q\bxi_i-\theta_i\|\leq \eps e^{-\tau_i}$ for $1\leq i\leq d-1$ and $|q|\leq \eps e^{\tau_1+\cdots+\tau_{d-1}}$, hence $q\prod_{i=1}^{d-1}\|q\xi_i-\theta_i\|_\bZ\leq \eps^d$. Since $(\bxi,\btheta)$ is irrational, the $q$ satisfying $\|q\bxi-\btheta\|_\bZ\leq \eps e^{-T}$ is unbounded as $T\to\infty$. Thus we get $\liminf_{q\to\infty}q\prod_{i=1}^{d-1}\|q\xi_i-\theta_i\|_\bZ\leq\eps^d$, and conclude that $\liminf_{q\to\infty}q\prod_{i=1}^{d-1}\|q\xi_i-\theta_i\|_\bZ=0$ as $\eps$ is arbitrary.

\end{proof}

\subsection{Measure classification for $A$-invariant measures}
The following measure classification theorem is a special case of \cite[Theorem 1.3]{EL18}.
\begin{thm}\cite[Theorem 1.3]{EL18}\label{measureclassification}
Let $\mu$ be an $A$-invariant and ergodic measure on $\widehat{X}$ and $\overline{\mu}=\pi_*\mu$. If $h_{\overline{\mu}}(a)>0$ for some $a\in A$, then $\mu$ is homogeneous.\end{thm}

For any $d\ge 2$ and $k|d$, we denote
$$L_{d,k}:=\left(\prod_{i=1}^{k}\operatorname{GL}_{d/k}(\bR)\right)\cap \operatorname{SL}_d(\bR).$$

\begin{lem}\cite[\S 6]{LW01}\label{SLdescription}
Let $\overline{\mu}$ be an $A$-invariant homogeneous measure on $X$ with $h_{\overline{\mu}}(a)>0$ for some $a\in A$. Then $\overline{\mu}$ is supported on a single orbit $L_{d,k}x$ for some $k|d$ and $k<d$. In particular, $Lx$ is non-compact.\end{lem}

\begin{prop}\label{Lesupp}
Let $\mu$ be an $A$-invariant homogeneous measure on $\widehat{X}$ with $h_{\overline{\mu}}(a)>0$ for some $a\in A$. Then $\Supp\mu$ is not contained in $\cL_\eps$ for any $\eps>0$.\end{prop}
\begin{proof}By Theorem \ref{measureclassification} and Lemma \ref{SLdescription}, $\mu$ is supported on a single orbit $Lx$ such that $\pi(L)=L_{d,k}$ for some $k|d$ and $k<d$. Furthermore, \cite[\S 5]{EL18} classified which homogeneous measures may occur in the positive base entropy case. Either $L=L_{d,k}\ltimes \bR^d$ or $\pi$ is injective on $L$.

If $L=L_{d,k}\ltimes \bR^d$, then $\mu$ is supported on $\pi^{-1}(L_{d,k}x)$. In this case, $\Supp\mu$ is not contained in $\cL_\eps$ for any $\eps>0$, since any points in $L_{d,k}x$ are contained in $\pi^{-1}(L_{d,k}x)$ but not in $\cL_\eps$. 

If $\pi$ is injective on $L$ then we claim that
$$\Supp\mu=\set{\left(\begin{matrix}
    g & gv \\ 0 & 1
\end{matrix}\right)\widehat{\Gamma}\in \widehat{X}: g\Gamma\in Lx, v\in \frac{1}{q}\bZ_{\operatorname{prim}}^d}$$
for some $q\in\bN$. Indeed, for any $g\Gamma\in \Supp \overline{\mu}$ the fiber $\pi^{-1}(g\Gamma)$ is identified with the torus $\bT^d$, so we may consider $\Supp \mu \cap \pi^{-1}(g\Gamma)$ as a subset of the torus. Then $\Supp \mu \cap \pi^{-1}(g\Gamma)$ is a $gL_{d,k}g^{-1}\cap \Gamma$-invariant subset of the torus, hence it is the whole torus or a finite set of rational points. Since $\pi$ is injective, the set $\Supp \mu \cap \pi^{-1}(g\Gamma)$ is a finite set of rational points. Thus the claim is proved. 

Since $Lx$ is non-compact in $X$, for any $\eps>0$ there exists $g_0\Gamma \in Lx$ such that $g_0\bZ^d$ contains non-zero vector of length less than $\frac{\eps}{2}$. It follows that there exists $\bm \in\bZ^d_{\operatorname{prim}}$ with $\|g_0\bm\|<\frac{\eps}{2}$. Observe that we have
$$ \left(\begin{matrix}
    g_0 & \frac{1}{q}g_0\bm \\ 0 & 1
\end{matrix}\right)\widehat{\Gamma}\in \Supp \mu$$
but the corresponding unimodular grid $g_0(\bZ^d+\frac{\bm}{q})$ contains a vector $g_0(\bm+\frac{\bm}{q})$ of length less than $(1+\frac{1}{q})\|g_0\bm\|<\eps$. We thus conclude that $\Supp\mu$ is not contained in $\cL_\eps$ for any $\eps>0$.
\end{proof}

\subsection{Partitions of $X$}
In this subsection we construct a finite partition $\cQ$ of $X$, which will be useful in the later entropy computation. Let $0<\lambda<1$ be given. Recall that $l_\lambda\in\bN$ is chosen so that $1-2^{-l_\lambda+1}>\lambda$. We set $l=l_\lambda$, $\eta=2^{-l_\lambda}$. For each $1\leq i \leq d-1$, $t\in t_\lambda\bN$ we can find a finite partition 
$$\cP_{i,t}=\set{P_{i,t,1},\ldots,P_{i,t,M_i(t)},\widetilde{\alpha}_i^{-1}(e^{\frac{t}{3d}},\sh_\lambda(t)], \widetilde{\alpha}_i^{-1}(\sh_\lambda(t),\infty)}$$ of $X$ such that for $1\leq k\leq M_i(t)$ each $P_{i,t,k}$ has diameter less than $\eta e^{-2t}=2^{-l_\lambda}e^{-2t}$.

Let us define a finite partition $\cQ$ of $X$ by
$$\cQ:=\bigvee_{i=1}^{d-1}\bigvee_{m=1}^{m(\lambda)}\bigvee_{j=1}^{\frac{t_{m(\lambda)}}{t_m}}a_{-jt_m\be_i}\cP_{i,t_m}.$$
For $N\in\bN$ and $1\leq i\leq d-1$ we denote
$$\cQ^{(N,i)}:=\bigvee_{j=1}^{N}a_{-jt_{m(\lambda)}\be_i}\cQ, \qquad \cQ^{(N)}:=\bigvee_{i=1}^{d-1}\cQ^{(N,i)}.$$

For any partition $\cP$ of $X$ and $y\in X$ we denote by $[y]_{\cP}$ the atom of $\cP$ containing $y$. The following lemma relates the atoms of the partition $\cQ^{(N)}$ to the supports of the dynamical height functions for higher rank actions constructed in Section $5$.
\begin{lem}\label{eq:relationpsi}
    Let $N\in\bN$, $x\in X'$, $y\in X$, and $\bxi\in [-\frac{1}{2},\frac{1}{2}]^{d-1}$. If $u(\bxi)x\in [y]_{\cQ^{(N)}}$ then $\widetilde{\psi}_{N,2^{-l_\lambda},l_\lambda}(u(\bxi)x,y)\geq 1$.
\end{lem}
\begin{proof}
    If $u(\bxi)x\in [y]_{\cQ^{(N)}}$ then for any $1\leq i \leq d-1$ and $1\leq j\leq N$ two points $a_{jt_{m(\lambda)}\be_i}u(\bxi)x$ and $a_{jt_{m(\lambda)}be_i}y$ are in the same atom of $\cQ$. It follows that for any $1\leq i \leq d-1$, $1\leq m\leq m(\lambda)$, and $1\leq j\leq \frac{t_{m(\lambda)}}{t_m}N$ two points $a_{jt_m\be_i}u(\bxi)x$ and $a_{jt_m\be_i}y$ are in the same atom of $\cP_{i,t_m}$.

    We now show that $u(\bxi)x\in \Upsilon_{\frac{t_{m(\lambda)}}{t_m}N,i,t_m,2^{-l_\lambda},l_\lambda}(y)$ for any $1\leq m\leq m(\lambda)$. 
    
    If $j\in J_{i,t_m,\operatorname{high}}(y)\cap\set{1,\ldots,\frac{t_{m(\lambda)}}{t_m}N}$ then $a_{jt_m\be_i}y\in \widetilde{\alpha}_i^{-1}(\sh_\lambda(t),\infty)$. Since $a_{jt_m\be_i}u(\bxi)x$ and $a_{jt_m\be_i}y$ are in the same atom of $\cP_{i,t_m}$, we have $\widetilde{\alpha}_i(a_{jt_m\be_i}u(\bxi)x)>\sh_\lambda(t)>D^{(l_\lambda)}\big((3\useconC{5})^t\useconE{101}\big)$. 
    
    If $j\in J_{i,t_m,\operatorname{low}}(y)\cap\set{1,\ldots,\frac{t_{m(\lambda)}}{t_m}N}$ then $a_{jt_m\be_i}y\in P_{i,t_m,k}$ for some $1\leq k\leq M_i(t)$. Then $a_{jt_m\be_i}u(\bxi)x$ is also in $P_{i,t_m,k}$, and the diameter of $P_{i,t_m,k}$ is less than $2^{-l_\lambda}e^{-2t}$. It implies that $a_{jt_m\be_i}u(\bxi)x\in \sB_{t,2^{-l_\lambda}}a_{jt_m\be_i}y$. 
    
    Hence, we verified that $u(\bxi)x\in \Upsilon_{\frac{t_{m(\lambda)}}{t_m}N,i,t_m,2^{-l_\lambda},l_\lambda}(y)$ for any $1\leq m\leq m(\lambda)$. It follows from the definition of $\widetilde{\beta}_{N,i,2^{-l_\lambda},l_\lambda}$ that $\widetilde{\beta}_{N,i,2^{-l_\lambda},l_\lambda}(u(\bxi)x,y)\geq 1$ for all $1\leq i\leq d-1$. We thus conclude that
    $$\widetilde{\psi}_{N,2^{-l_\lambda},l_\lambda}(u(\bxi)x,y)\geq \prod_{i=1}^{d-1}\widetilde{\beta}_{N,i,2^{-l_\lambda},l_\lambda}(u(\bxi)x,y)\geq1.$$
\end{proof}

\subsection{Construction of $A$-invariant measure}

In this subsection, we prove Theorem \ref{InhomLittlewood}. We argue by contradiction, so we suppose that $\dim_H \Xi>\frac{d-1}{2}$ throughout this subsection.

Let us denote
$$\cY:=\set{\bxi\in[-\tfrac{1}{2},\tfrac{1}{2}]^{d-1}: \liminf_{q\to\infty}q\prod_{i=1}^{d-1}\|q\xi_i\|_\bZ=0},$$
$$\cW:=\set{\bxi\in[-\tfrac{1}{2},\tfrac{1}{2}]^{d-1}: u(\bxi)\Gamma\textrm{ is }A^+\textrm{-divergent on average}}.$$
Recall that we have $\dim_H\cY=0$ by \cite[Theorem 1.5]{EKL06} and $\dim_H\cW\leq\frac{d-1}{2}$ by Theorem \ref{dimupperbdd}. Let us choose and fix $0<\lambda<1$ sufficiently close to $1$ so that
$$\dim_H\left(\Xi\setminus(\cY\cup\cW)\right)=\dim_H \Xi >(2-\lambda^6)\frac{d-1}{2},$$ 
and set $t=2^{l_\lambda-1}t_{m(\lambda)}$, where $l_\lambda, m(\lambda)\in\bN$ are as defined in Section 5.

For any compact set $K\subset X$, $\del>0$, and $m\in\bN$ let us define
\eq{\cZ_{K,\del}:=\set{\bxi\in[-\tfrac{1}{2},\tfrac{1}{2}]^{d-1}: \displaystyle\limsup_{N\to\infty}\frac{1}{N^{d-1}}\#\set{\btau\in\{t,\ldots,Nt\}^{d-1}:a_{\btau}u(\bxi)\Gamma\in K}\ge \del},}
\eq{\cZ_{K,\del,m}:=\set{\bxi\in[-\tfrac{1}{2},\tfrac{1}{2}]^{d-1}: \frac{1}{m^{d-1}}\#\set{\btau\in\{t,\ldots,mt\}^{d-1}:a_{\btau}u(\bxi)\Gamma\in K}\ge \del}.}
Then we may write
\eqlabel{eq:DivAvgPartition}{[-\tfrac{1}{2},\tfrac{1}{2}]^{d-1} \setminus\cW=\bigcup_{K\subset X}\bigcup_{\del>0} \cZ_{K,\del}.}
For any $\eps>0$ and $T\in\bN$ we also define 
\eq{\Xi_\eps:=\set{\bxi\in[-\tfrac{1}{2},\tfrac{1}{2}]^{d-1}: \textrm{ there exists } \btheta\in\mathbb{R}^{d-1} \textrm{ such that } q\prod_{i=1}^{d-1}\|q\xi_i-\theta_i\|_\bZ\ge \eps},}
\eq{\Xi_{\eps,T}:=\set{\bxi\in[-\tfrac{1}{2},\tfrac{1}{2}]^{d-1}: \textrm{ there exists } \btheta\in\bR^{d-1} \textrm{ such that }a\widehat{x}_{\bxi,\btheta}\in \cL_\eps \textrm{ for all }a\in A^+_{\ge T}}.}
Then Lemma \ref{InhomDani} implies that \eqlabel{eq:XiPartition}{\Xi\setminus\cY=\bigcup_{\eps>0}(\Xi_\eps\setminus \cY)\subseteq\bigcup_{\eps>0}\bigcup_{T\in\bN}\Xi_{\eps,T}.}
Combining \eqref{eq:DivAvgPartition} and \eqref{eq:XiPartition} we have
\eq{\Xi\setminus(\cY\cup\cW)=\left([-\tfrac{1}{2},\tfrac{1}{2}]^{d-1} \setminus\cW\right)\cap(\Xi\setminus\cY)\subseteq\bigcup_{K\subset X}\bigcup_{\del>0}\bigcup_{\eps>0}\bigcup_{T\in\bN} \left(\Xi_{\eps,T}\cap\cZ_{K,\del}\right).}
It follows that there exist $\del,\eps>0$, $T\in\bN$, and a compact set $K\subset X$ such that \eq{\dim_H(\Xi_{\eps,T}\cap \cZ_{K,\del})>(2-\lambda^6)\frac{d-1}{2}.} Note that $\cZ_{K,\del}=\displaystyle\bigcap_{M_0\in\bN}\displaystyle\bigcup_{m\ge M_0}\cZ_{K,\del,m}$. Thus, we can also find infinitely many $m$'s such that \eqlabel{eq:Hdimlowerbound}{\dim_{H}(\Xi_{\eps,T}\cap\cZ_{K,\del,m})>(2-\lambda^6)\frac{d-1}{2}.}

Let $\set{m_j}_{j=1}^{\infty}$ be a sequence of $m$'s satisfying \eqref{eq:Hdimlowerbound}. For any $j\in\bN$, let $S_j$ be a maximal $e^{-2m_j}$-separated set of $\Xi_{\eps,T}\cap\cZ_{K,\del,m_j}$. Then \eqref{eq:Hdimlowerbound} implies that $\# S_j\gg e^{(d-1)(2-\lambda^6) m_j}$. By definition of $\Xi_{\eps,T}$, for each $\bxi\in S_j$ there exists $\btheta(\bxi)\in\bR^{d-1}$ such that $a\widehat{x}_{\bxi,\btheta(\bxi)}\in\cL_\eps$ for any $a\in A^+_{\ge T}$. We define
$Y_j:=\set{\widehat{x}_{\bxi,\btheta(\bxi)}\in\widehat{X}:\bxi\in S_j}$ and for $j\in\bN$ let \eq{\nu_j:=\frac{1}{\# S_j}\displaystyle\sum_{\widehat{x}\in Y_j}\del_{\widehat{x}}=\frac{1}{\# S_j}\displaystyle\sum_{\bxi\in S_j}\del_{\widehat{x}_{\bxi,\btheta(\bxi)}}}
be the normalized counting measure on the set $Y_j\subset \widehat{X}$. We also denote
$$\overline{\nu_j}:=\frac{1}{\# S_j}\displaystyle\sum_{x\in \pi_*Y_j}\del_{x}=\frac{1}{\# S_j}\displaystyle\sum_{\bxi\in S_j}\del_{x_{\bxi}}$$
so that $\overline{\nu_j}:=\pi_*\nu_j$ for $j\in\bN$. By extracting a subsequence if necessary, there exists a probability measure $\mu$ on the one-point compactification $\widehat{X}\cup\set{\infty}$ such that
\eq{\mu_j:=\frac{1}{m_j^{d-1}}\displaystyle\sum_{\btau\in\set{t,\cdots,m_jt}^{d-1}}(a_{\btau})_*\nu_j\wstar\mu,}
and also define $\overline{\mu_j}:=\pi_*\mu_j$ for $j\in\bN$ and $\overline{\mu}:=\pi_*\mu$. Then clearly $\overline{\mu_j}$ weakly converges to $\overline{\mu}$.

We will prove the following properties of the measure $\mu$ we have constructed:
\begin{prop}\label{measureproperties}
Let $\mu$ be the probability measure on $\widehat{X}\cup\set{\infty}$ constructed as above, under the assumption $\dim_H \Xi>\frac{d-1}{2}$. The measure $\mu$ satisfies the following properties:
\begin{enumerate}
    \item $\mu$ is $a_{\btau}$-invariant for any $\btau\in (t\bZ)^{d-1}$,
    \item $\mu(\widehat{X})\ge \del$,
    \item $\Supp \mu\subseteq \cL_\eps\cup\set{\infty}$,
    \item $\displaystyle\sum_{i=1}^{d-1}h_{\overline{\mu}}(a_{t\be_i})>0$.
\end{enumerate}
\end{prop}

Indeed, it is straightforward to deduce Theorem \ref{InhomLittlewood} from Proposition \ref{measureproperties}:
\begin{proof}[Proof of Theorem \ref{InhomLittlewood} assuming Proposition \ref{measureproperties}]
Suppose $\dim_H \Xi>\frac{d-1}{2}$ for the sake of contradiction. Then the probability measure $\mu$ on $\widehat{X}\cup\set{\infty}$ satisfies the properties (1)-(4) of Proposition \ref{measureproperties}. Since $\mu(\widehat{X})\ge \del$, we may write $\mu=\eta \mu'+(1-\eta)\del_{\infty}$ for some $\eta\geq\del$ and a probability measure $\mu'$ on $\widehat{X}$, where $\del_\infty$ denotes the Dirac delta measure on $\set{\infty}$. Then $\mu'$ is $a_{\btau}$-invariant for any $\btau\in (t\bZ)^{d-1}$, $\Supp \mu'\subseteq\cL_\eps$, and $\sum_{i=1}^{d-1}h_{\pi_*\mu'}(a_{t\be_i})>0$. Let us define a probability measure $\mu^0$ on $\widehat{X}$ by
$$\mu^0:=\frac{1}{t^{d-1}}\int_{[0,t]^{d-1}}(a_{\btau})_*\mu'd\btau. $$
Then $\mu^0$ is $A$-invariant, $\Supp\mu^0\subseteq \cL_{e^{-t}\eps}$, and $\sum_{i=1}^{d-1}h_{\pi_*\mu^0}(a_{t\be_i})>0$.
Choosing an ergodic component of $\mu^0$, we may assume that $\mu^0$ is ergodic. Then $\mu^0$ is homogeneous by Theorem \ref{measureclassification}, however it contradicts Lemma \ref{Lesupp} since $\Supp\mu^0\subseteq \cL_{e^{-t}\eps}$. Therefore, we obtain
$\dim_H \Xi\leq\frac{d-1}{2}$.
\end{proof}

The rest of the section will be devoted to proving Proposition \ref{measureproperties}. Let us choose a small positive constant $\sig>0$ such that
$$(2-\lambda^4+20dl_\lambda \sig)(1-\sig)>2-\lambda^5,$$
and denote $\cD_j^*:=\set{t,\ldots, \lfloor \sig m_j \rfloor t}^{d-1}$. For each $j\in\bN$ we denote $\textbf{$m_j':=\lfloor(1-\sig) m_j\rfloor$}$.

\begin{lem}\label{eq:kappaalphabound}
    For any $j\in\bN$, $\btau\in\cD_j^*$, $\bxi\in\bR^{d-1}$, and $1\leq i\leq d-1$ we have
    $$\kappa_i(a_{\btau}u(\bxi)\Gamma)\leq e^{(d-1)\sig m_jt}, \qquad \alpha_i(a_{\btau}u(\bxi)\Gamma)\leq \useconE{101} e^{2(d-1)\sig m_jt}.$$
\end{lem}
\begin{proof}
    By definition of $\kappa_i$ we have $\kappa_i(a_{\btau}u(\bxi)\Gamma)\leq e^{\tau_1+\cdots+\tau_{d-1}}\leq e^{(d-1)\sig m_jt}$ for all $1\leq i\leq d-1$. It follows that $$\alpha_{i,\lambda}'(a_{\btau}u(\bxi)\Gamma)=\kappa_i(a_{\btau}u(\bxi)\Gamma)\operatorname{ht}_\lambda\big(\bar{a}_{\tau_1+\cdots+\tau_{d-1}}\bar{u}(\bxi)\Gamma\big)\leq e^{2(\tau_1+\cdots+\tau_{d-1})}\leq e^{2(d-1)\sig m_jt}$$
    for all $1\leq i\leq d-1$. The definition \eqref{alphatildedefinition} of $\widetilde{\alpha}_{i,\lambda}$ implies that
    $\widetilde{\alpha}_{i,\lambda}(x)\leq \max\set{\kappa_i(x)\useconE{101}, \alpha_{i,\lambda}'(x)}$ for any $x\in X'$. Hence we obtain
    $\widetilde{\alpha}_{i,\lambda}(a_{\btau}u(\bxi)\Gamma)\leq \useconE{101} e^{2(d-1)\sig m_jt}$ for all $1\leq i\leq d-1$.
\end{proof}

\begin{lem}\label{partitioncounting}
For any $j\in\bN$, $\btau\in \cD_j^*$, and $y\in X$ we have
$$\big((a_{\btau})_*\overline{\nu_j}\big)\big([y]_{\cQ^{(m_j')}}\big)\leq \useconE{109}e^{-(\lambda^5-\lambda^6)(d-1)m_jt}$$
for some constant \newconE{109}$\useconE{109}>0$.
\end{lem}
\begin{proof}
    Since $(a_{\btau})_*\overline{\nu_j}$ is the normalized counting measure on the set $\set{a_{\btau}u(\bxi)\Gamma: \bxi\in S_j}$, we have
    $$\big((a_{\btau})_*\overline{\nu_j}\big)\big([y]_{\cQ^{(m_j')}}\big)=\frac{1}{\# S_j}\#\set{\bxi\in S_j: a_{\btau}u(\bxi)\Gamma\in [y]_{\cQ^{(m_j')}}}.$$
    
    Let us define $F:[-\frac{1}{2},\frac{1}{2}]^{d-1}\to\bR$ by $F(\bxi)=\widetilde{\psi}_{m_j',2^{-l_\lambda},l_\lambda}(a_{\btau}u(\bxi)\Gamma;y)$. Observe that
    \eqlabel{eq:Fcounting}{\#\set{\bxi\in S_j: a_{\btau}u(\bxi)\Gamma\in [y]_{\cQ^{(m_j')}}}\leq \sum_{\bxi\in S_j}F(\bxi)}
    by Lemma \ref{eq:relationpsi}. By a similar argument to (2) of Lemma \ref{psisupp:traj} we obtain
    \eqlabel{eq:Fupperbound1}{\sum_{\bxi\in S_j}F(\bxi) \leq \useconE{102}^{d-1}2^{2(d-1)l_\lambda}\kappa(a_{\btau}\Gamma)^2e^{2(d-1)m_j't}\int_{[-\frac{1}{2},\frac{1}{2}]^{d-1}}\widetilde{\psi}_{m_j',2^{-l_\lambda+1},l_\lambda-1}(a_{\btau}u(\bxi)\Gamma;y)d\bxi.}

    Let $\set{\bxi_1,\ldots,\bxi_M}$ be a maximal $\frac{1}{2}e^{-2(1-\sig)m_j't}$-separated set so that $[-\frac{1}{2},\frac{1}{2}]^{d-1}$ is covered by $B(\bxi_k,\frac{1}{2}e^{-2(1-\sig)m_j't})$, and $M\ll e^{2(d-1)\sig m_j't}$. Then we have
    \eqlabel{eq:ksplit}{\begin{aligned}
        \int_{[-\frac{1}{2},\frac{1}{2}]^{d-1}}\widetilde{\psi}_{m_j',2^{-l_\lambda+1},l_\lambda-1}& (a_{\btau}u(\bxi)\Gamma;y)d\bxi\\&\leq \sum_{k=1}^{M}\int_{[-\frac{1}{2},\frac{1}{2}]^{d-1}}\widetilde{\psi}_{m_j',2^{-l_\lambda+1},l_\lambda-1}(u(\bxi)a_{\btau}u(\bxi_k)\Gamma;y)d\bxi.
    \end{aligned}}
    
    For each $1\leq k\leq M$ we apply Proposition \ref{nontrivbdd:traj} with $N=\frac{m_j'}{2^{l_\lambda-1}}$, $l=l_\lambda-1$, $\eta=2^{-l_\lambda+1}$, and $x=a_{\btau}u(\bxi_k)\Gamma$. Then we have
    \eqlabel{eq:Fupperbound2}{\begin{aligned}
        \int_{[-\frac{1}{2},\frac{1}{2}]^{d-1}}\widetilde{\psi}_{m_j',2^{-l_\lambda+1},l_\lambda-1}&(u(\bxi)a_{\btau}u(\bxi_k)\Gamma;y)d\bxi\\&\leq  \useconE{107}\kappa(a_{\btau}u(\bxi_k)\Gamma)^{8l_\lambda}e^{-\lambda^4(d-1)m_j't}\prod_{i=1}^{d-1}\widetilde{\alpha}_i(a_{\btau}u(\bxi_k)\Gamma)
    \end{aligned}}
    for some constant \newconE{107}$\useconE{107}>0$. Here we are using the fact from our choice of $l_\lambda$ that 
    $$(2^{l_\lambda-1}-1)N=(1-2^{-\lambda+1})m_j'\geq \lambda m_j' .$$  Combining \eqref{eq:Fcounting}, \eqref{eq:Fupperbound1}, \eqref{eq:ksplit}, \eqref{eq:Fupperbound2}, and the estimates from Lemma \ref{eq:kappaalphabound} altogether we obtain
    \eq{\begin{aligned}
        \#\set{\bxi\in S_j: a_{\btau}u(\bxi)\Gamma\in [y]_{\cQ^{(m_j')}}}&\leq \useconE{108}e^{(2-\lambda^4+20dl_\lambda \sig)(d-1)m_j't}\\&\leq \useconE{108}e^{(2-\lambda^4+20dl_\lambda \sig)(1-\sig)(d-1)m_jt}
    \end{aligned}}
    for some constant \newconE{108}$\useconE{108}>0$. Recall that $\sig$ is chosen so that 
    $$(2-\lambda^4+20dl_\lambda \sig)(1-\sig)>2-\lambda^5,$$
    and $\# S_j\gg e^{(d-1)(2-\lambda^6) m_j}$. It follows that
    $$\big((a_{\btau})_*\overline{\nu_j}\big)\big([y]_{\cQ^{(m_j')}}\big)\leq \frac{1}{\# S_j}\sum_{\bxi\in S_j}F(\bxi)\ll \frac{e^{(2-\lambda^5)(d-1)m_jt}}{e^{(2-\lambda^6)(d-1)m_jt}}=e^{-(\lambda^5-\lambda^6)(d-1)m_jt},$$
    where the implied constant depends only on $d$ and $\lambda$.    
\end{proof}

We are now ready to show Proposition \ref{measureproperties}.
\begin{proof}[Proof of Proposition \ref{measureproperties}]
(1) Since $(a_{\btau})_*\mu_j-\mu_j$ goes to zero measure as $j\to\infty$, $\mu$ is $a_{\btau}$-invariant for any $\btau\in(t\bZ)^{d-1}$.

(2) For any $\bxi\in S_j$ we have
$$\frac{1}{m_j^{d-1}}\displaystyle\sum_{\btau\in\{1,\ldots,m_j\}^{d-1}}\del_{a_{\btau}\widehat{x}_{\bxi,\btheta(\bxi)}}(K)\ge \del$$
 since $\widehat{x}_{\bxi,\btheta(\bxi)}\in \cZ_{K,\del,m_j}$. It follows that
$$\mu_j(K)=\frac{1}{\# S_j}\displaystyle\sum_{\bxi\in S_j}\frac{1}{m_j^{d-1}}\displaystyle\sum_{\btau\in\{1,\ldots,m_j\}^{d-1}}\del_{a_{\btau}\widehat{x}_{\bxi,\btheta(\bxi)}}(K)\ge \del$$
for any $j\in\bN$. Thus, $\mu(\widehat{X})\ge \del$.

(3) We show that $\mu(\widehat{X}\setminus \cL_\eps)=0$. Recall that  $a\widehat{x}_{\bxi,\btheta(\bxi)}\in \cL_\eps$ for any $\bxi\in S_j$ and $a\in A^+_{\ge T}$. It follows that
$$\frac{1}{m_j^{d-1}}\displaystyle\sum_{\btau\in\{1,\ldots,m_j\}^{d-1}}\del_{a_{\btau}\widehat{x}_{\bxi,\btheta(\bxi)}}(\widehat{X}\setminus \cL_\eps)\leq \frac{dTm_j}{m_j^2}=\frac{dT}{m_j}$$
for any $\bxi\in S_j$. It implies that
$$\mu_j(\widehat{X}\setminus \cL_\eps)=\frac{1}{\# S_j}\displaystyle\sum_{\bxi\in S_j}\frac{1}{m_j^{d-1}}\displaystyle\sum_{\btau\in\{1,\ldots,m_j\}^{d-1}}\del_{a_{\btau}\widehat{x}_{\bxi,\btheta(\bxi)}}(\widehat{X}\setminus \cL_\eps)\leq \frac{dT}{m_j}$$
for any $j\in\bN$, hence $\mu(\widehat{X}\setminus \cL_\eps)=0$.

(4) For $\btau\in \cD_j^*$ and $1\leq i\leq d-1$ we define
$$\mu_{j,\btau,i}:= \frac{1}{\lfloor(1-\sig) m_j\rfloor} \displaystyle\sum_{k=1}^{\lfloor(1-\sig) m_j\rfloor} (a_{\btau+kt\be_i})_*\nu_j.$$
Then for any $1\leq i\leq d-1$
\eqlabel{eq:mujdecomposition}{\sig^{d-1}\mu_j\leq \frac{1}{|\cD_j^*|}\displaystyle\sum_{\btau\in \cD_j^*} \mu_{j,\btau,i}.}

By Proposition \ref{partitioncounting} we have
\eqlabel{eq:atomdecay}{\big((a_{\btau})_*\overline{\nu_j}\big)\big([y]_{\cQ^{(m_j')}}\big)\leq \useconE{109}e^{-(\lambda^5-\lambda^6)(d-1)m_jt}.}
We refer the reader to \cite[Chapter 1 \& 2]{ELW} for the definitions and basic properties of the entropy. For all sufficiently large $j\in\bN$, \eqref{eq:atomdecay} impies that
\eqlabel{eq:StaticEntropyEstimate}{\begin{aligned}
    \displaystyle\sum_{i=1}^{d-1}H_{( a_{\btau})_*\overline{\nu_j}}\big(\cQ^{(m_j',i)}\big)&\geq  H_{( a_{\btau})_*\overline{\nu_j}}\left(\bigvee_{i=1}^{d-1}\cQ^{(m_j',i)}\right)\\&\geq (\lambda^5-\lambda^6)(d-1)m_j t-\log \useconE{109}\geq (\lambda^5-\lambda^6)m_jt.
\end{aligned}}

For $q\ge 1$ and large $j$, we write the Euclidean division of $m_j'-1$ by $q$
$$ m_j'-1=qm'+r, \textrm{ with }r\in\set{0,\ldots,q-1}.$$
By subadditivity of the entropy with respect to the partition,
\eq{H_{( a_{\btau})_*\overline{\nu_j}}\big(\cQ^{(m_j',i)}\big)\leq \sum_{k=1}^{m'}H_{(a_{\btau+(kq+r)t\be_i})_*\overline{\nu_j}}\big(\cQ^{(q,i)}\big)+2q\log |\cQ|}
for each $r\in \set{0,\ldots,q-1}$. Summing those inequalities for $r=0,\ldots,q-1$, and using the concavity of entropy with respect to measure, we get
\eq{\begin{aligned}
    qH_{( a_{\btau})_*\overline{\nu_j}}\big(\cQ^{(m_j',i)}\big)&\leq \sum_{k=0}^{m_j'-1}H_{(a_{\btau+kt\be_i})_*\overline{\nu_j}}\big(\cQ^{(q,i)}\big)+2q^2\log |\cQ|\\&\leq m_j'H_{\pi_*\mu_{j,\btau,i}}\big(\cQ^{(q,i)}\big)+2q^2\log |\cQ|.
\end{aligned}}
It follows from \eqref{eq:StaticEntropyEstimate} that for any $1\leq i\leq d-1$ and $\btau\in \cD_j^*$
\eq{\begin{aligned}
    \frac{1}{q}H_{\pi_*\mu_{j,\btau,i}}\big(\cQ^{(q,i)}\big)&\geq \frac{1}{m_j'}H_{(a_{\btau})_*\overline{\nu_j}}\big(\cQ^{(m_j',i)}\big)-\frac{2q\log|\cP|}{m_j'}\\&\geq (\lambda^5-\lambda^6)t-\frac{2q\log|\cP|}{m_j'}.
\end{aligned}}
Using again the concavity of entropy with respect to measure and \eqref{eq:mujdecomposition},
\eq{\begin{aligned}
    \sum_{i=1}^{d-1}\frac{1}{q}H_{\overline{\mu_j}}(\cQ^{(q,i)})&\geq \frac{\sig^{d-1}}{|\cD_j^*|}\sum_{i=1}^{d-1}\sum_{\btau\in\cD_j^*}\frac{1}{q}H_{\pi_*\mu_{j,\btau,i}}(\cQ^{(q,i)})+o(1)\\&\geq \sig^{d-1}\left((\lambda^5-\lambda^6)t-\frac{2q\log|\cP|}{m_j'}\right).
\end{aligned}}
In particular, by taking $j\to\infty$ and $q\to\infty$ we obtain the desired inequality
$$\sum_{i=1}^{d-1}h_{\overline{\mu}}(a_{t\be_i})\geq \sig^{d-1}(\lambda^5-\lambda^6)t>0.$$
\end{proof}

\def\cprime{$'$} \def\cprime{$'$} \def\cprime{$'$}
\providecommand{\bysame}{\leavevmode\hbox to3em{\hrulefill}\thinspace}
\providecommand{\MR}{\relax\ifhmode\unskip\space\fi MR }
\providecommand{\MRhref}[2]{%
  \href{http://www.ams.org/mathscinet-getitem?mr=#1}{#2}
}

\end{document}